\documentclass{amsart}
\usepackage[margin=1in]{geometry}
\usepackage{amsrefs}
\usepackage{amsmath}
\usepackage{amssymb}
\usepackage{amsthm}
\usepackage{amscd}
\usepackage{enumerate}

\usepackage{graphicx}
\usepackage{amsmath,amsthm,amsfonts,amscd,amssymb,comment,eucal,latexsym,mathrsfs}
\usepackage{stmaryrd}
\usepackage[all]{xy}

\usepackage{epsfig}

\usepackage[all]{xy}
\xyoption{poly}
\usepackage{fancyhdr}
\usepackage{wrapfig}
\usepackage{epsfig}

\theoremstyle{plain}
\newtheorem{thm}{Theorem}[section]
\newtheorem{prop}[thm]{Proposition}
\newtheorem{lem}[thm]{Lemma}

\newtheorem{conj}{Conjecture}

\theoremstyle{definition}
\newtheorem{defn}[thm]{Definition}
\theoremstyle{remark}

  \def\C{{\mathbb{C}}}   \def\F{{\mathbb{F}}}        \def\N{{\mathbb{N}}}  \def\P{{\mathbb{P}}}  \def\R{{\mathbb{R}}}        \def\Z{{\mathbb{Z}}}

            \def\cM{{\mathcal{M}}} \def\cN{{\mathcal{N}}}    \def\cR{{\mathcal{R}}}        

\newcommand\vre{\varepsilon}
\newcommand\varpesilon{\vre}
\newcommand\varespilon{\vre}

\newcommand\ev{\operatorname{ev}}

\newcommand\Haus{\operatorname{Haus}}

\newcommand\id{\operatorname{id}}

\renewcommand\Im{\operatorname{Im}}

\newcommand\orb{\operatorname{orb}}

\newcommand\Prob{\operatorname{Prob}}

\newcommand\rea{\operatorname{Re}}

\newcommand\Span{\operatorname{span}}

\newcommand\tr{\operatorname{tr}}
\newcommand\Tr{\operatorname{Tr}}

\newcommand{\ip}[1]{\langle #1 \rangle}

\newcounter{pcounter}
\begin{document}

\title{A random matrix approach to the Peterson-Thom conjecture}      
\author{Ben Hayes}\thanks{The author gratefully acknowledges support from  NSF Grants  DMS-1600802 , DMS-1827376, and DMS-2000105.}
\address{University of Virginia\\
          Charlottesville, VA 22904}
\email{brh5c@virginia.edu}
\date{\today}
\maketitle

\begin{abstract}
The Peterson-Thom conjecture asserts that any diffuse, amenable subalgebra of a free group factor is contained in a \emph{unique} maximal amenable subalgebra. This conjecture is motivated by related results in Popa's deformation/rigidity theory and Peterson-Thom's results on $L^{2}$-Betti numbers. We present an approach to this conjecture in terms of so-called \emph{strong convergence} of random matrices by formulating a conjecture which is a natural generalization of the Haagerup-Thorbj\o rnsen theorem whose validity would imply the Peterson-Thom conjecture. This random matrix conjecture is related to recent work of Collins-Guionnet-Parraud.

\end{abstract}


\section{Introduction}
Amenability is arguably the central concept in von Neumann algebra theory. Celebrated and fundamental work of Connes \cite{Connes} shows that amenable von Neumann algebras are precisely the hyperfinite ones and also gives several equivalent forms of amenability. Because of this famous work, amenable algebras are well understood and completely classified.

Given our understanding of amenable von Neumann algebras, it is natural to try to bootstrap our knowledge of arbitrary von Neumann algebras from our knowledge of the amenable ones. This naturally motivates the study of \emph{maximal amenable subalgebras} of a von Neumann algebra, those subalgebras which are amenable and are maximal with respect to inclusion among amenable subalgebras. In a landmark discovery \cite{PopaMaximalAmenable}, Popa showed that $L(\Z)$ is a maximal amenable subalgebra of $L(\F_{r})=L(\Z*\F_{r-1})$ for any $r\in \N$ (here and throughout the paper $\F_{r}$ denotes the free group on $r$ letters). This provided the first example of a maximal amenable subalgebra that was abelian (a phenomenon that was unexpected at the time), and it gave a negative answer to a related problem of Kadison stated during the Baton Rouge conference in 1967.
In fact, Popa's work establishes the more general result that $L(\Z)$ is \emph{maximal Gamma} in $L(\F_{r})$. The foundational insight of Popa was the usage of his \emph{asymptotic orthogonal property} to establish maximal amenability.
Many authors have used Popa's asymptotic orthogonal property to establish maximal amenability in various cases, see \cite{Ge96, Shen06, FangMaximalAmena, CFRW, Gao10, Brothier, Houdayerexotic, CyrilSAOP,Bleary}.

 Popa's deformation/rigidity theory also allows one to deduce strong ``malnormality" properties of free group factors. For example, primeness of all nonamenable subfactors \cite{PetersonDeriva}, as well as  the celebrated \emph{strong solidity} of free group factors \cite{OzPopaCartan}.
Much of the developments in deformation/rigidity theory go beyond free group factors and apply to   von Neumann algebras associated to  groups which have a combination of approximation properties and nontrivial cohomology \cite{Po01a, PopaL2Betti, PopaStrongRigidity, PopaStrongRigidtyII,PetersonDeriva, OzPopaII,  ChifanSinclair}, as well as crossed product algebras associated to actions of such groups. See \cite{PopaICM,Va06a,Va10a,Io12b,Io17c} for further results, including resolutions of long-standing open problems.
These developments parallel, and frequently require input from, the theory of $L^{2}$-Betti numbers for groups (developed in \cite{Atiyah}) as well as equivalence relations (developed in \cite{Gab2}). Given these connections, we should expect in general that results from the theory of $L^{2}$-Betti numbers will have natural analogues for von Neumann algebras.
In \cite{PetersonThom}, Peterson-Thom proved various indecomposability and malnormality results for groups with positive first $L^{2}$-Betti number. Based on their work, and  previous work of Ozawa-Popa, Peterson, and Jung \cite{OzPopaCartan, PetersonDeriva, JungSB}, they conjectured the following for von Neumann algebras.

\begin{conj}\label{conj:PT intro}
Fix $r>1.$ If $Q$ is a von Neumann subalgebra of $L(\F_{r})$ which is both diffuse and amenable, then there is a unique maximal amenable von Neumann subalgebra $P$ of $L(\F_{r})$ with $Q\subseteq P.$
\end{conj}
For the rest of the article, if $M$ is a von Neumann algebra, we will use the notation $N\leq M$ to mean that $N$ is a unital, von Neumann subalgebra of $M.$ Given $N\leq M$ with $M$ a finite von Neumann algebra and $N$ diffuse, we say that $N$ has the \emph{absorbing amenability property} (see \cite[Theorem 4.1]{CyrilAOP}) if whenever $Q\leq M$ is amenable and $Q\cap N$ is diffuse, we have $Q\subseteq N.$ An equivalent way of phrasing Conjecture \ref{conj:PT intro} is to say that if $r>1,$ then any maximal amenable $N\leq L(\F_{r})$ has the absorbing amenability property. For many examples of maximal amenable subalgebras of free group factors this has been verified \cite{WenAOP, 2AuthorsOneCup, ParShiWen}, and typically uses a generalization of Popa's asymptotic orthogonality property, called the \emph{strong asymptotic orthogonality property} implicitly defined in \cite[Theorem 3.1]{CyrilAOP}. Many exciting recent works \cite{BC2015, OzAbsor, HBAbsor} apply an alternative method using an analysis of states.

The fact that we can show the absorbing amenability property for many examples of maximal amenable subalgebras of free group factors is strong evidence for the Peterson-Thom conjecture, but as of yet the methods of proof for these examples have not led to a general approach to the problem. The goal of this paper is to provide such an approach through Voiculescu's free entropy dimension theory and random matrices. Free entropy dimension theory was initiated by Voiculescu in a series of papers \cite{FreeEntropyDimensionII,FreeEntropyDimensionIII}, and provides a powerful method to deduce indecomposability and malnormality results for free group factors (among other algebras). For example, Voiculescu used free entropy dimension, in combination with his previously established random matrix results \cite{VoicAsyFree,VoicAsyFreeStrong}, to give the first proof of absence of Cartan subalgebras in free group factors \cite{FreeEntropyDimensionIII}. Shortly after this work, \emph{primeness}  and \emph{thinness} of free group factors were first established by Ge, Ge-Popa using free entropy dimension theory \cite{GePrime,PopaGeThin}. See \cite{DykemaFreeEntropy, JungSB, HoudayerShlyakhtenko} for other applications.
Popa's deformation/rigidty theory and asymptotic orthogonality techniques apply to a wider range of algebras than free group factors/amalgamated free products, and do not require the Connes approximate embedding property for their applicability. However, some indecomposability results for free group factors shown using free entropy dimension theory cannot currently be approached by deformation/rigidity theory or the (strong) asymptotic orthogonal property \cite{DykemaFreeEntropy, PopaGeThin, Me6, FreePinsker}.

We now present our result which reduces the Peterson-Thom conjecture to a natural random matrix problem. It requires usage of the \emph{$1$-bounded entropy},  which was implicitly defined in \cite{JungSB} and explicitly in \cite{Me6}.  If $N\leq M$ are diffuse, tracial von Neumann algebras the \emph{$1$-bounded entropy of $N$ in the presence of $M$} (denoted $h(N:M)$) is some sort of measurement of ``how many" ways there are to ``simulate" $N$ by matrices which have an extension to a ``simulation" of $M$ by matrices (see Definition \ref{defn:S1B} for the precise definition). Throughout the paper, we say that a random self-adjoint matrix $X\in M_{k}(\C)_{s.a.}$ is \emph{GUE distributed} if \[\{X_{ii}:i=1,\cdots,k\}\cup \{\sqrt{2} \rea{X_{ij}}:1\leq i<j\leq k\}\cup \{\sqrt{2}\Im{X_{ij}}:1\leq i<j\leq k\}\]
is an independent family of Gaussian random variables each with mean $0$ and variance $\frac{1}{k}.$ For a natural number $k,$ we use $\C\ip{T_{1},T_{2},\cdots,T_{k}}$ for the $\C$-algebra of noncommutative polynomials in $k$-variables (i.e. the free $\C$-algebra in $k$ indeterminates).

\begin{thm}\label{T:main}
Fix an integer $r\geq 2.$ Consider the following statements.
\begin{enumerate}[(i)]
    \item \label{item:PT intro} If $Q\leq L(\F_{r})$ is diffuse and amenable, then there is a unique maximal amenable $P\leq L(\F_{r})$ with $Q\leq P.$
    \item If $Q\leq L(\F_{r})$ is nonamenable, then $h(Q:L(\F_{r}))>0.$ \label{item: Pinskers are amenable}
    \item For $k\in \N,$ let $X_{1}^{(k)},\cdots,X_{r}^{(k)}$,$Y_{1}^{(k)},\cdots, Y_{r}^{(k)}$ be random, self-adjoint $k\times k$ matrices which are independent and are each GUE distributed. Set $X^{(k)}\otimes 1_{M_{k}(\C)}=(X_{i}^{(k)}\otimes 1_{M_{k}(\C)})_{i=1}^{r}$, $1_{M_{k}(\C)}\otimes Y=(1_{M_{k}(\C)}\otimes Y_{i}^{(k)})_{i=1}^{r}.$ Let $s=(s_{1},\cdots,s_{r})$ be $r$ free-semicirculars each with mean zero and variance $1.$ Let $s\otimes 1_{C^{*}(s)}=(s_{i}\otimes 1_{C^{*}(s)})_{i=1}^{r}$, $1_{C^{*}(s)}\otimes s=(1_{C^{*}(s)}\otimes s_{i})_{i=1}^{r}\in (C^{*}(s)\otimes_{\min{}} C^{*}(s))^{r}.$
    Then with high probability the law of $(X^{(k)}\otimes 1_{C^{*}(s)},1_{C^{*}(s)}\otimes Y^{(k)})$ tends (as $k\to\infty$) to the law of $ (s\otimes 1_{C^{*}(s)},1_{C^{*}(s)}\otimes s)$ strongly. Namely, for every polynomial $P\in \C\ip{T_{1},\cdots,T_{r},T_{r+1},\cdots,T_{2r}}$ we have \begin{equation}\label{E:haus conv intro}\|P(X^{(k)}\otimes 1_{M_{k}(\C)},1_{M_{k}(\C)}\otimes Y^{(k)})\|_{\infty}\to_{k\to\infty}\|P(s\otimes 1_{C^{*}(s)},1_{C^{*}(s)}\otimes s)\|_{C^{*}(s)\otimes_{\min{}}C^{*}(s)}
    \end{equation}
    in probability.
    \label{item:tensor Haus}
    \end{enumerate}
Then (\ref{item:tensor Haus}) implies (\ref{item: Pinskers are amenable}) implies (\ref{item:PT intro}).
\end{thm}
As we remark explicitly in Sections \ref{S:main theorems}, \ref{sec: Jung} the GUE ensemble can be replaced by the Haar unitary ensemble (with the limit distribution being free Haar unitaries) and many of the other random matrix ensembles from random matrix theory (this is implied by the more general Theorem \ref{T:more general main theorem intro}). We state the results for the GUE ensemble mostly for convenience.
Strictly speaking, so-called \emph{strong convergence} is the conjunction of (\ref{E:haus conv intro}) with
\begin{equation}\label{eqn: conv in law intro}
\frac{1}{k^{2}}\Tr\otimes \Tr(P(X^{(k)}\otimes 1_{M_{k}(\C)},1_{M_{k}(\C)}\otimes Y^{(k)}))\to_{k\to\infty}\tau\otimes \tau(P(s\otimes 1_{C^{*}(s)},1_{C^{*}(s)}\otimes s))
\end{equation}
for every noncommutative polynomial $P\in\C\ip{T_{1},\cdots,T_{r},T_{r+1},\cdots,T_{2r}}$, where $\tau$ is the underlying tracial state on $C^{*}(s_{1},\cdots,s_{r})$. However, the fact that (\ref{eqn: conv in law intro}) holds for every $P\in\C\ip{T_{1},\cdots,T_{r},T_{r+1},\cdots,T_{2r}}$ is already a consequence of Voiculescu's asymptotic freeness theorem \cite{VoicAsyFree}.
The concept of strong convergence arises from groundbreaking work of Haagerup-Thorbj\o rnsen \cite{HTExt} who showed that the  law of an $r$-tuple of independent, $k\times k$ GUE distributed matrices converges strongly to the law of an $r$-tuple of freely independent semicircular variables which each have mean zero and variance $1.$ This work opened up an array of powerful tools which have been used in combination with delicate analytic and combinatorial arguments to establish strong convergence for many other ensembles, and also for a mixture of random and deterministic ensembles \cite{Male, MaleCollins, CollinsBordenave}. In particular, recent work  \cite{PisierSubExp, CGPStrongTen} shows that if $m_{k}$ is a sequence of positive integers with $|m_{k}|\leq Ck^{1/3}$ for some constant $C\geq 0,$ and if $X^{(k)}$ is an $r$-tuple of independent, $k\times k$ GUE distributed matrices, and $Y^{(m_{k})}$ is an $r$-tuple of independent, $m_{k}\times m_{k}$ GUE distributed matrices chosen independent of $X^{(k)},$ then (in the notation of the above theorem) the law of $(X^{(k)}\otimes 1_{M_{m_{k}}(\C)},1_{M_{k}(\C)}\otimes Y^{(m_{k})})$ converges strongly to the law of $(x\otimes 1_{C^{*}(x)},1_{C^{*}(x)}\otimes x).$ This was later improved to $|m_{k}|\leq C\frac{k}{(\log k)^{3}}$ \cite[Proposition 9.3]{MatrixConc}.  While this work does not quite resolve our conjecture (we need $m_{k}=k$), it nevertheless lends positive evidence to the validity of our approach.

Our work actually establishes something slightly more general than the above. To state this more general result requires as input the notion of exponential concentration of measure, a well-established tool in probability and geometric functional analysis (see Definition \ref{defn:ECM} for the precise definition). It also uses the highly general noncommutative functional calculus of Jekel initiated in \cite{JekelConvexPot, JekelEAP, FreePinsker}. We recall the precise construction in Section \ref{S:nc func calc}, but for now the reader should simply know that for $r\in \N$ and $R\in [0,\infty)$ there is a space $\mathcal{F}_{R,r,\infty}$ of ``noncommutative functions" defined on the $R$-ball of any von Neumann algebra which are uniformly $L^{2}$-continuous in an appropriate sense and have the property that if $(M,\tau)$ is any tracial von Neumann algebra, and $x\in M^{r}$ has $\|x_{j}\|\leq R,j=1,\cdots,r$ then given any $y\in W^{*}(x),$ there is an $f\in \mathcal{F}_{R,r,\infty}$ with $f(x)=y.$  For  natural numbers $k,r,$ we define a pseudometric $d^{\orb{}}$ on $M_{k}(\C)^{r}$ as follows. For $A\in M_{k}(\C),$ we set $\|A\|_{2}=\frac{1}{k}\Tr(A^{*}A).$ We then define
\[d^{\orb{}}(A,B)=\inf_{U\in \mathcal{U}(k)}\left(\sum_{j=1}^{r}\|UA_{j}U^{*}-B_{j}\|_{2}^{2}\right)^{1/2}.\]
We also use the notational convention that if $C\in M_{k}(\C)$, $D=(D_{1},D_{2},\cdots,D_{r})\in M_{k}(\C)^{r}$ then $CD=(CD_{1},CD_{2},\cdots,CD_{r})\in M_{k}(\C)^{r}$ and similarly for $DC.$ For $A\in M_{K}(\C)$, we let $A^{t}$ be its transpose.
The following is our more general random matrix result.

\begin{thm}\label{T:more general main theorem intro}
Let $X^{(k)}=(X_{j}^{(k)})_{j=1}^{l}$ be a tuple of $n(k)\times n(k)$-random matrices. Suppose that
\begin{itemize}
    \item There is an $(R_{j})_{j=1}^{l}\in [0,\infty)^{l}$ so that for all $j$
    \[\lim_{k\to\infty} \mathbb P(\|X_{j}^{(k)}\|_{\infty}\leq R_{j})=1.\]
    \item The law of $X^{(k)}$ converges in probability to the law of a tuple $x=(x_{j})_{j=1}^{l}$ in a tracial von Neumann algebra $(M,\tau)$ with  $M=W^{*}(x_{1},\cdots,x_{l}).$
    \item The probability distribution of $X^{(k)}$ exhibits exponential concentration of measure at scale $n(k)^{2}$ as $k\to\infty.$
\end{itemize}

\begin{enumerate}[(i)]
    \item \label{item:microstates collapse intro} Suppose that $Q\leq M$ is finitely generated, diffuse, and $h(Q:M)\leq 0.$ Suppose $y\in Q^{r}$ with $Q=W^{*}(y)$ and write $y=f(x)$ for some $f\in\mathcal{F}_{R,r,\infty}$. Then there exists $A^{(k)}\in M_{n(k)}(\C)^{l}$ such that $A^{(k)}$ converges to $x$ in law and so that
    \[d^{\orb}(f(X^{(k)}),f(A^{(k)}))\to 0,\]
    in probability. Namely, for every $\varepsilon>0$
    \[\P(d^{\orb}(f(X^{(k)}),f(A^{(k)}))<\varepsilon)\to_{k\to\infty}1.\]
    \item \label{item:Pinskers are amenable intro} Assume  that $C^{*}(x)$ is locally reflexive.
    Let $Y^{(k)}=(Y_{j}^{(k)})_{j=1}^{l}$ be an independent copy of $(X_{j}^{(k)})_{j=1}^{l}.$ If $(X^{(k)}\otimes 1_{M_{n(k)}}(\C)),1_{M_{n(k)}(\C)}\otimes (Y^{(k)})^{t})$ converges strongly in probability to $(x\otimes 1_{C^{*}(x)^{op}},1_{C^{*}(x)}\otimes x^{op}),$ then any diffuse $Q\leq M$ with $h(Q:M)\leq 0$ is necessarily amenable. In particular, given any diffuse, amenable $Q\leq M$ there is a unique maximal amenable $P\leq M$ with $Q\subseteq P.$
\end{enumerate}

\end{thm}

Let us comment a bit on the intuition for (\ref{item:microstates collapse intro}). The assumption that the law of $X^{(k)}$ converges in probability to the law of $x$ means that the randomly chosen matrix $X^{(k)}$ ``simulates" $x$ with high probability. In terminology introduced in Section \ref{S:laws}, we say that $X^{(k)}$ are \emph{microstates} for $x.$
The general properties of Jekel's noncommutative functional calculus then guarantee that the random matrix $f(X^{(k)})$ are also microstates for $f(x)=y.$  The conclusion of (\ref{item:microstates collapse intro}) then asserts that the random microstates for $y$ produced by the random matrix $f(X^{(k)})$ are all approximately unitarily equivalent to each other.
If one passes to the ultraproduct framework, then the picture becomes much clearer. These randomly chosen microstates then turn into (random) honest embeddings into a ultraproduct of matrices, and the conclusion of (\ref{item:microstates collapse intro}) is then the assertion that, with high probability, these different embeddings are all unitarily conjugate when restricted to $Q.$

Viewed through this lens,
part (\ref{item:Pinskers are amenable intro}) is connected with Jung's theorem \cite{JungConj} that a tracial von Neumann algebra which satisfies Connes approximate embeddability property is amenable if and only if any two embeddings into an ultraproduct of matrices are unitarily conjugate (see \cite{ScottSri2019ultraproduct} for a recent generalization of this fact to conjugation by unital, completely positive maps). In fact,  Jung's argument shows the following more general fact: if $(M,\tau)$ is a tracial von Neumann algebra which embeds into an ultraproduct of matrices then given any nonamenable $N\leq M$ there are two embeddings of $M$ into an ultraproduct of matrices which are not unitarily conjugate when restricted to $N.$ A consequence of (\ref{item: Pinskers are amenable}) is that, under the assumption of strong convergence, given a nonamenable $N\leq M$ a \emph{randomly chosen} pair of embeddings of $M$ into an ultraproduct of matrices are not unitarily conjugate when restricted to $N$. We refer the reader to Section \ref{sec: Jung} for a more precise discussion of parts (\ref{item:microstates collapse intro}),(\ref{item:Pinskers are amenable intro}) in an ultraproduct framework.

We close with a discussion of organization of the paper. Section \ref{S:background} is a discussion of background for the paper. In Section \ref{S:notation} we state our conventions and notation from von Neumann algebra and operator space theory. In Section \ref{S:laws} we recall the notion of noncommutative laws, and their use in defining $1$-bounded entropy. Here we also recall the definitions of the weak$^{*}$ and strong topologies on the space of laws. In Section \ref{S:measures}, we discuss (sequences of) measures on microstates spaces and the two important conditions on them we will use: being asymptotically supported on microstates spaces, and exponential concentration. Section \ref{S:nc func calc} describes Jekel's noncommutative functional calculus, as well as the modification we will need for the non-self-adjoint case. Strictly speaking, the usage of this general functional calculus is not necessary for the proofs of the main results and earlier versions of this paper did not use it. However, its usage drastically simplifies both the conception and the deduction of the main results and so we think its inclusion is worthwhile. Section \ref{S:main theorems} contains the proofs of Theorems \ref{T:main},\ref{T:more general main theorem intro}. Specifically, in Section \ref{subsec:microstates collapse} we prove Theorem \ref{T:more general main theorem intro} (\ref{item:microstates collapse intro}) which states that under an assumption of exponential concentration there is a ``collapse" of the microstates space where the vast majority of the measure lives near a single unitary conjugation orbit. The methods of proof here are similar to those in \cite{FreePinsker}. Section \ref{sub sec: proof of Pinsker are amenable intro} contains a proof of Theorem \ref{T:more general main theorem intro} (\ref{item:Pinskers are amenable intro}), and it is here that both strong convergence and local reflexivity play a crucial role. In Section \ref{subsec: proof of main theorem} we deduce Theorem \ref{T:main} from Theorem \ref{T:more general main theorem intro}. In Section \ref{sec: Jung} we explain how the results are related to Jung's theorem, and give reformulations of the main results in an ultraproduct framework. In Section \ref{sec: Jung} we also introduce several conjectures related to the Peterson-Thom conjecture and Theorem \ref{T:main}, and we explicitly explore their relative strength.  Finally, we close in Section \ref{s:close} with a few comments on the approach. In particular, we discuss the discontinuity in the strong topology of taking tensors, and how exactness of free group factors may allow one to follow previous approaches to proving strong convergence in probability.

\textbf{Acknowledgements}
 The initial stages of this work were carried out at the Hausdorff Research Institute for Mathematics during the 2016 trimester program ``Von Neumann Algebras." I thank the Hausdorff institute for their hospitality. Conversations during the ``Quantitative Linear Algebra" program at the Institute of Pure and Applied Mathematics at the University of California, Los Angeles were also insightful. I thank IPAM for its hospitality. I would like to thank Roy Araiza, Benoit Collins, Yoann Dabrowski, David Jekel, and Thomas Sinclair for inspirational conversations related to this work. I thank the anonymous referee for their numerous comments, which greatly improved the paper.

\section{Background}\label{S:background}

\subsection{General convention and notation}\label{S:notation}
For $k\in \N,$ we let $M_{k}(\C)$ be the space of $k\times k$ matrices over $\C,$ and $M_{k}(\C)_{s.a.}$ be the space of $k\times k$ \emph{self-adjoint} matrices over $\C.$ We also use $\mathcal{U}(k)$ for the unitaries in $M_{k}(\C).$ We define
$\tr\colon M_{k}(\C)\to \C$
by
\[\tr(A)=\frac{1}{k}\sum_{j=1}^{k}A_{jj}.\]
We define a Hilbert space inner product on $M_{k}(\C)$ by $\ip{A,B}=\tr(B^{*}A),$ and we let $\|\cdot\|_{2}$ be the norm induced by this inner product. We use $S^{2}(n,\tr)$ for $M_{n}(\C)$ equipped with this Hilbertian structure. For an index set $J,$ a finite $F\subseteq J,$ and $A\in M_{k}(\C)^{J}$ we set
\[\|A\|_{2,F}=\left(\sum_{j\in F}\|A_{j}\|_{2}^{2}\right)^{1/2}.\]
If $J$ itself is finite, and $F=J$ we will often use $\|\cdot\|_{2}$ instead of $\|\cdot\|_{J}$.
The pair $(M_{k}(\C),\tr)$ is an important example of a more general concept.
\begin{defn}
 A
\emph{tracial von Neumann} algebra is a pair $(M,\tau)$ where $M$ is a von Neumann algebra, and $\tau\colon M\to \C$ is a faithful, normal, tracial state.
\end{defn}
For a von Neumann algebra $M$ we use $M_{*}$ for the normal linear functionals $M\to \C.$ We call $M_{*}$ the \emph{predual of $M$}, it is a Banach space under the operator norm.
For von Neumann algebras, we will adopt similar conventions as in the case of matrices. For example, $\mathcal{U}(M),M_{s.a.}$ will refer to the unitaries and self-adjoints in $M.$ For reasons that will become clear shortly, for $x\in M$ we use $\|x\|_{\infty}$ for the operator norm of $x.$  Given a von Neumann algebra $M$ we shall use $N\leq M$ to mean that $N$ is a unital von Neumann subalgebra of $M.$ If $M$ is a von Neumann algebra, $J$ an index set and $x=(x_{j})_{j\in J}\in M^{J},$ and $y\in M,$ we will use $yx,xy$ for $(yx_{j})_{j\in J}$, $(x_{j}y)_{j\in J}\in M^{J},$ respectively. Similarly, if $M_{1},M_{2}$ are von Neumann algebras and $\pi\colon M_{1}\to M_{2}$ is a $*$-homomorphism, then for an index set $J$ and $x=(x_{j})_{j\in J}\in M_{1}^{J}$ we will use $\pi(x)$ for $(\pi(x_{j}))_{j\in j}\in M_{2}^{J}.$

While a von Neumann algebra is assumed to come with an ambient embedding into bounded operators on a Hilbert space $\mathcal{H}$, one significant advantage of a \emph{tracial} von Neumann algebra is that there is a natural representation of the algebra we can build from the trace. Given a tracial von Neumann algebra $(M,\tau),$ we define an inner product on $M$ by
\[\ip{a,b}=\tau(b^{*}a).\]
We use $\|\cdot\|_{2}$ for the norm induced by this inner product, and we let $L^{2}(M,\tau)$ be the Hilbert space which is the completion under this inner product. It is direct to show (see \cite[Section 7.1.1]{anantharaman-popa}) that for all $x,y\in M$
\[\|xy\|_{2}\leq \|x\|_{\infty}\|y\|_{2},\,\,\,\,\,\,\|xy\|_{2}\leq \|y\|_{\infty}\|x\|_{2}.\]
Thus the operators $y\mapsto xy,$ $y\mapsto yx$ extend continuously to bounded operators on $L^{2}(M,\tau)$. Moreover, if $M$ is given as a von Neumann algebra of operators on  a Hilbert space $\mathcal{H}$, then the above inequality proves that the norms of the operators $y\mapsto xy,$ $y\mapsto yx$ acting on $L^{2}(M,\tau)$ are equal to the norm of $x$ as an element of $B(\mathcal{H}).$ For $\xi\in L^{2}(M,\tau),$ we use $x\xi,$ $\xi x$ for the image of $\xi$ under these operators. Given a tracial von Neumann algebra $(M,\tau)$ the above allows us to view it as a von Neumann algebra of operators on $L^{2}(M,\tau)$ by left multiplication. We will essentially always view a tracial von Neumann algebra in this manner and ignore whatever other Hilbert space it arises from.

We will need to use tensor products at various points in the paper. For vector spaces $V,W$ we use $V\otimes_{\textnormal{alg}}W$ for their \emph{algebraic} (i.e. not completed)\emph{ tensor product}.
If $\mathcal{H}_{1},$ $\mathcal{H}_{2}$ are Hilbert spaces, we let $\mathcal{H}_{1}\otimes \mathcal{H}_{2}$ denote their Hilbert space tensor product. Given $T\in B(\mathcal{H}_{1})$, $S\in B(\mathcal{H}_{2})$, we let $T\otimes S$ be the unique operator in $B(\mathcal{H}_{1}\otimes \mathcal{H}_{2})$ given by
\[(T\otimes S)(\xi\otimes \eta)=T\xi\otimes S\eta\mbox{ for $\xi\in \mathcal{H}_{1},$ $\eta\in \mathcal{H}_{2}$.}\]
For von Neumann algebras $M_{j}\subseteq B(\mathcal{H}_{j}),j=1,2$ we let
\[M_{1}\overline{\otimes}M_{2}=\overline{\Span\{T\otimes S:T\in M_{1},S\in M_{2}\}}^{SOT}.\]

At various important points in the paper, we will need to use approximation properties in terms of \emph{completely bounded/completely positive maps}. These are the appropriate morphisms for what are now called operator spaces/operator systems.
\begin{defn}
A \emph{(concrete) operator space} is a closed, linear subspace of $B(\mathcal{H})$ for some Hilbert space $\mathcal{H}.$ A \emph{(concrete) operator system} is a closed, linear subspace of $B(\mathcal{H})$ which is closed under adjoints and contains the identity operator.
\end{defn}

If $E\subseteq B(\mathcal{H})$ is an operator space, then we may view $M_{n}(E)\subseteq B(\mathcal{H}^{\oplus n})$ in a natural way. So if we are given $A\in M_{n}(E)$ then the embedding $M_{n}(E)\subseteq B(\mathcal{H}^{\oplus n})$ allows us to make sense of $\|A\|_{M_{n}(E)}.$
Properly speaking, an operator space is really a Banach space $E$ together with the data of these norms on $M_{n}(E)$ and one can give an axiomatic description for such norms to arise from an embedding into $B(\mathcal{H})$ (see \cite[Theorem 2.3.5]{EffrosRuan}). We will stick to concrete operator spaces (i.e. given as a subspace of $B(\mathcal{H})$) for the purposes of this paper.
Given operator spaces $E,F$ and a bounded, linear map $T\colon E\to F$ we define for $n\in \N,$ $T\otimes \id_{M_{n}(\C)}\colon M_{n}(E)\to M_{n}(F)$ by $[(T\otimes \id_{M_{n}(\C)})(A)]_{ij}=T(A_{ij})$ for $A\in M_{n}(E).$ We say that $T$ is \emph{completely bounded} if
\[\sup_{n}\|T\otimes \id_{M_{n}(\C)}\|<\infty,\]
the norm in question being the operator norm. If $T$ is completely bounded, we set
\[\|T\|_{cb}=\sup_{n}\|T\otimes \id_{M_{n}(\C)}\|.\]
We say that $T$ is \emph{completely contractive} if $\|T\|_{cb}\leq 1.$
We let $CB(E,F)$ be the completely bounded maps $E\to F$ and will often use $CB(E)$ instead of $CB(E,E).$ If $E_{1},E_{2}$ are operator spaces and $E_{j}\subseteq B(\mathcal{H}_{j}),j=1,2,$ then we let $E_{1}\otimes_{\min{}}E_{2}$ be the operator space given by
\[\overline{\Span\{A\otimes B:A\in E_{1},B\in E_{2}\}}^{\|\cdot\|_{\infty}}\subseteq B(\mathcal{H}_{1}\otimes \mathcal{H}_{2}).\]
If $E_{j},F_{j},j=1,2$ are operator spaces and $T_{j}\colon E_{j}\to F_{j},j=1,2$ are completely bounded, then the map $T_{1}\otimes T_{2}\colon E_{1}\otimes_{\textnormal{alg}}E_{2}\to F_{1}\otimes_{\textnormal{alg}}F_{2}$ extends continuously to a completely bounded map $E_{1}\otimes_{\min{}}E_{2}\to F_{1}\otimes_{\min{}}F_{2}$ which we still denote $T_{1}\otimes T_{2}.$ We also have
\[\|T_{1}\otimes T_{2}\|_{cb}=\|T_{1}\|_{cb}\|T_{2}\|_{cb}.\]
This is decidedly not true if we consider \emph{bounded} maps instead of \emph{completely bounded} maps, and indeed arguably the main motivation for completely bounded maps and operator spaces is to provide a context in which one can extend bounded maps to tensor products.

For operator systems there is a natural order structure at play. Suppose $E\subseteq B(\mathcal{H})$ is an operator system, and consider  the embeddings $M_{n}(E)\subseteq B(\mathcal{H}^{\oplus n}).$ We can then define the positive elements in $M_{n}(E)$ to be those which are positive as operators on $B(\mathcal{H}^{\oplus n}).$ Since $E$ is an operator system, it has an abundance of positive elements, e.g. every element of $E$ is a linear combination of $4$ positive elements. Given operator systems $E,F$ a map $T\colon E\to F$ is \emph{positive} if $T(x)\geq 0$ for all $x\in E$ with $x\geq 0.$ It is \emph{completely positive} if $T\otimes 1_{M_{n}(\C)}$ is positive for all $n.$  We say $T$ is \emph{unital} if $T(1)=1.$ We use $CP(E,F)$ and $UCP(E,F)$ for the completely positive and unital, completely positive maps $E\to F$ respectively. It is a fact that for $T\in CP(E,F)$ we have $\|T\|_{cb}=\|T(1)\|$ (see \cite[Lemma 5.1.1]{EffrosRuan}).  As in the operator space case, if $E_{j},F_{j},j=1,2$ are operator systems and $T_{j}\colon E_{j}\to F_{j},j=1,2$ are completely positive, then so is $T_{1}\otimes T_{2}\colon E_{1}\otimes_{\min{}}E_{2}\to F_{1}\otimes_{\min{}}F_{2}.$ As in the operator space case, the analogous statement is \emph{false} for \emph{positive maps}.

Since any $C^{*}$-algebra can be  embedded in bounded operators on a Hilbert space, we may view any closed subspace of a $C^{*}$-algebra as an operator space. Similarly, we may view any closed subspace which is closed under adjoints and contains the unit as an operator system.


If $M_{j},j=1,2,$ $N_{j},j=1,2$ are von Neumann algebras and $T_{j}\colon M_{j}\to N_{j},j=1,2$ are normal, completely bounded maps then $T_{1}\otimes T_{2}$ has a unique, normal extension to a map $M_{1}\overline{\otimes}M_{2}\to N_{1}\overline{\otimes}N_{2}$ which we still denote $T_{1}\otimes T_{2}.$ Moreover,
\[\|T_{1}\otimes T_{2}\|_{CB(M_{1}\overline{\otimes}M_{2},N_{1}\overline{\otimes}N_{2})}=\|T_{1}\|_{cb}\|T_{2}\|_{cb}.\]
Further, if each $T_{j}$ is completely positive, then so is $T_{1}\otimes T_{2}.$

\subsection{Laws, Microstates, and $1$-Bounded Entropy}\label{S:laws}
Given an index set $J,$ we let
$\C^{*}\ip{(T_{j})_{j\in J}}$
be the $*$-algebra of noncommutative $*$-polynomials in the abstract variables $(T_{j})_{j\in J}.$ We may think of $\C^{*}\ip{(T_{j})_{j\in J}}$ as the (algebraically) free $*$-algebra indexed by $J.$ If $J=\{1,\cdots,n\},$ we typically use
$\C^{*}\ip{T_{1},\cdots,T_{n}}$
for $\C^{*}\ip{(T_{j})_{j=1}^{n}}.$
If we are given a $*$-algebra $A,$ and a tuple $x\in A^{J},$ then by algebraic freeness there is a unique $*$-homomorphism $\ev_{x}\colon \C^ {*}\ip{(T_{j})_{j\in J}}\to A$ such that $\ev_{x}(T_{j})=x_{j}.$ For $P\in \C^ {*}\ip{(T_{j})_{j\in J}},$ we denote $\ev_{x}(P)$ by $P((x_{j})_{j\in J}).$ Again, if $J=\{1,\cdots,n\},$ we usually use $P(T_{1},\cdots,T_{n}).$

\begin{defn}\label{defn:laws}
Let $J$ be an index set. A linear functional $\ell\colon \C^{*}\ip{(T_{j})_{j\in J}}\to \C$ is called a \emph{tracial law} if there is a $R\colon J\to [0,\infty)$ so that
\begin{itemize}
    \item $\ell(P^{*}P)\geq 0$ for all $P\in\C^{*}\ip{(T_{j})_{j\in J}},$
    \item $\ell(1)=1,$
    \item $\ell(PQ)=\ell(QP)$ for all $P,Q\in \C^{*}\ip{(T_{j})_{j\in J}}.$
    \item for all $n\in \N,$ all $j_{1},j_{2},\cdots,j_{n}\in J$ and all $\sigma_{1},\cdots,\sigma_{n}\in \{1,*\},$
    \[|\ell(T_{j_{1}}^{\sigma_{1}}T_{j_{2}}^{\sigma_{2}}\cdots T_{j_{n}}^{\sigma_{n}})|\leq R_{j_{1}}R_{j_{2}}R_{j_{3}}\cdots R_{j_{n}}.\]
\end{itemize}
We let $\Sigma_{J}$ be the space of tracial laws indexed by $J.$ If $J=\{1,\cdots,n\},$ we typically use $\Sigma_{n}$ instead of $\Sigma_{J}.$ Given a function $R\colon J\to [0,\infty),$ we let $\Sigma_{R,J}$ be the set of all laws $\ell$ satisfying the fourth item above for this specific $R.$ If $R\in [0,\infty)$ we will frequently use $\Sigma_{R,J}$ for $\Sigma_{\widehat{R},J}$ where $\widehat{R}\colon J\to [0,\infty)$ is the function which is constantly $R.$ As above, if $J=\{1,\cdots,n\}$ we will frequently use $\Sigma_{R,n}$ (in both the case that $R$ is a function and the case that it is a constant).
\end{defn}

The above may be regarded as an abstract definition of a law. If we are concretely given a tracial von Neumann algebra $(M,\tau)$ and a tuple $x\in M^{J}$ for some indexing set $J,$ we define the \emph{law of $x$} to be the linear functional
\[\ell_{x}\colon \C^{*}\ip{(T_{j})_{j\in J}}\to \C\]
given by $\ell_{x}(P)=\tau(P((x_{j})_{j\in J})).$ We always equip $M_{k}(\C)$ with its unique tracial state $\tr$ given by
\[\tr(A)=\frac{1}{k}\sum_{j=1}^{k}A_{jj}.\]
So if $A\in M_{k}(\C)^{J},$ we have a notion of its law $\ell_{A}.$

In fact, every \emph{abstract} law arises as a \emph{concrete} law for some tuple in a tracial von Neumann algebra. This follows from the GNS (Gelfand-Naimark-Segal) construction, which we sketch here. Let $J$ be an index set and $\ell\in \Sigma_{J}.$ Define a semi-inner product on $\C^{*}\ip{(T_{j})_{j\in J}}$ by
\[\ip{P,Q}=\ell(Q^{*}P).\]
For $P\in \C^{*}\ip{(T_{j})_{j\in J}}$ we set
\[\|P\|_{L^{2}(\ell)}=\ell(P^{*}P)^{1/2},\]
and we define
\[W=\{P\in \C^{*}\ip{(T_{j})_{j\in J}}:\|P\|_{L^{2}(\ell)}=0\},\mbox{ and }
V=\C^{*}\ip{(T_{j})_{j\in J}}/W.\]
The semi-inner product $\C^{*}\ip{(T_{j})_{j\in J}}$ descends to a genuine inner product on $V,$ and we let $L^{2}(\ell)$ be the Hilbert space which is the completion of $V$ under the norm coming from this inner product. From the fourth bullet point in Definition \ref{defn:laws}, one can deduce that there is a $C\colon \C^{*}\ip{(T_{j})_{j\in J}}\to [0,\infty]$ so that
\[\|PQ\|_{L^{2}(\ell)}\leq C(P)\|Q\|_{L^{2}(\ell)}\]
for all $P,Q\in \C^{*}\ip{(T_{j})_{j\in J}}.$
So we may proceed as in  the tracial von Neumann algebra case to deduce that there is a well-defined $*$-homomorphism
$\pi_{\ell}\colon \C^{*}\ip{(T_{j})_{j\in J}}\to B(L^{2}(\ell))$
satisfying
\begin{equation}\label{eqn:GNS rep}
  \pi_{\ell}(P)(Q+W)=PQ+W
\end{equation}
for all $P,Q\in \C^{*}\ip{(T_{j})_{j\in J}}.$ Set \begin{equation}\label{eqn:vNa of a law}
    W^{*}(\ell)=\overline{\pi_{\ell}(\C^{*}\ip{(T_{j})_{j\in J}})}^{SOT},
\end{equation} and let $x=(\pi_{\ell}(T_{j}))_{j\in J}\in M^{J}.$ We then have a faithful, normal, tracial state $\tau_{\ell}\colon M\to \C$ given by
\begin{equation}\label{eqn:GNS trace}
\tau_{\ell}(a)=\ip{a(1+W),1+W}, \mbox{ for $a\in M$}
\end{equation}
and by construction the law of $x$ with respect to $\tau_{\ell}$ is $\ell.$ It is an exercise using the spectral theorem to show that
\begin{equation}\label{eqn: recover infty norm}
  \|\pi_{\ell}(P)\|_{\infty}=\sup_{k}\ell((P^{*}P)^{k})^{1/2k}=\lim_{k\to\infty}\ell((P^{*}P)^{k})^{1/2k}.
\end{equation}
for all $P\in \C^{*}\ip{(T_{j})_{j\in J}}.$ So if
$R\in [0,\infty)^{J}$ and $\ell\in \Sigma_{R,J}$, then $\|\pi_{\ell}(x_{j})\|_{\infty}\leq R_{j}$ for all $j\in J.$

Laws may be viewed as a natural noncommutative extension of probability measures.
If $(M,\tau)$ is a tracial von Neumann algebra and $x\in M$ is \emph{normal}, we let $\mu_{x}\in \Prob(\C)$ be the spectral measure of $x$ defined by $\mu_{x}(E)=\tau(1_{E}(x))$ for all Borel $E\subseteq C.$ Then, by definition, for all $P\in \C^{*}\ip{T}$ we have
\begin{equation}\label{E:spectral measure defn}
\tau(P(x))=\int P(z)\,d\mu_{x}(z).
\end{equation}
Here we are using $P(z)$ for the image of $T$ under the unique $*$-homomorphism $\C^{*}\ip{T}\to \C$ given by $T\mapsto z.$ Of course, $\C^{*}\ip{T}$ is noncommutative, whereas $z\mapsto P(z)$ is given by a (different) \emph{commutative} polynomial in $z$ and $\overline{z}.$ Since $\mu_{x}$ is compactly supported, the Stone-Weierstrass theorem tells us that equation (\ref{E:spectral measure defn}) \emph{uniquely} determines $\mu_{x}.$ This equation may be read as
\[\ell_{x}(P)=\int P(z)\,d\mu_{x}(z)\]
and so we see that the law of $x$ encodes the same information as the spectral measure of $x.$

Fix a set $J.$ Since $\Sigma_{J}$ is a subset of the algebraic dual of $\C^{*}\ip{(T_{j})_{j\in J}}$ it can be naturally endowed with the weak$^{*}$-topology. So a basic neighborhood of $\ell\in \Sigma_{J}$ is given by
\[U_{F,\varepsilon}(l) =\bigcap_{Q\in F}\{\phi \in \Sigma_{J}:|\phi(Q)-\ell(Q)|<\varepsilon\}\]
for a finite $F\subseteq \C^{*}\ip{(T_{j})_{j\in J}}$ and an $\varepsilon>0.$
We leave it as an exercise to verify that for every $R\in [0,\infty)^{J}$ we have that $\Sigma_{R,J}$ is compact in the weak$^{*}$-topology.

We recall the  notion of \emph{freely independent random variables}, which forms the basis for Voiculescu's free probability. Let $(M,\tau)$ be a tracial von Neumann algebra, and let $(A_{j})_{j\in J}$ be $*$-subalgebras of $M.$ We say that $(A_{j})_{j\in J}$ are \emph{freely independent} (or \emph{free}) if for all $n\in \N,$ and all $j\in J^{n}$ with $j_{1}\ne j_{2}$, $j_{2}\ne j_{3}$, $j_{3}\ne j_{4}$, $\cdots$, $j_{n-1}\ne j_{n},$ and for all $a\in M^{n}$ with $a_{i}\in A_{j_{i}}$ and $\tau(a_{i})=0$ we have
\[\tau(a_{1}a_{2}\cdots a_{n})=0.\]
Say that $(x_{j})_{j\in J}\in M^{J}$ are \emph{freely independent} (or \emph{free}) if the $*$-algebras they generate are free as a $J$-tuple. This necessarily forces $(W^{*}(x_{j}))_{j\in J}$ to be free. Given any collection $(M_{j},\tau_{j})_{j\in J},$ one may find (see \cite[Chapter 1]{VoiculescuDykemaNica}) another tracial von Neumann algebra $(M,\tau)$ for which there are trace-preserving embeddings $M_{j}\hookrightarrow M$ so that if we identify $M_{j}$ with its image under this embedding, then $M=W^{*}\left(\bigcup_{j}M_{j}\right),$ and $(M_{j})_{j\in J}$ are free. If $(\widetilde{M},\widetilde{\tau})$ is another such algebra, then there is a unique trace-preserving isomorphism $(M,\tau)\cong (\widetilde{M},\widetilde{\tau})$ which respects the embeddings of $(M_{j},\tau_{j})$ into $(M,\tau),$ $(\widetilde{M},\widetilde{\tau})$ for each $j\in J.$ So we may define the free product of $(M_{j},\tau_{j})$, denoted $\ast_{j\in J}(M_{j},\tau_{j}),$ to be any such algebra. Given index sets $(J_{i})_{i\in I},$ and $\ell_{i}\in \Sigma_{J_{i}}$ define $\ell=\ast_{i\in I}\ell_{i}\in \Sigma_{\sqcup_{i}J_{i}}$ to be the law of $x=(\pi_{\ell_{i}}(T_{j}))_{j\in J_{i},i\in I}$ in $\ast_{i\in I}(W^{*}(\ell_{i}),\tau_{\ell_{i}}).$ By its very nature, freeness of noncommutative variables depends only upon their joint law. So we will often omit reference to the underlying von Neumann algebra. For example, we will often say ``suppose $x=(x_{1},\cdots,x_{r})$ is a free tuple". Provided we specify $\ell_{x_{j}}$ for all $j,$ this unambiguously gives $\ell_{x}.$ Since many of our results only require knowledge of the law of $x,$ this will suffice for our purposes. One case of utmost importance is the following. A tuple $s=(s_{1},\cdots,s_{r})$ is a \emph{free semicircular family} if it is a free family, each $s_{j}$ is self-adjoint, and for each $j$ we have that
\[d\mu_{s_{j}}=\frac{1}{2\pi\sigma^{2}}\sqrt{4\sigma^{2}-(x-\mu)^{2}}1_{[\mu-2\sigma,\mu+2\sigma]}\,dx\]
for some $\mu\in \R,$ $\sigma\in (0,\infty).$

\begin{defn}
Let $J$ be an index set, and fix $R\colon [0,\infty)\to J.$ Given a set $\mathcal{O}\subseteq \Sigma_{R,J}$ with nonempty interior (relative to  $\Sigma_{R,J}$) and an $k\in \N,$ we define \emph{Voiculescu's space of $(\mathcal{O},k)$ microstates} to be
\[\Gamma^{(k)}_{R}(\mathcal{O})=\{A\in M_{k}(\C)^{J}:\ell_{A}\in \mathcal{O}, \|A_{j}\|_{\infty}\leq R_{j}\mbox{ for all $j\in J$}\}.\]
\end{defn}

The reader may be more familiar with the following case. Let $(M,\tau)$ be a tracial von Neumann algebra, and let $J$ be an index set. Let $x\in M^{J}$ and choose $R\in [0,\infty)^{J}$ with $\|x_{j}\|< R_{j}$ for all $j\in J.$ Then it is typical to consider $\Gamma^{(k)}_{R}(\mathcal{O})$ for $\mathcal{O}$ a weak$^{*}$-neighborhood of $x.$ Indeed, it is common to denote this by $\Gamma^{(k)}_{R}(x;\mathcal{O})$ even though it does not require $x$ for its definition. If $W^{*}(x)=M,$ then we have that $M$ embeds into an ultrapower of the hyperfinite $\textrm{II}_{1}$-factor if and only if for every neighborhood $\mathcal{O}$ of the law of $x$ in $\Sigma_{R,J}$ there is an integer $k\in \N$ so that
\[\Gamma^{(k)}_{R}(\mathcal{O})\ne \varnothing.\]
Because of this, the spaces $\Gamma^{(k)}_{R}(\mathcal{O})$ are often regarded as spaces of ``finitary approximations" of $x,$ and they form the basis for microstates free entropy, microstates free entropy dimension, and the $1$-bounded entropy of $x.$

There is a mild, but very important, variant of this which takes into account microstates which have an ``extension" to a larger algebra. Suppose that $(M,\tau)$ is a tracial von Neumann algebra, and that $I,J$ are index sets. Let $y\in M^{I},$  and fix $R\in [0,\infty)^{I\sqcup J}$ with $\|y_{i}\|_{\infty}\leq R_{i}$ for all $i\in I.$ Let $\mathcal{O}\subseteq \Sigma_{R,I\sqcup J}$ and assume that
\[\{\ell\big|_{\C^{*}\ip{(T_{i})_{i\in I}}}:\ell \in \mathcal{O}\}\]
is a neighborhood of $\ell_{y}.$ For an integer $k\in \N$ we define \emph{Voiculescu's microstates space for $y$ in the presence of $\mathcal{O}$}, denoted $\Gamma_{R}^{(k)}(y:\mathcal{O}),$ by
\[\Gamma_{R}^{(k)}(y:\mathcal{O})=\{A\in M_{k}(\C)^{I}:\mbox{ there exists a $B\in M_{k}(\C)^{J}$ with $\ell_{A,B}\in \mathcal{O}$}\}.\]
Typically, one takes $y\in M_{k}(\C)^{J}$ and $\mathcal{O}$ to be a neighborhood of $\ell_{y,x}.$ In this case, one thinks of $\Gamma_{R}^{(k)}(y:\mathcal{O})$ as ``microstates for $y$ which have an extension to microstates for $(y,x)$".

Given a set $\Omega,$ a \emph{pseudometric} on $\Omega$ is a function $d\colon \Omega\times \Omega\to [0,\infty)$ satisfying
\begin{itemize}
    \item $d(x,y)=d(y,x)$ for all $x,y\in \Omega$,
    \item $d(x,z)\leq d(x,y)+d(y,z)$ for all $x,y,z\in \Omega.$
\end{itemize}
Given a pseudometric $d$ on $\Omega,$ an $r>0,$ and an $x\in \Omega,$ we let
\[B_{r}(x,d)=\{y\in \Omega:d(x,y)<r\}.\]
For $E\subseteq X$ and $r>0,$ we let
\[N_{r}(E,d)=\bigcup_{x\in E}B_{r}(x,d).\]
We call $N_{r}(E,d)$ the $r$-neighborhood of $E.$
For $\varespilon>0$ and $E\subseteq \Omega,$ we let $K_{\varepsilon}(E,d)$ be the minimal cardinality of a set $F\subseteq E$ which has $N_{\varpesilon}(F,d)\supseteq E.$ If there is no such finite set $F,$ then by convention $N_{\varepsilon}(F,d)=\infty.$  If $F\subseteq
\C,$ then $N_{\varepsilon}(F)$ will refer to the $\varepsilon$-neighborhood of $F$ with respect to the Euclidean distance on $\C$ and we will not make reference to the fact that we are using the Euclidean distance.

The most important pseudometric for our purposes is the following. Given an index set $J,$ a  finite set $F\subseteq J,$ and a natural number $k,$ we define a pseudometric $d^{\orb{}}_{F}$ on $M_{k}(\C)^{J}$ by
\[d^{\orb{}}_{F}(A,B)=\inf_{U\in U(k)}\|A-UBU^{*}\|_{2,F}.\]
As in the case of $\|\cdot\|_{2,F}$ if $J$ itself is finite we will usually use $d^{\orb{}}$ instead of $d^{\orb{}}_{J}$.

\begin{defn}\label{defn:S1B}
Let $(M,\tau)$ be a tracial von Neumann algebra, and $y\in M^{I},x\in M^{J}.$ Fix $R\in [0,\infty)^{I\cup J}$ with $\|x_{j}\|_{\infty}\leq R_{j}$ for all $j\in J,$ and $\|y_{i}\|_{\infty}\leq R_{i}$ for all $i\in I.$ For a weak$^{*}$-neighborhood $\mathcal{O}$ of $\ell_{y,x}$ and a finite $F\subseteq I,$ we set
\[K_{\varepsilon,F}^{\orb{}}(y:\mathcal{O},\|\cdot\|_{2})=\limsup_{k\to\infty}\frac{1}{k^{2}}\log K_{\varepsilon}(\Gamma_{R}^{(k)}(y:\mathcal{O}),d^{\orb{}}_{F}). \]
We then define
\[K_{\varepsilon,F}^{\orb{}}(y:x)=\inf_{\mathcal{O}}K_{\varepsilon,F}^{\orb{}}(y:\mathcal{O},\|\cdot\|_{2}),\]
\[h(y:x)=\sup_{\varepsilon, F}K_{\varepsilon,F}^{\orb{}}(y:x),\]
where the infimum is over all weak$^{*}$-neighborhoods $\mathcal{O}$ of $\ell_{y,x}$ and the supremum is over all $\varepsilon>0$ and finite subsets $F$ of $J.$  We call $h(y:x)$ the \emph{1-bounded entropy of $y$ in the presence of $x$}.
\end{defn}
It follows from \cite[Theorem A.9]{Me8} that if $J',I'$ are other index sets, and if $y'\in M^{J'},x'\in M^{I'},$ and  $W^{*}(y)=W^{*}(y'),W^{*}(x,y)=W^{*}(x',y'),$ then
\[h(y:x)=h(y':x').\]
Suppose $N\leq M,$ and that $N$ is diffuse. If $y\in N^{J},x\in M^{I}$ with $W^{*}(y)=N,W^{*}(x,y)=M,$ we may define the \emph{1-bounded entropy of $N$ in the presence of $M$} by
\[h(N:M)=h(y:x).\]
We think of $h(N:M)$ as some precise measurement of the ``size of the space of microstates for $N$ which have an extension to $M$".
Note that since we allow arbitrary index sets, the quantity $h(N:M)$ is \emph{always} defined (provided $N$ is diffuse), though it may be $-\infty.$ We set $h(M)=h(M:M)$, and call $h(M)$ the \emph{$1$-bounded entropy of $M$}.

We now turn to permanence properties the $1$-bounded entropy enjoys. We say that a von Neumann algebra $M$ is \emph{hyperfinite} if there is an increasing net $(M_{\alpha})_{\alpha}$ of finite-dimensional von Neumann subalgebras of $M$ with
\[M=\overline{\bigcup_{\alpha}M_{\alpha}}^{WOT}.\]
By a celebrated result of Connes' \cite{Connes}, this is equivalent to several other properties of $M$ such as being amenable  (\cite[Chapter 10]{anantharaman-popa}).  Because of Connes' famous and deep work we will use hyperfinite and amenable interchangeably. We use $\cR$ for the (unique modulo isomorphism) hyperfinite $\textrm{II}_{1}$-factor.
 For a tracial von Neumann algebra $(M,\tau)$ and $x\in M^{J}$ for some set $J,$ we let $\delta_{0}(x)$ be the microstates free entropy dimension of $x.$ We will not need the precise definition, and refer the reader to \cite[Definition 6.1]{FreeEntropyDimensionII} for the details.  We assume that all our von Neumann algebras are diffuse for all the properties listed below.

\begin{list}{{\bf P\arabic{pcounter}:} ~ }{\usecounter{pcounter}}

\item $h(N:M)\geq 0$ if $N\leq M$ and every von Neumann subalgebra of $M$ with separable predual embeds into an ultrapower of $\cR,$ and $h(N:M)=-\infty$ if there exists a von Neumann subalgebra of $M$ with separable predual which does not embed into an ultrapower of $\cR$. (Exercise from the definitions.)
\item $h(N_{1}:M_{1})\leq h(N_{2}:M_{2})$ if $N_{1}\leq N_{2}\leq M_{2}\leq M_{1}$. \label{I:monotonicity of 1 bounded entropy} (Exercise from the definitions.)

\item $h(N:M)=0$ if $N\leq M$ and $N$ is diffuse and hyperfinite. \label{I:hyperfinite has 1-bdd ent zero}
(Exercise from the definitions).

\item For $M$ diffuse, $h(M)<\infty$ if and only if $M$ is strongly $1$-bounded in the sense of Jung. (See \cite[Proposition A.16]{Me8}). \label{I:preserves SB}

\item $h(M)=\infty$ if $M = \mathrm{W}^*(x_{1},\cdots,x_{n})$ where $x_{j}\in M_{sa}$ for all $1\leq j\leq n$ and $\delta_{0}(x_{1},\cdots,x_{n})>1$. For example, this applies if $M=L(\F_{n}),$ for $n>1$. (This follows from Property \ref{I:preserves SB} and \cite[Corollary 3.5]{JungSB}).

\item $h(N_{1}\vee N_{2}:M)\leq h(N_{1}:M)+h(N_{2}:M)$ if $N_{1},N_{2}\leq M$ and $N_{1}\cap N_{2}$ is diffuse. (See \cite[Lemma A.12]{Me8} \label{I:subadditivity of 1 bdd ent}).

\item Suppose that $(N_{\alpha})_{\alpha}$ is an increasing chain of diffuse von Neumann subalgebras of a von Neumann algebra $M$. Then
\[h\left(\bigvee_{\alpha}N_{\alpha}:M\right)=\sup_{\alpha}h(N_{\alpha}:M).\]
\label{I:increasing limits of 1bdd ent first variable} (See \cite[Lemma A.10]{Me8}).

\item $h(N:M)=h(N:M^{\omega})$ if $N\leq M$ is diffuse, and $\omega$ is a free ultrafilter on an infinite set. (See \cite[Proposition 4.5]{Me8}). \label{I:omegafying in the second variable}

\item $h(\mathrm{W}^*(\cN_{M}(N)):M)=h(N:M)$ if $N\leq M$ is diffuse.  Here $\cN_{M}(N)=\{u\in \mathcal{U}(M):uN u^{*}=N\}.$ (This is a special case of \cite[Theorem 3.8]{Me8}). \label{I:passing to normalizers preserves entropy}

\end{list}
These properties are sufficient by themselves to deduce  the landmark results of Voiculescu \cite{FreeEntropyDimensionIII}, Ge \cite{GePrime} that free group factors do not have Cartan subalgebras, and are prime, as well as the fact that a von Neumann algebra generated by a family with free entropy dimension bigger than $1$ does not have Property Gamma. We refer the reader to \cite[Section 1.2]{FreePinsker} for a more detailed discussion on this.

The $1$-bounded entropy allows us to single out a particularly nice set of von Neumann subalgebras of a fixed tracial von Neumann algebra.

\begin{defn}
Suppose $(M,\tau)$ is a tracial von Neumann algebra, and that $P\leq M.$ We say that $P$ is a \emph{Pinsker algebra in $M$} if $h(P:M)\leq 0$ and for every $P\leq Q\leq M$ with $Q\ne P$ we have $h(Q:M)>0.$
\end{defn}

Recall that if $M$ is a von Neumann algebra, then $Q\leq M$ is \emph{maximal amenable} if $Q$ is amenable and for every $N\leq M$ with $Q\subseteq N$ and $N$ amenable, we have $N=Q.$ It follows from Property P\ref{I:subadditivity of 1 bdd ent} that if $Q\leq M$ is diffuse, and $h(Q:M)=0,$ then there is a \emph{unique} Pinsker $P\leq M$ with $Q\leq P.$
If $P\leq M$ is Pinsker \emph{and} amenable, it is necessarily \emph{maximal amenable} by Property P\ref{I:hyperfinite has 1-bdd ent zero}. Moreover, by P\ref{I:subadditivity of 1 bdd ent}, P\ref{I:hyperfinite has 1-bdd ent zero} it has the \emph{absorbing amenability property}. Namely, if $Q\leq M$ is amenable and $Q\cap P$ is diffuse, then $Q\leq P.$ It also has the following \emph{Gamma stability}  property (in the sense of \cite{CyrilAOP}): if $Q\leq M$ is such that $Q'\cap M^{\omega}$ and $Q\cap P$ are diffuse, then $Q\leq P.$

Of relevance to the Peterson-Thom conjecture is the following.

\begin{prop}\label{P:CPE implies PT}
Let $(M,\tau)$ be a tracial von Neumann algebra. Suppose that every Pinsker algebra in $M$ is amenable. Then given any diffuse, amenable $Q\leq M$ there is a unique, maximal amenable $P\leq M$ with $Q\subseteq P.$
\end{prop}

\begin{proof}
Suppose that $Q\leq M$ is diffuse and amenable. Let $P\leq M$ be the unique Pinsker algebra in $M$ with $Q\subseteq P.$ By assumption, $P$ is amenable. Since $P$ is Pinsker, it is necessarily maximal amenable. Suppose $\widehat{P}\leq M$ is another maximal amenable subalgebra of $M$ with $Q\subseteq \widehat{P}.$ Then
\[\widehat{P}\cap P\supseteq Q,\]
and since $Q$ is diffuse this forces $\widehat{P}\cap P$ to be diffuse. So  $h(P\vee \widehat{P}:M)\leq 0$, by Property P\ref{I:subadditivity of 1 bdd ent}. Since $P$ is Pinsker, $\widehat{P}\vee P\leq P$, which forces $\widehat{P}\subseteq P.$ Since $\widehat{P}$ is maximal amenable, we have $\widehat{P}=P$ and this completes the proof.
\end{proof}

The weak$^{*}$-topology on laws is what allows us to define Voiculescu's microstates, and by extension the $1$-bounded entropy. We will  also need another topology on the space of laws. Fix an index set $J.$ Recall the definition of $\pi_{\ell}$ for $\ell\in \Sigma_{J}$ discussed after Definition $\ref{defn:laws}.$
For $P\in \C^{*}\ip{(T_{j})_{j\in J}}$, $\ell\in \Sigma_{J}$, set
\[\|P\|_{L^{\infty}(\ell)}=\|\pi_{\ell}(P)\|_{\infty}.\]

It is \emph{not true}  that the map
\[\Sigma_{J}\times \C^{*}\ip{(T_{j})_{j\in J}}\to [0,\infty]\]
given by
$(\ell,P)\mapsto \|P\|_{L^{\infty}(\ell)}$
is continuous in the first variable. However, from (\ref{eqn: recover infty norm}) we have the following semi-continuity: if $\ell_{\alpha}$ is a net in $\Sigma_{J}$ and $\ell_{\alpha}\to \ell$ weak$^{*}$, then for all $P\in \C^{*}\ip{(T_{j})_{j\in J}}$
\[\|P\|_{L^{\infty}(\ell)}\leq \liminf_{\alpha}\|P\|_{L^{\infty}(\ell_{\alpha})}.\]
This motivates the definition of a different topology on $\Sigma_{J}.$
\begin{defn}
Let $J$ be an index set. The \emph{strong topology} on $\Sigma_{J}$ is the coarsest topology finer than the weak$^{*}$-topology which makes the map $\Sigma_{J}\to [0,\infty)$ given by $P\mapsto \|P\|_{L^{\infty}(\ell)}$ continuous for each $P\in \C^{*}\ip{(T_{j})_{j\in J}}.$
\end{defn}
Given  $\ell\in \Sigma_{J},$ a neighborhood basis at $\ell$ in the strong topology may be given by
\[\mathcal{O}_{V,F,\varepsilon}(\ell)=V\cap \bigcap_{P\in F}\{\phi\in \Sigma_{J}:|\|P\|_{L^{\infty}(\ell)}-\|P\|_{L^{\infty}(\phi)}|<\varepsilon\}\]
ranging over weak$^{*}$-neighborhoods $V$ of $\ell,$ finite sets $F\subseteq \C^{*}\ip{(T_{j}){j\in J}},$ and $\varepsilon\in (0,\infty).$ In fact, by semicontinuity, the sets
\[\mathcal{V}_{V,F,\varepsilon}(\ell)=V\cap \bigcap_{P\in F}\{\phi\in \Sigma_{J}:\|P\|_{L^{\infty}(\phi)}<\|P\|_{L^{\infty}(\ell)}+\varepsilon\}\]
ranging over weak$^{*}$-neighborhoods $V$ of $\ell,$ finite sets $F\subseteq \C^{*}\ip{(T_{j}){j\in J}},$ and $\varepsilon\in (0,\infty)$ form a neighborhood basis of $\ell\in \Sigma_{J}$ in the strong topology.

 Suppose we are given a sequence $(M_{k},\tau_{k})_{k}$ of tracial von Neumann algebras, $x_{k}\in M_{k}^{J},$  another von Neumann algebra $(M,\tau)$ and $x\in M^{J}.$ We will then say that \emph{the distribution of $x_{k}$ converges strongly to the distribution of $x$} if $\ell_{x_{k}}\to \ell_{x}$ in the strong topology. Concretely, this is just the conjunction of the following two properties:
 \begin{itemize}
     \item $\tau_{k}(P(x_{k}))\to \tau(P(x))$
for all $P\in \C^{*}\ip{(T_{j})_{j\in J}},$ and
\item $\|P(x_{k})\|_{\infty}\to \|P(x)\|_{\infty}$ for all $P\in \C^{*}\ip{(T_{j})_{j\in J}}.$
 \end{itemize}
 So our notion of strong convergence agrees with that already discussed in \cite{Male, MaleCollins}. To further illustrate the  meaning of strong convergence, we close with the following Lemma relating strong convergence to Hausdorff convergence of spectra. This is a well-known result, and we mainly prove it to give the reader some insight and practice as to what strong convergence is and why it is important. Recall that the Hausdorff metric on nonempty, compact subsets of $\C$ is given by
 \[d_{\Haus{}}(E,F)=\inf\{r>0:E\subseteq N_{r}(F)\mbox{ and } F\subseteq N_{r}(E)\}.\]

 \begin{lem}\label{L:strong is just Haus}
 Fix an index set $J.$ Let $(M_{k},\tau_{k}),k\in \N$ be a sequence of tracial von Neumann algebras and $x_{k}\in M_{k}^{J}.$ Let $(M,\tau)$ be a tracial von Neumann algebra and $x\in M^{J}.$ Assume that $\sup_{k}\|x_{k,j}\|_{\infty}<\infty$
 for all $j\in J.$ Suppose that $\ell_{x_{k}}\to \ell_{x}$ weak$^{*}.$  Then $\ell_{x_{k}}\to\ell_{x}$ strongly if and only if  for every $P\in \C^{*}\ip{(T_{j})_{j\in J}}$ with $P=P^{*}$ we have that $\sigma(P(x_{k}))\to \sigma(P(x))$ in the Hausdorff metric.
 \end{lem}

\begin{proof}
First, suppose that for every self-adjoint $P\in \C^{*}\ip{(T_{j})_{j\in J}}$ we have that $\sigma(P(x_{k}))\to \sigma(P(x))$ in the Hausdorff metric. Now fix $Q\in \C^{*}\ip{(T_{j})_{j\in J}},$ and let $\varepsilon>0.$  Then for all sufficiently large $k,$ we have that $\sigma((Q^{*}Q)(x_{k}))\subseteq N_{\varepsilon}(\sigma((Q^{*}Q)(x))).$ Since the norm of a self-adjoint element is given by its spectral radius, it follows that
$\|Q(x_{k})\|_{\infty}^{2}=\|(Q^{*}Q)(x_{k})\|_{\infty}\leq \varepsilon+\|(Q^{*}Q)(x)\|_{\infty}= \varepsilon+\|Q(x)\|_{\infty}^{2}$
for large $k.$ Since $\varepsilon>0$ is arbitrary, we have shown that
$\limsup_{k\to\infty}\|Q(x_{k})\|_{\infty}\leq \|Q(x)\|_{\infty}.$
The fact that
$\|Q(x)\|_{\infty}\leq \liminf_{k\to\infty}\|Q(x_{k})\|_{\infty}$ is already a consequence of weak$^{*}$  convergence of $\ell_{x_{k}}$ to $\ell_{x}.$

For the reverse direction, assume that $\ell_{x_{k}}\to \ell_{x}$ strongly. First choose a $M>0$ so that $\|P(x_{k})\|_{\infty}\leq M$ for all $k.$ This is possible as $\sup_{k}\|x_{k,j}\|_{\infty}<\infty$ for all $j\in J.$ Note that we have $\|f(P(x_{k}))\|_{\infty}\to \|f(P(x))\|_{\infty}$ for all $f\in C([-M,M]).$ Indeed, the set of $f\in C([-M,M])$ for which $\|f(P(x_{k}))\|_{\infty}\to_{k\to\infty} \|f(P(x))\|_{\infty}$ can be directly shown to be a closed subset of $C([-M,M]),$ and by our assumption of strong convergence it contains all polynomials. So $\|f(P(x_{k}))\|_{\infty}\to \|f(P(x))\|_{\infty}$ for all $f\in C([-M,M])$ by the Stone-Weierstrass theorem. Let $\varepsilon>0,$ and apply Urysohn's Lemma to find a continuous function $\phi\in C([-M,M])$ which is $0$ on $\sigma(P(x))$ and is $1$ on $N_{\varepsilon}(\sigma(P(x))^{c}\cap [-M,M].$ Then $\|\phi(P(x_{k}))\|_{\infty}\to_{k\to\infty} \|\phi(P(x))\|_{\infty}=0,$ and so for all large $k$ we have $\|\phi(P(x_{k}))\|_{\infty}<\frac{1}{2}.$ By the spectral mapping theorem, $\sigma(\phi(P(x_{k})))=\phi(\sigma(P(x_{k}))).$ Since $\phi=1$ on $N_{\varepsilon}(\sigma(P(x)))^{c}\cap [-M,M],$ and $\|\phi(P(x_{k}))\|$ is the supremum of $|\phi|$ over $\sigma(P(x_{k})),$ it follows that $\sigma(P(x_{k}))\subseteq N_{\varepsilon}(\sigma(P(x)))$ for all large $k.$

So it just remains to show that $\sigma(P(x))\subseteq N_{\varepsilon}(\sigma(P(x_{k})))$ for all large $k.$ For every $t\in \sigma(P(x)),$ choose a $\psi_{t}\in C([-M,M])$ with $\psi_{t}(t)=1$ and $\psi_{t}\big|_{B_{\varepsilon/2}(t)^{c}\cap [-M,M]}=0.$ As above, there is a $K_{t}\in \N$ so that for all $k\geq K_{t}$ we have $\|\psi_{t}(P(x_{k}))\|_{\infty}\geq 1/2.$ As in the above paragraph, this implies that $B_{\varepsilon/2}(t)\cap \sigma(P(x_{k}))\ne \varnothing$ and so $t\in N_{\varepsilon/2}(\sigma(P(x_{k})))$ for all $k\geq K_{t}.$ Since $\sigma(P(x))$ is compact, we can choose $t_{1},\cdots,t_{n}\in \sigma(P(x))$ so that $\sigma(P(x))\subseteq N_{\varepsilon/2}(\{t_{1},\cdots,t_{n}\}).$ Set $K=\max(K_{t_{1}},\cdots,K_{t_{n}}).$ Then for all $k\geq K,$
\[\sigma(P(x))\subseteq N_{\varepsilon/2}(\{t_{1},\cdots,t_{n}\})\subseteq N_{\varepsilon}(\sigma(P(x_{k}))).\]

\end{proof}

Note that the weak$^{*}$-convergence of $\ell_{x_{k}}\to \ell_{x}$ implies the weak$^{*}$-convergence of $\mu_{P(x_{k})}$ to $\mu_{P(x)}$ for all self-adjoint $P\in \C^{*}\ip{(T_{j})_{j\in J}}$. The above lemma than asserts that strong convergence of $\ell_{x_{k}}\to \ell_{x}$ means that there one does not have any ``outliers" in the spectrum of $P(x_{k}).$ In random matrices, this is often called having a ``hard edge."

We phrased Lemma \ref{L:strong is just Haus} in terms of strong convergence of laws of specific elements because it is more natural for the reader who might have some experience with random matrices. For example, if $A^{(k)}\in M_{n(k)}(\C)^{J}$ and if there is a tracial von Neumann algebra $(M,\tau)$ and a tuple $x\in M^{J}$ with $\ell_{A^{(k)}}\to \ell_{x}$ weak$^{*},$ then strong convergence of $A^{(k)}$ to $x$ just asserts that for  any self-adjoint $P\in \C^{*}\ip{(T_{j})_{j\in J}}$ the spectral distribution of $P(A^{(k)})$ converges weak$^{*}$ to the spectral measure of $P(x),$ and the spectrum of $P(A^{(k)})$ converges to the spectrum of $P(x)$ in the Hausdorff sense. However, one can phrase Lemma \ref{L:strong is just Haus} without referring to any ambient tracial von Neumann algebra. Suppose we have a sequence $\ell_{n}\in \Sigma_{R,J}$ for some $R\in [0,\infty)^{J}$ with $\ell_{n}\to \ell$ weak$^{*}.$ Lemma \ref{L:strong is just Haus}  is then equivalent to saying that $\ell_{n}\to \ell$ strongly if and only if for all self-adjoint $P\in \C^{*}\ip{(T_{j})_{j\in J}}$ we have $\sigma(\pi_{\ell_{k}}(P))\to \sigma(\pi_{\ell}(P))$ in the Hausdorff sense.

\subsection{Measures on microstates and concentration thereof}\label{S:measures}

\begin{defn}
Let $(M,\tau)$ be a tracial von Neumann algebra, $I$ an index set, and $x\in M^{J}.$  Suppose we have a  sequence $n(k)\in \N$ with $n(k)\to \infty,$ and $\mu^{(k)}\in \Prob(M_{n(k)}(\C)^{I}).$ Then we say that $\mu^{(k)}$ is \emph{asymptotically supported on microstates for $x$} if  there exists an $R\in [0,\infty)^{I}$ with
\begin{itemize}
    \item $\|x_{i}\|_{\infty}\leq R_{i}$ for all $i\in I,$ \item  we have
$\mu^{(k)}(\Gamma_{R}^{(n(k))}(\mathcal{O}))\to_{k\to\infty}1$ for every weak$^{*}$-neighborhood $\mathcal{O}$ of $\ell_{x}.$
\end{itemize}
Suppose $J$ is another index set, $y\in M^{J},$ and $\nu^{(k)}\in \Prob(M_{k}(\C)^{J}).$ Then we say that $\nu^{(k)}$ is \emph{asymptotically supported on microstates for $y$ in the presence of $x$} if there is an $R\in [0,\infty)^{J\sqcup I}$ so that
\begin{itemize}
    \item $\|x_{i}\|_{\infty}\leq R_{i}$ for all $i\in I,$
    \item $\|y_{j}\|_{\infty}\leq R_{j}$ for all $j\in J,$
    \item $\nu^{(k)}(\Gamma_{R}^{(n(k))}(y:\mathcal{O}))\to 1$ for every weak$^{*}$-neighborhood $\mathcal{O}$ of $\ell_{y,x}.$
\end{itemize}
\end{defn}

Recall that a \emph{Polish space} is a topological space $X$ which is separable and completely metrizable. Such a space is naturally equipped with its Borel $\sigma$-algebra, which is the $\sigma$-algebra generated by the open subsets of $X.$ For a Polish space $X,$ we let $\Prob(X)$ be the space of Borel probability measures on $X.$

\begin{defn}
 A \emph{pseudometric measure space} is a triple $(X,\mu,d)$ where
\begin{itemize}
\item $X$ is a Polish space,
\item $\mu\in \Prob(X),$
\item $d$ is a continuous pseudometric on $X$ (giving $X\times X$ the product topology).
\end{itemize}
Given a pseudometric measure space $(X,\mu,d)$ we define its \emph{concentration function}
$\alpha_{\mu,d}\colon (0,\infty)\to [0,1]$
by
\[\alpha_{\mu,d}(\varpesilon)=\inf\{\mu(N_{\varepsilon}(E,d)^{c}):E\subseteq X\mbox{ is Borel and } \mu(E)\geq 1/2\}.\]
\end{defn}
An alternative (and typically more useful) way to view the concentration function is as follows. If $E\subseteq X$ is Borel, and $\mu(E)\geq 1/2,$ then
\[\mu(N_{\varepsilon}(E,d))\geq 1-\alpha_{\mu,d}(\varepsilon).\]
Typically one is interested in sequences of pseudometric measure spaces $(X_{k},\mu_{k},d_{k})$ so that $\alpha_{\mu_{k},d_{k}}(\varepsilon)$ decays rapidly for each fixed $\varespilon>0.$

\begin{defn}\label{defn:ECM}
Let $(X_{k},\mu_{k},d_{k})$ be a sequence of pseudometric measures spaces,  $\mu^{(k)}\in \Prob(X_{k}),$ and $(r_{k})_{k}\in (0,\infty)^{\N}$ with $r_{k}\to \infty.$ We say that $(X_{k},\mu^{(k)},d_{k})$ exhibits \emph{exponential concentration at scale $r_{k}$} if for every $\varepsilon>0,$
\[\limsup_{k\to\infty}\frac{1}{r_{k}}\log \alpha_{\mu^{(k)},d_{k}}(\varepsilon)<0.\]
Suppose that $n(k)\in \N,$ with $n(k)\to\infty,$ and that $J$ is a set. Suppose that  we have a sequence $\mu^{(k)}\in \Prob(M_{k}(\C)^{J}).$ Then we say that $\mu^{(k)}$ exhibits \emph{exponential concentration at scale $n(k)^{2}$} if for every finite $F\subseteq J$ the sequence $(M_{n(k)}(\C)^{J},\mu^{(k)},d_{F}^{\orb{}})$ exhibits exponential concentration at scale $n(k)^{2}.$
\end{defn}
If it is clear form the context, we will often drop reference to $X_{k},d_{k}$ and say ``$\mu_{k}$ exhibits exponential concentration at scale $r_{k}$."
As we shall  show shortly (see Lemma \ref{L:concentration expansion}) exponential concentration implies that if $E_{k}\subseteq X_{k}$ are Borel and asymptotically have ``nontrivial" size, i.e.
\[\lim_{k\to\infty}\mu^{(k)}(E_{k})^{1/r_{k}}=1,\]
then
\[\lim_{k\to\infty}\mu^{(k)}(N_{\varepsilon}(E_{k},d))=1\]
for every $\varepsilon>0.$
So for every sequence of Borel subsets of $X_{k}$ which are ``not exponentially small", then, no matter how small $\varepsilon$ is, expanding $E_{k}$ to its $\varpesilon$-neighborhood makes it ``nearly everything." This is the reason for the name ``concentration function", it gives precise control over the rate at which the measures must concentrate near sets that are not ``exponentially small". Despite being a very strong concentration property, there are nevertheless many examples of natural sequences of metric measure spaces which satisfy exponential concentration, and this concept is of frequent use in probability theory and functional analysis.

\subsection{$L^{2}$-Continuous Functional Calculus}\label{S:nc func calc}

We discuss a generalization $L^{2}$-continuous functional calculus of Jekel developed in \cite[Section 3]{JekelEAP}, and developed further  in \cite[Section 2]{FreePinsker}. That functional calculus was defined for self-adjoint noncommutative variables and we will need the version for general variables defined in \cite[Section 13.7]{JekelThesis}. We start by recalling the construction and general properties in the self-adjoint case, and then explain how to give the appropriate definition in generality and derive the corresponding results from the self-adjoint case.

We will need to introduce some notation for the self-adjoint case. Given an index set $J,$ we let $\C\ip{(T_{j})_{j\in J}}$ be the algebra of noncommutative polynomials (\emph{not} $*$-polynomials) in the abstract variables $(T_{j})_{j\in J}.$ We may view $\C\ip{(T_{j})_{j\in J}}$ as the free $\C$-algebra indexed by $J.$ We turn $\C\ip{(T_{j})_{j\in J}}$ into a $*$-algebra, by giving it the unique $*$-structure which makes $T_{j}$ self-adjoint for all $j\in J.$ When viewed as a $*$-algebra, we may think of $\C\ip{(T_{j})_{j\in J}}$ as the universal $*$-algebra generated by self-adjoint elements indexed by $J.$ We will also need a space of self-adjoint laws. We adopt similar notational conventions as in the non-self-adjoint case, e.g. if $J=\{1,\cdots,n\}$ we will typically use $\C\ip{T_{1},\cdots,T_{n}}$ instead of $\C\ip{(T_{j})_{j\in J}}$.

\begin{defn}
Let $J$ be an index set, and $R\in [0,\infty)^{J}.$ Let $\ev_{T}\colon\C^{*}\ip{(S_{j})_{j\in J}}\to \C\ip{(T_{j})_{j\in J}}$ be the unique $*$-homomorphism satisfying $\ev_{T}(S_{j})=T_{j}$ for all $j\in J.$ Let $\Sigma_{J}^{(s)}$ be the set of all linear functionals $\ell\colon \C\ip{(T_{j})_{j\in J}}\to \C$ so that $\ell\circ \ev_{T}\in \Sigma_{J}.$ We let
\[\Sigma_{R,J}^{(s)}=\{\ell\in \Sigma_{J}^{(s)}:\ell \circ \ev_{T}\in \Sigma_{R,J}\}.\]
\end{defn}
Concretely, a linear functional $\ell\colon \C\ip{(T_{j})_{j\in J}}\to \C$ is in $\Sigma_{J}^{(s)}$ if and only if it satisfies the following axioms:
\begin{itemize}
\item $\ell(PQ)=\ell(QP)$ for all $Q,P\in \C\ip{(T_{j})_{j\in J}},$
\item $\ell(P^{*}P)\geq 0$ for all $P\in \C\ip{(T_{j})_{j\in J}},$
\item $\ell(1)=1,$
\item there is a $R\in [0,\infty)^{J}$ so that for all $n\in \N$, and all $j_{1},j_{2},\cdots,j_{n}\in J$,
\[|\ell(T_{j_{1}}T_{j_{2}}\cdots T_{j_{n}})|\leq R_{j_{1}}R_{j_{2}}\cdots R_{j_{n}}.\]
\end{itemize}
Moreover, given $\ell\in \Sigma_{J}^{(s)}$ and $R\in [0,\infty)^{J},$ we have that $\ell\in \Sigma_{R,J}^{(s)}$ if and only if it satisfies the fourth bullet point for this $R.$
\begin{defn}
Fix an index set $J$ and $R \in (0,+\infty)^{J}$.  Consider the space
\[
\mathcal{A}_{R,J}^{(s)} = C(\Sigma_{R,J}^{(s)}) \otimes_{\textnormal{alg}} \C\ip{(T_{j})_{j\in J}}
\]
Given  a tracial von Neumann algebra $(M,\tau)$ and $x \in M_{sa}^J$ with $\|x_{j}\|\leq R_j$, we define the evaluation map to be the linear map $\ev_{x}: \mathcal{A}_{R,J}^{(s)}\to M$ satisfying
\[\ev_{x}(\phi\otimes P)=\phi(\ell_{x})P(x), \mbox{ for $\phi \in C(\Sigma_{R,J}^{(s)})$, $P\in \C\ip{(T_{j})_{j\in J}}$.}\]
We then define a semi-norm on $\mathcal{A}_{R,J}^{(s)}$ by
\[
\|f\|_{R,2} = \sup_{(M,\tau), x} \|\ev_{x}(f)\|_{L^2(M,\tau)},
\]
where the supremum is over all tracial $\mathrm{W}^*$-algebras $(M,\tau)$ and all $x\in \cM_{sa}^I$ with $\|x_{j}\| \leq R_j$ for all $j\in J$. Denote by $\mathcal{F}_{R,J,2}^{(s)}$ the completion of $\mathcal{A}_{R,J}^{(s)} / \{f \in \mathcal{A}_{R,J}: \|f\|_{R,2} = 0\}$.
\end{defn}
In \cite{FreePinsker}, the superscripts $(s)$ are not there, so for example $\mathcal{A}_{R,J,2}^{(s)}$ is denoted by $\mathcal{A}_{R,J,2}$ etc. We have elected to use the superscript $(s)$ here to reference the fact that spaces $\Sigma_{R,J}^{(s)},$ $\mathcal{A}_{R,J,2}^{(s)},$ $\mathcal{F}_{R,J,2}^{(s)}$ are noncommutative function spaces of \emph{self-adjoint} variables, in contrast to the function spaces for non-self-adjoint variables we will discuss imminently.

 By construction, for every tracial von Neumann algebra $(M,\tau)$, and for every  $x \in M_{s.a.}^I$ with $\|x_j\| \leq R_j$, the evaluation map $\ev_{\mathbf{x}}: \mathcal{A}_{R,J}^{(s)} \to M$ extends to a well-defined $\mathcal{F}_{R,J,2}^{(s)} \to L^2(M,\tau)$, which we continue to denote by $\ev_{x}$, and we will also write $f(x) = \ev_{x}(f)$.
 If $(M,\tau)$ is a tracial von Neumann algebra, and $\xi\in L^{2}(M,\tau)\setminus M$ we set $\|\xi\|_{\infty}=\infty.$ For $f\in \mathcal{F}_{R,J,2}^{(s)}$ we set
 \[\|f\|_{R,\infty}=\sup_{x,(M,\tau)}\|f(x)\|_{\infty}\in [0,+\infty]\]
 where the supremum is over all tracial von Neumann algebras and all $x\in M_{s.a.}^{J}.$ We now recall the main properties of this construction, with pointers to \cite{FreePinsker} where the relevant details are shown. We let \[\mathcal{F}_{R,J,\infty}^{(s)}=\{f\in \mathcal{F}_{R,J,2}:\|f\|_{R,\infty}<\infty\}.\]

 \begin{list}{{\bf P\arabic{pcounter}:} ~ }{\usecounter{pcounter}}
 \item The natural multiplication, addition, and $*$-algebra operations on $C(\Sigma_{R,J}^{(s)})\otimes_{\textnormal{alg}}\C\ip{(T_{j})_{j\in J}}$ have a  unique  extension to $\mathcal{F}_{R,J,\infty}^{(s)}$ which satisfies
 \[\|f\|_{R,\infty}=\|f^{*}\|_{R,\infty},\,\,\ \|f\|_{R,2}=\|f^{*}\|_{R,2}\]
 \[\|fg\|_{R,\infty}\leq \|f\|_{R,\infty}\|g\|_{R,\infty},\,\,\,\, \|fg\|_{R,2}\leq \|f\|_{R,\infty}\|g\|_{R,2}.\]
 These operations together with the norm $\|\cdot\|_{R,\infty}$ turn $\mathcal{F}^{(s)}_{R,J,\infty}$ into a $C^{*}$-algebra. \cite[Lemma 2.3]{FreePinsker}
 \item For any tracial von Neumann algebra $(M,\tau),$ and any $x\in M_{s.a.}^{J}$ with $\|x_{j}\|_{\infty}\leq R_{j}$ for all $j\in J,$  the evaluation map $\ev_{x}\colon \mathcal{F}^{(s)}_{R,J,\infty}\to M$ is surjective \cite[Proposition 2.4]{FreePinsker}. In fact, given any $a\in M$ there is an $f\in \mathcal{F}^{(s)}_{R,J,\infty}$ with $\|f\|_{R,\infty}\leq \|a\|_{\infty}$ so that $\ev_{x}(f)=a.$
 \item Every $f\in \mathcal{F}^{(s)}_{R,J,\infty}$ is $\|\cdot\|_{2}$-uniformly continuous in the following sense. For every $\varespilon>0,$ there is a $\delta>0$ and a finite $F\subseteq J$ so that if $(M,\tau)$ is a tracial von Neumann algebra and $x,y\in M_{s.a.}^{J}$ with $\|x_{j}\|_{\infty}$,$\|y_{j}\|_{\infty}\leq R_{j}$ for all $j\in J$ and $\|x_{j}-y_{j}\|_{2}<\delta$ for all $j\in F,$ then $\|f(x)-f(y)\|_{2}<\varepsilon.$ \cite[Proposition 2.8]{FreePinsker}.
 \end{list}
We now wish to define an analogous space of ``noncommutative functions" when the variables are not self-adjoint, and we will want it to satisfy analogues of the above 3 properties. So fix an index set $J,$ and $R\in [0,\infty)^{J}.$
Define
\[\mathcal{A}_{R,J}=C(\Sigma_{R,J})\otimes_{\textnormal{alg}}\C^{*}\ip{(T_{j})_{j\in J}}.\]
Given a tracial von Neumann algebra $(M,\tau)$ and $x\in M^{J}$ with $\|x\|_{\infty}\leq R_{j}$ for all $j\in J,$ we let $\ev_{x}\colon \mathcal{A}_{R,J}\to M$ be the linear map satisfying $\ev_{x}(\phi\otimes P)=\phi(\ell_{x})P(x)$ for $\phi \in C(\Sigma_{R,J})$,$P\in \C^{*}\ip{(T_{j})_{j\in J}}$. For $f\in \mathcal{A}_{R,J,2}$ we will use $f(x)$ for $\ev_{x}(f).$
Define a seminorm $\|\cdot\|_{R,2}$ on $\mathcal{A}_{R,J}$ by
\[\|f\|_{R,2}=\sup_{x,(M,\tau)}\|f(x)\|_{2},\]
where the supremum is over all tracial von Neumann algebras $(M,\tau)$ and all $x\in M^{J}.$ We then let $\mathcal{F}_{R,J,2}$ be the completion of
\[\mathcal{A}_{R,J,2}/\{f\in \mathcal{A}_{R,J,2}:\|f\|_{R,2}=0\}\]
under the norm induced by $\|\cdot\|_{R,2}.$  For $f\in \mathcal{F}_{R,J,2}$ we let
\[\|f\|_{R,\infty}=\sup_{x, (M,\tau)}\|f(x)\|_{\infty}\in [0,+\infty],\]
and we set
\[\mathcal{F}_{R,J,\infty}=\{f\in \mathcal{F}_{R,J,2}:\|f\|_{R,\infty}<\infty\}.\]
For a tracial von Neumann algebra $(M,\tau)$ and $x\in M^{J},$ we then have that $f(x)\in M.$

The algebras $\mathcal{F}_{R,J,\infty}$,$ \mathcal{F}_{R,J,\infty}^{(s)}$ are both examples of algebras which are completions of a $C^{*}$-algebra with respect to uniform 2-norm coming from a family of traces. These are now known as \emph{uniformly tracially complete $C^{*}$-algebras}. Ozawa defined such a completion when the family consisted of \emph{all} traces (see \cite[p. 351-352]{Ozawa2013}), the special case of convex subsets of the trace space appeared recently in the study of classification of nuclear $C^{*}$-algebras and their homomorphisms (see \cite{BBSTWW2019, CETWW2019, CETW2020}).

The following is proved exactly as in \cite[Lemma 2.3]{FreePinsker}.

\begin{prop}\label{prop: its a C*-alg}
Let $J$ be an index set and $R\in [0,\infty)^{J}.$ Then the  product and   $*$-operation have a unique extension to product and $*$-operations on $\mathcal{F}_{R,J,\infty}$ which satisfy the axioms of a $*$-algebra as well as the following estimates
 \[\|f\|_{R,\infty}=\|f^{*}\|_{R,\infty},\,\,\ \|f\|_{R,2}=\|f^{*}\|_{R,2}\]
 \[\|fg\|_{R,\infty}\leq \|f\|_{R,\infty}\|g\|_{R,\infty},\,\,\,\, \|fg\|_{R,2}\leq \|f\|_{R,\infty}\|g\|_{R,2}.\]
Under these extended operations and the norm $\|\cdot\|_{R,\infty},$ the $*$-algebra $\mathcal{F}_{R,J,\infty}$ is a $C^{*}$-algebra.
\end{prop}

We now turn to the other two main properties of $\mathcal{F}_{R,J,\infty}$ we will want. If $(M_{k},\tau_{k}),k=1,2$ are two tracial von Neumann algebras and $\Theta\colon M_{1}\to M_{2}$ is a trace-preserving $*$-homomorphism, then $\|\Theta(x)\|_{2}=\|x\|_{2}$ for all $x\in M_{k}$. It thus follows that $\Theta$ extends uniquely to an isometry $L^{2}(M_{1},\tau_{1})\to L^{2}(M_{2},\tau_{2})$ which we still denote by $\Theta$.
\begin{thm}\label{thm:nc func calc properties}
Let $J$ be an index set and $R\in [0,\infty)^{J}.$ We then have the following properties of the noncommutative function space $\mathcal{F}_{R,J,\infty}.$
\begin{enumerate}[(i)]
\item Let $(M,\tau)$ be a tracial von Neumann  algebra and $x\in M^{J}$ with $\|x\|_{\infty}\leq R_{j}$ for all $j\in J.$ Then the map $\mathcal{F}_{R,J,\infty}\to W^{*}(x)$ given by $f\mapsto f(x)$ is surjective. In fact, for all $a\in W^{*}(x),$ there is an $f\in \mathcal{F}_{R,J,\infty}$ with $\|f\|_{R,\infty}\leq \|a\|_{\infty}$ and so that $f(x)=a.$ \label{item:surjective nc func calc}
\item Every $f\in \mathcal{F}_{R,J,2}$ is $\|\cdot\|_{2}$-uniformly continuous in the following sense. For every $\varespilon>0,$ there is a $\delta>0$ and a finite $F\subseteq J$ so that if $(M,\tau)$ is any tracial von Neumann algebra and $x,y\in M^{J}$ with $\|x_{j}-y_{j}\|_{2}<\delta$ for all $j\in F,$ we have $\|f(x)-f(y)\|_{2}<\varepsilon.$ \label{item:unif cont nc func calc}
\item \label{item:naturality of nc func calc}
Suppose $(M_{k},\tau_{k}),k=1,2$ are tracial von Neumann algebras and $x\in \prod_{j\in J}\{a\in M_{1}:\|a\|_{\infty}\leq R_{j}\},$ and that $\Theta\colon M_{1}\to M_{2}$ is a trace-preserving, unital, normal $*$-homomorphism. Then $f(\Theta(x))=\Theta(f(x))$ for all $f\in \mathcal{F}_{R,J,2}.$

\end{enumerate}

\end{thm}

\begin{proof}
Let $K=J\times \{0,1\},$ and define self-adjoint elements of $\C^{*}\ip{(T_{j})_{j\in J}}$ indexed by $K$ as follows:
\[A_{(j,0)}=\frac{T_{j}+T_{j}^{*}}{2},\,\,\, A_{(j,1)}=\frac{T_{j}-T_{j}^{*}}{2i},\]
for all $j\in J.$
Let $\pi\colon \C\ip{(S_{k})_{k\in K}}\to \C^{*}\ip{(T_{j})_{j\in j}}$ be the unique $*$-homomorphism satisfying $\pi(S_{k})=A_{k}$ for all $k\in K.$ Then $\pi$ is surjective. Define a continuous map
$\Psi\colon \Sigma_{R,J}\to \Sigma^{(s)}_{R,K}$
by
\[\Psi(\ell)(P)=\ell(\pi(P)),\]
and let $\widehat{\Psi}\colon C(\Sigma_{R,K}^{(s)})\to C(\Sigma_{R,J})$ be the induced map defined by $\widehat{\Psi}(\phi)=\phi\circ \Psi.$

Suppose $(M,\tau)$ is any tracial von Neumann algebra and $x\in M^{J}$ satisfies $\|x\|_{\infty}\leq R_{j}$ for all $j\in J.$ Define $y\in M^{K}$ by
\[y_{(j,0)}=\frac{x_{j}+x_{j}^{*}}{2},\,\,\, y_{(j,1)}=\frac{x_{j}-x_{j}^{*}}{2i}\]
for all $j\in J.$ Direction calculations show that for all $f\in \mathcal{A}_{R,K}^{(s)}$ we have $\ev_{y}(f)=\ev_{x}[(\widehat{\Psi}\otimes \pi)(f)].$ Indeed, since $\ev_{y},\ev_{x},\widehat{\Psi},$ and $\pi$ are all $*$-homomorphisms, it suffices to check this equation on an element of the form $\phi\otimes S_{k}$ for some $k\in K.$ In this case, the desired equality follows from the fact that $\ev_{y}(S_{k})=y_{k}$, $\ev_{x}(\pi(S_{k}))=y_{k},$ and $\Psi(\ell_{x})=\ell_{y}.$
It follows that for all $f\in \mathcal{A}_{R,K}^{(s)}$ we have
\[\|[\widehat{\Psi}\otimes \pi](f)\|_{R,2}\leq \|f\|_{R,2},\,\,\,\,\|[\widehat{\Psi}\otimes \pi](f)\|_{R,\infty}\leq \|f\|_{R,\infty}.\]
From the above two inequalities it follows that $\widehat{\Psi}\otimes \pi$ uniquely extends to maps, still denoted $\widehat{\Psi}\otimes \pi$, from $\mathcal{F}_{R,K,2}^{(s)}\to \mathcal{F}_{R,J,2}$, $\mathcal{F}_{R,K,\infty}^{(s)}\to \mathcal{F}_{R,J,\infty},$ which are  $\|\cdot\|_{R,2}$--$\|\cdot\|_{R,2},$ $\|\cdot\|_{R,\infty}$--$\|\cdot\|_{R,\infty}$ contractions. Moreover, we still have that $\ev_{x}\circ \widehat{\Psi}\otimes \pi=\ev_{y}.$

(\ref{item:surjective nc func calc}): Given $x\in M^{J},$ let $y\in M_{s.a.}^{K}$ be defined as above. Then by \cite[Proposition 2.4]{FreePinsker} for any $a\in W^{*}(x)=W^{*}(y),$ there is an $g\in \mathcal{F}_{R,K,\infty}^{(s)}$ with $g(y)=a$ and $\|g\|_{R,\infty}\leq \|a\|_{\infty}.$ Set $f=[\widehat{\Psi}\otimes \pi](g),$ then as $\ev_{x}\circ \widehat{\Psi}\otimes \pi=\ev_{y}$ we know $f(x)=a.$ Since $\widehat{\Psi}\otimes \pi$ is $\|\cdot\|_{R,\infty}$--$\|\cdot\|_{R,\infty}$ contractive, it follows that $\|f\|_{R,\infty}\leq\|g\|_{R,\infty}\leq  \|a\|_{\infty}.$

(\ref{item:unif cont nc func calc}):
Let $V$ the set of all $f\in \mathcal{F}_{R,J,2}$ which satisfy the conclusion of (\ref{item:unif cont nc func calc}). As elements of $\mathcal{F}_{R,K,2}^{(s)}$ are uniformly continuous, it follows that $V$ contains $[\widehat{\Psi}\otimes \pi](\mathcal{F}_{R,K,2}^{(s)}).$ In particular, $V$ is dense. It then suffices to show that $V$ is $\|\cdot\|_{R,2}$--closed. For  every $f\in \mathcal{F}_{R,2},$ and every tracial von Neumann algebra $(M,\tau)$ and all $x\in \prod_{j\in J}\{y\in M:\|y\|_{\infty}\leq R_{j}\}$ we have
\[\|f(x)\|_{2}\leq \|f\|_{R,2}.\]
From the above estimate, it is a standard argument to show that $V$ is closed.

(\ref{item:naturality of nc func calc}): First, observe that because $\Theta$ is trace-preserving, we know that $\ell_{\Theta(x)}=\ell_{x}.$ From here, the conclusion is direct to check for the case that $f\in \mathcal{A}_{R,J}.$ For $f\in \mathcal{F}_{R,J,2}$ we have the estimate
\[\max(\|f(\Theta(x))\|_{2},\|f(x)\|_{2})\leq \|f\|_{R,2}.\]
The above estimate allows to deduce the conclusion for a general element of $\mathcal{F}_{R,J,2}$ from the case of elements of $\mathcal{A}_{R,J}$ by approximation.

\end{proof}

Given an index set $J,$ and an $R\in [0,\infty)^{J},$  for any tracial von Neumann algebra $(M,\tau),$ we may abuse notation and view $f$ as a map
\[f\colon \prod_{j\in J}\{a\in M:\|a\|_{\infty}\leq R_{j}\}\to M\]
via $x\mapsto f(x).$ Given another index set $J'$ and $R'\in [0,\infty)^{J'},$ we define
\[\mathcal{F}_{R,R',J,J'}=\{f=(f_{j'})_{j'\in J'}\in (\mathcal{F}_{R,J,\infty})^{J'}:\|f_{j'}\|_{R,\infty}\leq R_{j'}'\mbox{ for all $j'\in J$}\}.\]
Then $f$ also determines a map
\[f\colon \prod_{j\in J}\{a\in M:\|a\|_{\infty}\leq R_{j}\}\to \prod_{j'\in J'}\{a\in M:\|a\|_{\infty}\leq R_{j'}'\mbox{ for all $j'\in J'$}\}\]
by $x\mapsto (f_{j'}(x))_{j'\in J'}.$ In particular, all of this makes sense for $M=M_{k}(\C).$ Given a $\mu\in \Prob(M_{k}(\C)^{J}),$ with
\[\mu\left(\prod_{j\in J}\{a\in M_{k}(\C):\|a\|_{\infty}\leq R_{j}\}\right)=1,\]
we slightly abuse notation and use $f_{*}(\mu)$ for the measure on $\prod_{j'\in J'}\{a\in M_{k}(\C):\|a\|_{\infty}\leq R_{j'}'\mbox{ for all $j'\in J'$}\}$ which is the pushforward of $\mu$ under the map $x\mapsto (f_{j'}(x))_{j'\in J'}$.
A nice consequence of $\|\cdot\|_{2}$-uniform continuity is that taking pushforwards of measures preserves exponential concentration.

\begin{prop}\label{prop:nc pushf preserves conc}
Let $J$,$J'$  be countable index set, and $R\in [0,\infty)^{J}$, $R'\in [0,\infty)^{J'},$ and $f\in \mathcal{F}_{R,R',J,J'}.$ Suppose we are given a sequence $\mu^{(k)}\in \Prob(M_{n(k)}(\C)^{J})$ with
\[\mu^{(k)}\left(\prod_{j\in J}\{A\in M_{n(k)}(\C):\|A\|_{\infty}\leq R_{j}\}\right)=1.\]
If $\mu^{(k)}$ has exponential concentration at scale $n(k)^{2},$ then so does $f_{*}\mu^{(k)}.$
\end{prop}

\begin{proof}
Fix a finite $F'\subseteq J',$ and an $\varespilon>0.$ By Theorem \ref{thm:nc func calc properties} (\ref{item:unif cont nc func calc}), we may choose a finite $F\subseteq J$ and a $\delta>0$ so that if $(M,\tau)$ is any tracial von Neumann algebra, and $x,y\in M^{J}$ satisfy $\|x_{j}\|_{\infty}$,$\|y_{j}\|_{\infty}\leq R_{j}$ for all $j\in J$ and $\|x_{j}-y_{j}\|_{2}<\delta$ for all $j\in F,$ then $\|f(x)-f(y)\|_{2,F'}<\varepsilon.$

By Theorem \ref{thm:nc func calc properties} (\ref{item:naturality of nc func calc}) for every $g\in \mathcal{F}_{R,J,\infty}$, every tracial von Neumann algebra $(M,\tau)$, every $u\in \mathcal{U}(M),$ and every $x\in \prod_{j\in J}\{a\in M:\|a\|_{\infty}\leq R_{j}\}$ we have $g(uxu^{*})=ug(x)u^{*}.$ Hence, for all $k\in \N,$ and all $A,B\in \prod_{j\in J}\{C\in M_{k}(\C):\|C\|_{\infty}\leq R_{j}\}$ with $d^{\orb{}}_{F}(A,B)<\delta,$ we have $d^{\orb{}}_{F'}(f(A),f(B))<\varespilon.$

So suppose $\Omega\subseteq M_{n(k)}(\C)^{J'}$ and $f_{*}\mu^{(k)}(\Omega)\geq 1/2.$ Then $\mu^{(k)}(f^{-1}(\Omega))\geq 1/2,$ so \[\mu^{(k)}(N_{\delta}(f^{-1}(\Omega),d^{\orb{}}_{F}))\geq 1-\alpha_{\mu^{(k)},d^{\orb{}}_{F}}(\delta).\]
Our choice of $\delta$ implies that $N_{\delta}(f^{-1}(\Omega),d^{\orb{}}_{F})\subseteq f^{-1}(N_{\varepsilon}(\Omega,d^{\orb{}}_{F'})),$ and thus
\[f_{*}\mu^{(k)}(N_{\varepsilon}(\Omega,d^{\orb{}}_{F'}))\geq 1-\alpha_{\mu^{(k)},d^{\orb{}}_{F}}(\delta).\] So
\[\limsup_{k\to\infty}\frac{1}{n(k)^{2}}\log \alpha_{f_{*}\mu^{(k)},d^{\orb{}}_{F'}}(\varepsilon)\leq \limsup_{k\to\infty}\frac{1}{n(k)^{2}}\log \alpha_{\mu^{(k)},d^{\orb{}}_{F}}(\delta)<0.\]

\end{proof}

We also need the following analogues of \cite[Propostion 2.6 (1) and (2)]{FreePinsker}, whose proofs are identical.

\begin{lem}
Let $J,J'$ be index sets and $R\in [0,\infty)^{J},$ $R'\in [0,\infty)^{J'}.$ Suppose $f\in \mathcal{F}_{R,R',J,J'}.$
\begin{enumerate}[(i)]
\item Suppose $(M_{k},\tau_{k}),k=1,2$ are two tracial von Neumann algebras, and $x_{k}\in \prod_{j\in J}\{a\in M_{k}:\|a\|_{\infty}\leq R_{j}\},k=1,2.$ If $\ell_{x_{1}}=\ell_{x_{2}},$ then $\ell_{f(x_{1})}=\ell_{f(x_{2})}.$ \label{I:well defined law}
\item Define a map $f_{*}\colon \Sigma_{R,J}\to \Sigma_{R',J'}$ as follows. Given $\ell\in \Sigma_{R,J},$ let $\pi_{\ell}$, $W^{*}(\ell)$ be as in (\ref{eqn:GNS rep}),(\ref{eqn:vNa of a law}) and equip $W^{*}(\ell)$ with the trace $\tau_{\ell}$ given by (\ref{eqn:GNS trace}). Set $x=(\pi_{\ell}(T_{j}))_{j\in J},$ and define $f_{*}\ell=\ell_{f(x)}.$ Then $f_{*}$ is weak$^{*}$-weak$^{*}$ continuous.
\label{I:weak* continuity of nc calc}
\end{enumerate}
\end{lem}
Recall that the point of the construction in (\ref{eqn:GNS rep}) was that $\ell=\ell_{x}.$
So by (\ref{I:well defined law}), for any tracial von Neumann algebra $(M,\tau),$ and any $y\in M^{J}$ with $\ell_{y}=\ell,$ we have $f_{*}\ell=\ell_{f(y)}.$ So, for example, if we have a sequence $(x_{n})_{n}$ of tuples in tracial von Neumann algebras and if $\ell_{x_{n}}\to^{\textnormal{weak}^{*}}_{n\to\infty} \ell_{x}$ for some other tuple $x,$ then by (\ref{I:weak* continuity of nc calc}) we know $\ell_{f(x_{n})}\to^{\textnormal{weak}^{*}}_{n\to\infty} \ell_{f(x)}$ provide $x_{n},x$ satisfy the appropriate norm bounds to define $f(x_{n}),f(x).$

We also need a simple consequence of the above, which is that microstates behave well with respect to the noncommutative function spaces $\mathcal{F}_{R,R',J,J'}.$ The proof is the same as in \cite[Corollary 2.7]{FreePinsker}.

\begin{lem}
Let $J,J'$ be index sets, and $R\in [0,\infty)^{J}$ and $R'\in [0,\infty)^{J'}.$ Fix an $f\in \mathcal{F}_{R,R',J,J'},$ a tracial von Neumann algebra $(M,\tau)$ and $x\in \prod_{j\in J}\{a\in M:\|a\|_{\infty}\leq R_{j}\}.$ Then for any weak$^{*}$-neighborhood $\mathcal{V}$ of $\ell_{f(x),x}$ in $\Sigma_{R'\sqcup R,J'\sqcup J}$ there is a weak$^{*}$-neighborhood $\mathcal{O}$ of $\ell_{x}$ so that
\[f(\Gamma_{R}^{(k)}(\mathcal{O}))\subseteq \Gamma_{R'\sqcup R}^{(k)}(f(x):\mathcal{V}).\]

\end{lem}

\section{Proofs of the main theorems} \label{S:main theorems}

\subsection{Microstates Collapse and the proof of Theorem \ref{T:more general main theorem intro} (\ref{item:microstates collapse intro})} \label{subsec:microstates collapse}
 Intuitively, Theorem \ref{T:more general main theorem intro} (\ref{item:microstates collapse intro}) asserts that if we sample microstates for $M$ according to the sequence of measures $\mu^{(k)},$ and use them to ``induce" (via the function $f$) microstates for $N,$ then ``most" of these microstates for $N$ are unitarily conjugate to each other. This will be proved in a manner entirely similar to the proof of \cite[Proposition 3.3]{FreePinsker}, with only minor changes in place to take care of the fact that we are dealing with unitary conjugation orbits of microstates instead of relative microstates as in \cite[Section 3.3]{FreePinsker}.

\begin{lem}\label{L:concentration expansion}
Let $(X,\mu,d)$ be a pseudometric measure space. If $\Omega\subseteq X$ and $\mu(\Omega)> \alpha_{\mu,d}(\varepsilon)$, then
\[\mu(N_{2\varepsilon}(\Omega,d))\geq 1-\alpha_{\mu,d}(\varespilon).\]
\end{lem}

\begin{proof}
Suppose $\mu(\Omega)> \alpha_{\mu,d}(\varepsilon),$ and set $\Theta=N_{\varepsilon}(\Omega,d)^{c}.$ Then
\[\mu(N_{\varepsilon}(\Theta)^{c})\geq \mu(\Omega)>\alpha_{\mu,d}(\varepsilon).\]
The definition of $\alpha$ implies that $\mu(\Theta)<1/2.$ So $\mu(N_{\varepsilon}(\Omega,d))>1/2,$ and this in turn implies that
\[\mu(N_{2\varepsilon}(\Omega,d))\geq 1-\alpha_{\mu,d}(\varepsilon).\]
\end{proof}

It will frequently be useful to note the following facts about sequences of measures which are asymptotically concentrated on microstates spaces and have exponential concentration.

\begin{lem}\label{lem: reduction lemma}
Let $(M,\tau)$ be a tracial von Neumann algebra, $I$ an index set, $R\in [0,\infty)^{I}$ and $x\in \prod_{i\in I}\{a\in M:\|a\|_{\infty}\leq R_{i}\}.$ Assume we are given integers $n(k)$ for $k\in \N$ with $n(k)\to\infty$ and $\mu^{(k)}\in \Prob(M_{k}(\C)^{J})$ with
\[\mu^{(k)}(\Gamma_{R}^{(n(k))}(\mathcal{O}))\to 1\]
for all weak$^{*}$ neighborhoods $\mathcal{O}$ of $\ell_{x}$ in $\Sigma_{R,I}$. Further assume that  $\mu^{(k)}$ has exponential concentration with scale $n(k)^{2}$.
\begin{enumerate}[(i)]
\item \label{item:fixing norm bounds}
 Assume that
\[\lim_{k\to\infty}\mu^{(k)}\left(\prod_{i\in I}\{A\in M_{n(k)}(\C):\|A\|_{\infty}\leq R_{i}\}\right)^{\frac{1}{n(k)}^{2}}= 1.\]
Define $\nu^{(k)}\in \Prob(M_{n(k)}(\C)^{I})$ by
\[\nu^{(k)}(E)=\frac{\mu^{(k)}\left(E\cap \prod_{i\in I}\{A\in M_{n(k)}(\C):\|A\|_{\infty}\leq R_{i}\}\right)}{\mu^{(k)}\left(\prod_{i\in I}\{A\in M_{n(k)}(\C):\|A\|_{\infty}\leq R_{i}\}\right)}.\]
Then $\nu^{(k)}$ still has exponential concentration with scale $n(k)^{2}$.

\item \label{item: auto exp decay}
Assume that
\[\lim_{k\to\infty}\mu^{(k)}\left(\prod_{i\in I}\{A\in M_{n(k)}(\C):\|A\|_{\infty}\leq R_{i}\}\right)= 1,\]
and define $\nu^{(k)}$ as in (\ref{item:fixing norm bounds}).
For every weak$^{*}$ neighborhood $\mathcal{O}$ of $\ell_{x}$ in $\Sigma_{R,I}$ we have
\[\limsup_{k\to\infty}\frac{1}{n(k)^{2}}\log \nu^{(k)}(\Gamma_{R}^{(n(k))}(\mathcal{O})^{c})<0.\]
\end{enumerate}

\end{lem}

\begin{proof}

(\ref{item:fixing norm bounds}): To see that $\nu^{(k)}$ still exhibits exponential concentration fix $\varespilon>0$, and suppose that $E_{k}\subseteq M_{n(k)}(\C)^{J}$ has $\nu^{(k)}(E_{k})\geq \frac{1}{2}$ for all $k.$ Then
\[\mu^{(k)}(E_{k})\geq \frac{1}{2}\mu^{(k)}\left(\prod_{i\in I}\{A\in M_{n(k)}(\C):\|A\|_{\infty}\leq R_{i}\}\right).\]
By our assumptions and Lemma \ref{L:concentration expansion}, we have
\[\mu^{(k)}(N_{2\varepsilon}(E_{k},d^{\orb{}}_{F})^{c})\leq \alpha_{\mu^{(k)},d^{\orb{}}_{F}}(\varepsilon)\]
for all large $k.$ But then for all large $k$ we have
\[\nu^{(k)}(N_{2\varepsilon}(E_{k},d^{\orb{}}_{F})^{c})\leq \frac{\alpha_{\mu^{(k)},d^{\orb{}}_{F}}(\varepsilon)}{\mu^{(k)}\left(\prod_{i\in I}\{A\in M_{n(k)}(\C):\|A\|_{\infty}\leq R_{i}\}\right)}.\]
Hence
\[\alpha_{\nu^{(k)},d^{\orb{}}_{F}}(2\varepsilon)\leq \frac{\alpha_{\mu^{(k)},d^{\orb{}}_{F}}(\varepsilon)}{\mu^{(k)}\left(\prod_{i\in I}\{A\in M_{n(k)}(\C):\|A\|_{\infty}\leq R_{i}\}\right)}\]
for all large $k.$
This estimate and our hypotheses on $\mu^{(k)}\left(\prod_{i\in I}\{A\in M_{n(k)}(\C):\|A\|_{\infty}\leq R_{i}\}\right)$ are enough to show that $\nu^{(k)}$ still has exponential concentration with scale $n(k)^{2}$.

(\ref{item: auto exp decay}): We may choose a weak$^{*}$-neighborhood $\mathcal{V}$ of the law of $x,$ a $\delta>0$ and a finite subset $F\subseteq I$ so that
\[N_{\delta}(\Gamma_{R}^{n(k)}(\mathcal{V}),\|\cdot\|_{2,F})\cap \prod_{i\in I}\{A\in M_{n(k)}(\C):\|A\|_{\infty}\leq R_{i}\}\subseteq \Gamma_{R}^{n(k)}(\mathcal{O})\]
for all $k\in \N.$ Since $\Gamma_{R}^{n(k)}(\mathcal{V})$ is conjugation invariant, it follows that
\[N_{\delta}(\Gamma_{R}^{n(k)}(\mathcal{V}),d^{\orb{}}_{F})\cap \prod_{i\in I}\{A\in M_{n(k)}(\C):\|A\|_{\infty}\leq R_{i}\}\subseteq \Gamma_{R}^{n(k)}(\mathcal{O}).\]
Our assumptions on $\nu^{(k)}$ guarantee that for all large $k$, we have  $\nu^{(k)}(\Gamma_{R}^{n(k)}(\mathcal{V}))\geq \frac{1}{2}.$ So
\[\nu^{(k)}(\Gamma_{R}^{n(k)}(\mathcal{O})^{c})\leq \alpha_{\nu^{(k)},d^{\orb}_{F}}(\delta)\]
for all large $k.$ Taking $\frac{1}{n(k)^{2}}\log$ of both sides and letting $k\to\infty$ completes the proof, by
(\ref{item:fixing norm bounds})

\end{proof}

We will deduce Theorem \ref{T:more general main theorem intro} (\ref{item:microstates collapse intro}), as a consequence of the following more general result.

\begin{thm}\label{thm:microstates collapse generalization}
Let $(M,\tau)$ be a tracial von Neumann algebra and $N\leq M$ with $h(N:M)=0.$ Fix index sets $I,J$ with $J$ countable, and let $x\in M^{I}$, $y\in N^{J}$ be given.  Suppose that $\widehat{R}\in [0,\infty)^{I\sqcup J}$ with $\|x_{i}\|_{\infty}\leq \widehat{R}_{i}$ for all $i\in I,$ and $\|y_{j}\|_{\infty}\leq \widehat{R}_{j}$ for all $j\in J$ and so that $M=W^{*}(x).$ Set $R=\widehat{R}\big|_{I},$ $R'=\widehat{R}\big|_{J}.$ Write $y=f(x)$ for some $f\in \mathcal{F}_{R,R',I,J}.$

Assume that $n(k)\in \N$ is a sequence of integers with $n(k)\to \infty,$ and that $\mu^{(k)}\in \Prob(M_{n(k)}(\C)^{I})$ satisfies
\[\mu^{(k)}(\Gamma_{R}^{(n(k))}(\mathcal{O}))\to_{k\to\infty}1\]
for every weak$^{*}$-neighborhood $\mathcal{O}$ of $\ell_{x}$ in $\Sigma_{R,I}$. Further assume that
\[\lim_{k\to\infty}\mu^{(k)}\left(\prod_{i\in I}\{A\in M_{k}(\C):\|A\|_{\infty}\leq R_{i}\}\right)=1.\]
If $(\mu^{(k)})_{k}$ has exponential concentration at scale $n(k)^{2},$ then there is a sequence $\Omega_{k}\subseteq \prod_{i\in I}\{C\in M_{k}(\C):\|C\|_{\infty}\leq R_{i}\}$ with the following properties:
\begin{itemize}
    \item $\mu^{(k)}(\Omega_{k})\to 1,$
    \item  for every weak$^{*}$-neighborhood $\mathcal{O}$ of $\ell_{x}$ we have $\Omega_{k}\subseteq\Gamma_{R}^{(n(k))}(\mathcal{O})$ for all sufficiently large $k,$
    \item for every finite $F\subseteq J$
    \[\lim_{k\to\infty}\sup_{A_{1},A_{2}\in \Omega_{k}}d^{\orb{}}_{F}(f(A_{1}),f(A_{2}))=0.\]
    \end{itemize}

\end{thm}

\begin{proof}

By Lemma \ref{lem: reduction lemma} we may, and will, assume that
\[\mu^{(k)}\left(\prod_{i\in I}\{A\in M_{n(k)}(\C):\|A\|_{\infty}\leq R_{i
}\}\right)=1.\]

We first start with the following claim.

\emph{Claim: For every finite $F\subseteq J,$ for every $\varepsilon>0,$ and for every weak$^{*}$-neighborhood $\mathcal{O}$ of $\ell_{x},$ there is a sequence $\Omega_{k}\subseteq \Gamma_{R}^{(n(k))}(\mathcal{O})$ (depending upon $\varepsilon,F,\mathcal{O}$) satisfying}
\begin{itemize}
    \item $\lim_{k\to\infty}\mu^{(k)}(\Omega_{k})=1,$ and
    \item $\limsup_{k\to\infty}\sup_{A_{1},A_{2}\in \Omega_{k}}d^{\orb{}}_{F}(f(A_{1}),f(A_{2}))\leq \varepsilon.$
\end{itemize}

To prove the claim, let $\nu^{(k)}=f_{*}\mu^{(k)}$ as defined in the discussion preceding Proposition \ref{prop:nc pushf preserves conc}. Set
\[\eta=-\limsup_{k\to\infty}\frac{1}{n(k)^{2}}\alpha_{\nu^{(k)},d^{\orb{}}_{F}}(2\varepsilon).\]
By Proposition \ref{prop:nc pushf preserves conc}, we know $\eta>0.$
Since $h(N:M)\leq 0,$ we may choose a weak$^{*}$-neighborhood $\mathcal{V}$ of $\ell_{(y,x)}$ so that
\[ K_{\varepsilon,F}^{\orb{}}(y:\mathcal{V},\|\cdot\|_{2})\leq \frac{\eta}{8}.\]
Let $\Xi_{k}\subseteq\Gamma_{\widehat{R}}^{(n(k))}(y:\mathcal{V})$ be $\varepsilon$-dense with respect to $d^{\textnormal{orb}}_{F}$ and so that
\[|\Xi_{k}|=K_{\varepsilon}(\Gamma_{\widehat{R}}^{(n(k))}(y:\mathcal{V}),d^{\textnormal{orb}}_{F}).\]
Let $\phi_{k}\colon \Gamma_{\widehat{R}}^{(n(k))}(y:\mathcal{V})\to \Xi_{k}$ be Borel maps which satisfy
\[d^{\orb{}}_{F}(A,\phi_{k}(A))<\varepsilon\mbox{ for all $A\in \Gamma_{\widehat{R}}^{(n(k))}(y:\mathcal{V})$.}\]
Set
\[\Theta_{k}=\{A\in \Gamma_{\widehat{R}}^{(n(k))}(y:\mathcal{V}):\nu^{(k)}(N_{2\varepsilon}(A,d^{\orb{}}_{F}))\geq \exp(-n(k)^{2}\eta/2)\},\]
and $\Delta_{k}= \Gamma^{(n(k))}(y;\mathcal{V})\setminus \Theta_{k}.$ Observe that for every $A\in \Delta_{k},$ we have $\nu^{(k)}(N_{\varepsilon}(\phi_{k}(A),d^{\orb{}}_{F}))<\exp(-n(k)^{2}\eta^{2}/2).$ So
\[\nu^{(k)}(\Delta_{k})\leq \sum_{B\in \phi_{k}(\Delta_{k})}\nu^{(k)}(N_{\varepsilon}(B,d^{\orb{}}_{F}))<\exp(-n(k)^{2}\eta/2)|\Xi_{k}|.\]
Thus
$\nu^{(k)}(\Delta_{k})\leq\exp(-n(k)^{2}/4)$ for all large $k$,
and so $\nu^{(k)}(\Delta_{k})\to 0.$ So
\[\nu^{(k)}(\Theta_{k})=\nu^{(k)}(\Gamma_{R}^{(n(k))}(y:\mathcal{O}))-\nu^{(k)}(\Delta_{k})\to 1,\]
as $\nu^{(k)}$ is asymptotically supported on the microstates space for $y$ in the presence of $x.$
 Suppose $B_{1},B_{2}\in \Theta_{k}.$ If $k$ is sufficiently large, then by Lemma \ref{L:concentration expansion}
\[N_{4\varpesilon}(B_{1},d^{\orb{}}_{F})\cap N_{4\varepsilon}(B_{2},d^{\orb{}}_{F})\ne \varnothing,\]
and thus $d^{\orb{}}_{F}(B_{1},B_{2})\leq 8\varepsilon.$ By definition of $\nu^{(k)},$ we have
$\mu^{(k)}(f^{-1}(\Theta_{k}))=\nu^{(k)}(\Theta_{k})\to 1.$
So if we set $\Omega_{k}=f^{-1}(\Theta_{k})\cap \Gamma_{R}^{(n(k))}(\mathcal{O}),$ then it is direct to show that $\Omega_{k}$ has the desired properties with $\varepsilon$ replaced by $8\varepsilon.$ Since $\varepsilon>0$ is arbitrary, this proves the claim.

To prove the theorem, let $(F_{m})_{m}$ be an increasing sequence of finite subsets of $J$ with $J=\bigcup_{m}F_{m},$ and let $\mathcal{O}_{m}$ be a decreasing sequence of weak$^{*}$-neighborhoods of $\ell_{x}$ in $\Sigma_{R,J}$ with $\bigcap_{m}\mathcal{O}_{m}=\{\ell_{x}\}.$ By the claim, for every positive integer $m,$ we may choose a sequence $\Omega_{k,m}\subseteq \Gamma_{\widehat{R}}^{(n(k))}(y:\mathcal{O}_{m})$ with
\[\limsup_{k\to\infty}\sup_{A_{1},A_{2}\in \Omega_{k,m}}d^{\textnormal{orb}}_{F_{m}}(f(A_{1}),f(A_{2}))<2^{-m},\]
\[\lim_{k\to\infty}\mu^{(k)}(\Omega_{k,m})=1.\]
We may thus find a strictly increasing sequence $1<K_{1}<K_{2}<\cdots$ of integers so that for every positive integer $m$
\[\sup_{\substack{k\geq K_{m},\\ A_{1},A_{2}\in \Omega_{k,m}}}d^{\textnormal{orb}}_{F_{m}}(f(A_{1}),f(A_{2}))<2^{-m},\mbox{   and }\inf_{k\geq K_{m}}\mu^{(k)}(\Omega_{k,m})\geq 1-2^{-m}.\]
Define $\Omega_{k}$ as follows. For $k<K_{1},$ set $\Omega_{k}=\varnothing,$ and for $k\geq K_{1}$ let $m$ the unique integer so that $K_{m}\leq k<K_{m+1}$ and set $\Omega_{k}=\Omega_{k,m}.$ It is then direct to verify that $\Omega_{k}$ has the desired properties.

\end{proof}

This recovers  Theorem \ref{T:more general main theorem intro} (\ref{item:microstates collapse intro}) as follows.

\begin{proof}[Proof of  Theorem \ref{T:more general main theorem intro} (\ref{item:microstates collapse intro}) from Theorem \ref{thm:microstates collapse generalization}]
Let $\mu^{(k)}$ be the distribution of $X^{(k)}$, and fix $Q\leq M$ with $h(Q:M)\leq 0$. Let $R_{j},y,f$ be as in the statement of Theorem \ref{T:more general main theorem intro}
and set $R=(R_{j})_{j=1}^{n}$. The statement that $\ell_{X^{(k)}}\to \ell_{x}$ in probability implies that for every weak$^{*}$-neighborhood $\mathcal{O}$ of $\ell_{x}$ we have $\mu^{(k)}(\Gamma_{R}^{(k)}(\mathcal{O}))\to_{k\to\infty}1$. Additionally, the assumptions of Theorem \ref{T:more general main theorem intro} imply that
\[\lim_{k\to\infty}\mu^{(k)}\left(\prod_{j=1}^{n}\{A\in M_{n}(\C):\|A\|_{\infty}\leq R_{j}\}\right)=1,\]
and that $\mu^{(k)}$ has exponential concentration at scale $n(k)^{2}$. Let $\Omega_{k}$ be as in the conclusion to Theorem \ref{thm:microstates collapse generalization}. Since $\mu^{(k)}(\Omega_{k})\to 1$, for all large $k$ we can find  $A^{(k)}\in \Omega_{k}$. Then $\ell_{A^{(k)}}\to \ell_{x}$ in law, since for every weak$^{*}$-neighborhood $\mathcal{O}$ of $\ell_{x}$ we have $\Omega_{k}\subseteq \Gamma_{R}^{(k)}(\mathcal{O})$ for all large $k$.

If $\varepsilon>0$, then we may find a $K$ so that for all $k\geq K$ we have
\[\sup_{B\in \Omega_{k}}d^{\orb{}}(f(B),f(A^{(k)}))<\varepsilon.\]
So for all $k\geq K$
\[\P(d^{\orb{}}(f(X^{(k)}),f(A^{(k)}))<\varespilon)=\mu^{(k)}(\{X:d^{\orb{}}(f(X^{(k)}),f(A^{(k)}))<\varespilon\})\geq \mu^{(k)}(\Omega_{k})\to_{k\to\infty}1.\]
Thus $d^{\orb{}}(f(X^{(k)}),f(A^{(k)}))\to 0$ in probability.

\end{proof}

As we remarked in the introduction, we will see in Section \ref{sec: Jung} (see Theorem \ref{thm:as conjugation}) that one can use an ultraproduct framework to reformulate the above result in terms of a ``random Jung theorem."

\subsection{Proof of Theorem \ref{T:more general main theorem intro} (\ref{item:Pinskers are amenable intro})}\label{sub sec: proof of Pinsker are amenable intro}

The Peterson-Thom conjecture is inherently a question about von Neumann algebras, whereas strong convergence of laws is inherently a question about $C^{*}$-algebras. For example, strong convergence can be reformulated in terms of trace-preserving embeddings into \emph{$C^{*}$-ultraproducts}. So a significant aspect of Theorem \ref{T:more general main theorem intro} is the assertion that we can reduce the \emph{von Neumann question} of validity of the Peterson-Thom conjecture to a \emph{$C^{*}$-question} about  strong convergence. In order to do this, we will need to assume that  our given von Neumann algebra  can be approximated by any ``nice enough" weak$^{*}$-dense $*$-subalgebra in a manner robust enough  to preserve some key structure of the von Neumann algebra. In particular, we will make use of the Connes-Haagerup characterization  of nonamenability of a von Neumann algebra in terms of norms of ``Laplace-like" operators in the tensor of the algebra with its opposite. Thus, we will
need to assume that our approximation process keeps norms under control when we pass to tensor products. Maintaining control over norms when passing to tensor products is the raison d'\^{e}tre for the notions of completely bounded/completely positive maps.
So the above discussion naturally leads us to the consideration of approximation properties formulated via completely positive and completely bounded maps.

Given a  $C^{*}$-algebra $A$,  there is a canonical way to view $A^{**}$ as a von Neumann algebra. Moreover   the natural inclusion $A\hookrightarrow A^{**}$ allows us to view $A^{**}$ as the \emph{universal enveloping von Neumann algebra of $A$}.  Namely, given any $*$-representation $\pi\colon A\to B(\mathcal{H})$ with $\mathcal{H}$ a Hilbert space, there is a unique, normal $*$-representation $\widetilde{\pi}\colon A^{**}\to B(\mathcal{H})$ with $\widetilde{\pi}\big|_{A}=\pi$. Moreover, $\widetilde{\pi}(A)=\overline{\pi(A)}^{SOT}$ (see \cite[Theorem 2.4]{Taka} for a proof of all of this).

\begin{defn}\label{D:locally reflexive}
We say that a (unital) $C^{*}$-algebra $A$ is \emph{locally reflexive} if given any finite-dimensional operator system $E\subseteq A^{**},$ there is a net $\phi_{\alpha}\colon E\to A$ of completely positive maps with $\|\phi_{\alpha}\|_{cb}\leq 1$ and so that $\phi_{\alpha}(x)\to_{\alpha} x$ in the weak$^{*}$-topology.
\end{defn}
An alternate way to phrase this is as follows. Let $E,F$ be operator systems. If $F$ is an operator system concretely embedded in $B(\mathcal{H})$ with $\mathcal{H}$ a Hilbert space, then we can give $CP(E,F)$ the \emph{point-WOT} topology. So a basic neighborhood of $\phi\in CP(E,F)$ is given by
\[\mathcal{O}_{G_{1},G_{2},\varepsilon}(\phi)=\bigcap_{x\in G_{1},\xi,\eta\in G_{2}}\{\psi\in CP(E,F):|\ip{\phi(x)\xi,\eta}-\ip{\psi(x)\xi,\eta}|<\varepsilon\}\]
for finite sets $G_{1}\subseteq E$, $G_{2}\subseteq \mathcal{H},$ and an $\varepsilon\in (0,\infty).$
  Let $A$ be a $C^{*}$-algebra. For an operator space $E\subseteq A^{**},$ we use $\iota_{E}$ for the inclusion map $E\hookrightarrow A^{**}.$
Locally reflexivity is then just the assertion that
\[\iota_{E}\in \overline{\{\phi\in CP(E,A^{**}):\phi(E)\subseteq A,\|\phi\|_{cb}\leq 1\}}^{point-WOT},\]
for every finite dimensional $E\subseteq A^{**}.$ The main result on locally reflexivity that we need is that every exact $C^{*}$-algebra is locally reflexive (see \cite{KirchbergExact,KirchbergCAR} and also \cite[Theorem 9.3.1]{BO}). Since exact $C^{*}$-algebras are ubiquitous in free probability, this provides us with an adequate source of examples. E.g., the reduced free group $C^{*}$-algebra is locally reflexive, as is the $C^{*}$-algebra generated by a free semicircular family. Indeed, given any free tuple $(x_{1},\cdots,x_{k})\in M^{k}$  in a tracial von Neumann algebra $(M,\tau),$  with each $x_{j}$ being normal, we have that $C^{*}(x_{1},\cdots,x_{k})$ is exact by \cite{DykemaExact,DykemaDimaExact}.

It should be emphasized that $A^{**}$ is a very large von Neumann algebra. For example, it is only in very rare circumstances that $A^{*}$ is separable (e.g. this does not occur if $A$ contains a copy of $C(X)$ where $X$ is an uncountable compact Hausdorff space). Consequently, it is rare that $A^{**}$ can be represented on a separable Hilbert space. However, the fact that $A^{**}$ is the universal enveloping von Neumann algebra allows us to deduce more concrete approximations for other von Neumann algebras associated to $A.$ Recall that if $\mathcal{H}$ is a Hilbert space, and $M\subseteq B(\mathcal{H})$ is a von Neumann algebra, then $M$ is a \emph{von Neumann completion of $A$} if there is a faithful $*$-representation $\pi\colon A\to B(\mathcal{H})$ with $M=\overline{\pi(A)}^{SOT}.$ Suppose $A$ is locally reflexive and $M$ is a von Neumann completion of $A,$ and view $A\subseteq M.$ By universality of $A^{**}$ it follows that if $E\subseteq M$ is a finite-dimensional operator system, then
\[\iota_{E}\in \overline{\{\phi\in CP(E,M):\phi(E)\subseteq A,\|\phi\|_{cb}\leq 1\}}^{point-WOT},\]
where $\iota_{E}\colon E\to M$ is the inclusion map. This is the precise manner in which we shall use local reflexivity to approximate elements of $M$ by a prescribed weak$^{*}$-dense $*$-subalgebra.

We also need to recall some notation and a result of Haagerup.
 We have an action $\#$ of $M_{k}(\C)\otimes M_{k}(\C)$ on $S^{2}(k,\tr)$ defined on elementary tensors by
\[(A\otimes B)\#C=ACB^{t}.\]
It is direct to check that this gives a $*$-isomorphism
\[M_{k}(\C)\otimes M_{k}(\C)\cong B(S^{2}(k,\tr)).\]
Since $*$-isomorphisms between $C^{*}$-algebras are isometric, it follows that
\[\|x\|_{\infty}=\|x\#\|_{B(S^{2}(k,\tr))}\]
for all $x\in M_{k}(\C)\otimes M_{k}(\C).$
For a tracial von Neumann algebra $(M,\tau)$ and $x\in M,$ we let $M^{op}$ be the von Neumann algebra which as a set is $\{x^{op}:x\in M\}.$ The vector space operations and the $*$-operation is the same as in $M,$ but the product is the opposite:
\[x^{op}y^{op}=(yx)^{op}.\]
For $x\in M,$ we let $\overline{x}=(x^{*})^{op}.$ Note that we have a canonical identification $M_{k}(\C)\cong M_{k}(\C)^{op}$ given by $A\mapsto (A^{t})^{op}.$ For $A\in M_{k}(\C),$ we let $\overline{A}=(A^{*})^{t}.$ Technically, this means we have two different notions of $\overline{A}$ for $A\in M_{k}(\C).$ One as an element of $M_{k}(\C)$ and one as an element of $M_{k}(\C)^{op}.$ However, under the identification $M_{k}(\C)\cong M_{k}(\C)^{op}$ given above, these two notations coincide. Since we always identify $M_{k}(\C)$,$ M_{k}(\C)^{op}$ via the map $A\mapsto (A^{t})^{op},$ this will not cause confusion. The way we shall use nonamenability is in the following characterization of nonamenability of tracial von Neumann algebras, due to Haagerup.
\begin{thm}[Haagerup,  Lemma 2.2 in \cite{HaagerupAmenable}]
Let $(M,\tau)$ be a tracial von Neumann algebra. Then $M$ is nonamenable if and only if there is a nonzero central projection $f\in M$ and $u_{1},\cdots,u_{r}\in \mathcal{U}(Mf)$ so that
\[\left\|\frac{1}{r}\sum_{j=1}^{r}u_{j}\otimes\overline{u_{j}}\right\|_{\infty}<1.\]
\end{thm}

In order to prove Theorem \ref{T:more general main theorem intro} (\ref{item:Pinskers are amenable intro}), we need to reduce the  validity of the Peterson-Thom conjecture to the \emph{$C^{*}$-question} of strong convergence. We begin with the following Proposition, which gives a general result along these lines.
We remark that the argument for the proof of this Proposition is  analogous with a method of proof of Chifan-Sinclair (see \cite[Theorem 3.2]{ChifanSinclair}) in the context of Popa's deformation/rigidity theory.

\begin{prop}\label{prop:almost unitary conjugacy}
Let $(M,\tau)$ be a tracial von Neumann algebra, $I$ an index set, and $x\in M^{I}$ with $W^{*}(x)=M.$ Fix $R\in [0,\infty)^{I}$ with $\|x_{i}\|_{\infty}\leq R_{i}$ for all $i\in I.$  Suppose that $C^{*}(x)$ is locally reflexive and that $Q\leq M$ is nonamenable.  Then there  is an $r\in \N,$ an $F\in (\mathcal{F}_{R,I,\infty})^{r}$ with $F(x)\in Q^{r}$ and an $\varepsilon>0$ which satisfies the following property. Assume we are given
\begin{itemize}
    \item positive integers $(n(k))_{k=1}^{\infty}$ with $n(k)\to \infty,$
    \item $A^{(k)},B^{(k)}\in \prod_{i\in I}\{C\in M_{n(k)}(\C):\|C\|_{\infty}\leq R_{i}\}$
\end{itemize}
such that the law of $(A^{(k)}\otimes 1_{M_{n(k)}(\C)},1_{M_{n(k)}(\C)}\otimes (B^{(k)})^{t})$ converges strongly to the law of $(x\otimes 1_{C^{*}(x)^{op}},1_{C^{*}(x)}\otimes x^{op})$. Then
\[\liminf_{k\to\infty}d^{\orb{}}(F(A^{(k)}),F(B^{(k)}))\geq \varepsilon.\]
\end{prop}
\begin{proof}
 By  \cite[Lemma 2.2]{HaagerupAmenable}, we may find a nonzero projection $p\in Z(Q)$ and $u_{1},\cdots,u_{r}\in \mathcal{U}(Qp)$ so that
\[C=\left\|\frac{1}{r}\sum_{j=1}^{r}u_{j}\otimes\overline{u_{j}}\right\|_{\infty}<1.\]
Fix any $C'\in (C,1).$ Choose $P\in \mathcal{F}_{R,I}$ with $\|P\|_{R,\infty}\leq 1$ and $P(x)=p,$ and $F_{j}\in \mathcal{F}_{R,J}$,$j=1,\cdots,r$ with $F_{j}(x)=u_{j}$ and $\|F_{j}\|_{R,\infty}\leq 1.$ Set $F=(F_{1},\cdots,F_{r}).$ Suppose we have
\begin{itemize}
    \item positive integers $(n(k))_{k=1}^{\infty}$ with $n(k)\to \infty,$
    \item $A^{(k)},B^{(k)}\in \prod_{i\in I}\{C\in M_{n(k)}(\C):\|C\|_{\infty}\leq R_{i}\}$
\end{itemize}
so that the law of $(A^{(k)}\otimes 1_{M_{n(k)}(\C)},1_{M_{n(k)}(\C)}\otimes (B^{(k)})^{t})$ converges strongly to the law of $(x\otimes 1_{C^{*}(x)^{op}},1_{C^{*}(x)}\otimes x^{op})$. Choose unitaries $U^{(k)}\in \mathcal{U}(k)$ so that
\[d^{\orb}(F(A^{(k)}),F(B^{(k)}))=\|F(A^{(k)})-U^{(k)}F(B^{(k)})(U^{(k)})^{*}\|_{2}.\]
Then,
\[d^{\orb}(F(A^{(k)}),F(B^{(k)}))^{2}=\sum_{j=1}^{r}(\|F_{j}(A^{(k)})\|_{2}^{2}+\|F_{j}(B^{(k)})\|_{2}^{2})-2\sum_{j=1}^{r}\rea{\tr(F_{j}(A^{(k)})U^{(k)}F_{j}(B^{(k)})^{*}(U^{(k)})^{*})}.\]
By weak$^{*}$ convergence of laws,
\[\liminf_{k\to\infty}d^{\orb}(F(A^{(k)}),F(B^{(k)}))^{2}\geq 2r\tau(p)-2\limsup_{k\to\infty}\sum_{j=1}^{r}\rea{\tr(F_{j}(A^{(k)})U^{(k)}F_{j}(B^{(k)})^{*}(U^{(k)})^{*})}.\]
Since $F_{j}(x)P(x)=F_{j}(x),$ and $\|F_{j}\|_{R,\infty}\leq 1$,$\|P\|_{R,\infty}\leq 1$ for all $j=1,\cdots,r,$ we have \[\|F_{j}(B^{(k)})-P(B^{(k)})F_{j}(B^{(k)})P(B^{(k)})^{*}\|_{2}\to 0.\]
So using once again that $\|F_{j}\|_{R,\infty}\leq 1$ for all $j=1\cdots,r,$ it follows that
\begin{align*}
   \liminf_{k\to\infty}d^{\orb}(F(A^{(k)}),F(B^{(k)}))^{2}&\geq 2r\tau(p)\\
   &-2\limsup_{k\to\infty}\sum_{j=1}^{r}\rea{\tr(F_{j}(A^{(k)})U^{(k)}P(B^{(k)})F_{j}(B^{(k)})^{*}P(B^{(k)})^{*}(U^{(k)})^{*})}\\
   &\geq 2r\tau(p)\\
   &-2\limsup_{k\to\infty}\|P(B^{(k)})\|_{2}\left\|\sum_{j=1}^{r}F_{j}(A^{(k)})U^{(k)}P(B^{(k)})F_{j}(B^{(k)})^{*}\right\|_{2}
\end{align*}
Since $\|P(B^{(k)})\|_{2}\to \|P(x)\|_{2}=\sqrt{\tau(p)},$ we obtain:
\begin{equation}\label{eqn:lower bounds on orbit distance}
\liminf_{k\to\infty}d^{\orb}(F(A^{(k)}),F(B^{(k)}))^{2}\geq 2r\tau(p)-2\sqrt{\tau(p)}\limsup_{k\to\infty}\left\|\sum_{j=1}^{r}F_{j}(A^{(k)})U^{(k)}P(B^{(k)})F_{j}(B^{(k)})^{*}\right\|_{2}
\end{equation}

To bound the second term in this expression,
let $E=\Span(\{u_{j}\}_{j=1}^{r}\cup\{1\}\cup \{u_{j}^{*}\}_{j=1}^{r}).$ By local reflexivity
 we have
 \[(u_{j})_{j=1}^{r}\in\overline{\{(\phi(u_{j}))_{j=1}^{r}:\phi\in CP(E,C^{*}(x)),\|\phi\|_{cb}\leq 1\}}^{WOT}.\]
 So, by convexity,
 \[(u_{j})_{j=1}^{r}\in\overline{\{(\phi(u_{j}))_{j=1}^{r}:\phi\in CP(E,C^{*}(x)),\|\phi\|_{cb}\leq 1\}}^{SOT}.\]
 Hence we may find a sequence $\phi_{m}\colon E\to C^{*}(x)$ of contractive, completely positive maps with
 \[\|\phi_{m}(u_{j})-u_{j}\|_{2}\to 0\]
 for all $j=1,\cdots,r.$ Choose $Q_{j,m}\in \C^{*}\ip{(T_{j})_{j\in J}}$ with
 \[\|Q_{j,m}(x)-\phi_{m}(u_{j})\|_{\infty}\leq \min\left(\frac{C'-C}{2},2^{-m}\right)\mbox{ for all $j=1,\cdots,r$,}\]
 \[\|Q_{j,m}\|_{\infty}\leq 1\mbox{ for all $j=1,\cdots,r$.}\]
 Since $u_{j}=F_{j}(x),$ and $\|P\|_{R,\infty}$, $\|Q_{j,m}\|_{R,\infty}$, $\|F_{j}\|_{R,\infty}\leq 1$ for all $j=1,\cdots,r$, we obtain
 \begin{align*}
   \left\|\sum_{j=1}^{r}F_{j}(A^{(k)})U^{(k)}P(B^{(k)})F_{j}(B^{(k)})^{*}\right\|_{2}&\leq  \sum_{j=1}^{r}\|F_{j}(A^{(k)})-Q_{j,m}(A^{(k)})\|_{2}\\
   &+\sum_{j=1}^{r}\|F_{j}(B^{(k)})-Q_{j,m}(B^{(k)})\|_{2}\\
   &+\left\|\sum_{j=1}^{r}Q_{j,m}(A^{(k)})U^{(k)}P(B^{(k)})Q_{j,m}(B^{(k)})^{*}\right\|_{2}\\ &\leq \sum_{j=1}^{r}\|F_{j}(A^{(k)})-Q_{j,m}(A^{(k)})\|_{2}\\
   &+\sum_{j=1}^{r}\|F_{j}(B^{(k)})-Q_{j,m}(B^{(k)})\|_{2}\\
   &+\left\|\sum_{j=1}^{r}Q_{j,m}(A^{(k)})\otimes \overline{Q_{j,m}(B^{(k)})}\right\|_{\infty}\|P(B^{(k)})\|_{2}.
 \end{align*}
 Using strong convergence,
\[\limsup_{k\to\infty}\left\|\sum_{j=1}^{r}F_{j}(A^{(k)})U^{(k)}P(B^{(k)})F_{j}(B^{(k)})^{*}\right\|_{2}\leq 2\sum_{j=1}^{r}\|u_{j}-Q_{j,m}(x)\|_{2}+\sqrt{\tau(p)}\left\|\sum_{j=1}^{r}Q_{j,m}(x)\otimes \overline{Q_{j,m}(x)}\right\|_{\infty}\]
By our choice of $Q_{j,m},$ we have
\[\left\|\sum_{j=1}^{r}Q_{j,m}(x)\otimes \overline{Q_{j,m}(x)}\right\|_{\infty}\leq (C'-C)r+\left\|\sum_{j=1}^{r}\phi_{m}(u_{j})\otimes \overline{\phi_{m}(u_{j})}\right\|_{\infty}\leq C'r,\]
where in the last step we use the definition of $C$ and the fact that $\|\phi_{m}\|_{cb}\leq 1$ implies $\|\phi_{m}\otimes \phi_{m}^{op}\|_{cb}\leq 1.$ So altogether we have shown that
\[\limsup_{k\to\infty}\left\|\sum_{j=1}^{r}F_{j}(A^{(k)})U^{(k)}P(B^{(k)})F_{j}(B^{(k)})^{*}\right\|_{2}\leq C'r\sqrt{\tau(p)}+2\sum_{j=1}^{r}\|u_{j}-Q_{j,m}(x)\|_{2}.\]
Inserting this into (\ref{eqn:lower bounds on orbit distance}),
\[\liminf_{k\to\infty}d^{\orb}(F(A^{(k)}),F(B^{(k)}))^{2}\geq 2r\tau(p)(1-C')-4\sum_{j=1}^{r}\|u_{j}-Q_{j,m}(x)\|_{2}.\]
 Letting $m\to\infty$ and then $C'\to C$ shows that
 \[\liminf_{k\to\infty}d^{\orb}(F(A^{(k)}),F(B^{(k)}))\geq \sqrt{2\tau(p)(1-C)}.\]
 So setting $\varepsilon=\sqrt{2\tau(p)(1-C)}$ completes the proof.

\end{proof}

We will give a cleaner way to state the above Proposition in terms of ultraproducts in Section \ref{sec: Jung} (see Proposition \ref{prop:not conj up}).  For now, we proceed to the proof of Theorem \ref{T:more general main theorem intro}.

\begin{proof}[Proof of Theorem \ref{T:more general main theorem intro} (\ref{item:Pinskers are amenable intro})]

Set $R=(R_{1},\cdots,R_{l})\in [0,\infty)^{l}.$ Let $\mu^{(k)}\in \Prob(M_{k}(\C)^{l})$ be the distribution of $(X_{j}^{(k)})_{j=1}^{l}$.
Suppose, for the sake of contradiction, that $Q\leq M$ is nonamenable and $h(Q:M)\leq 0.$ Since our hypotheses necessarily imply that $M$ embeds into an ultrapower of $\mathcal{R},$ it follows that $h(Q:M)=0.$ Let $r\in \N$ and $F\in (\mathcal{F}_{R,l})^{r}$ with $F(x)\in Q^{r}$ and $\varepsilon>0$ be as in the conclusion to Proposition \ref{prop:almost unitary conjugacy}. By Theorem \ref{thm:microstates collapse generalization}, we may choose a sequence $\Omega_{k}\subseteq \prod_{j=1}^{l}\{A\in M_{k}(\C):\|A\|_{\infty}\leq l\}$ with $\mu^{(k)}(\Omega_{k})=1$ and so that
\[\lim_{k\to\infty}\sup_{A,B\in \Omega_{k}}d^{\orb}(F(A),F(B))=0.\]
By strong convergence in probability, we may choose a sequence $\Theta_{k}\subseteq M_{k}(\C)^{l}\times M_{k}(\C)^{l}$ with
\[\mu^{(k)}\times \mu^{(k)}(\Theta_{k})\to 1,\]
and so that for any sequence $(A^{(k)},B^{(k)})\in \Theta_{k}$ the law of $(A^{(k)}\otimes 1_{M_{k}(\C)},1_{M_{k}(\C)}\otimes (B^{(k)})^{t})$ converges strongly to the law of $(x\otimes 1,1\otimes x^{op}).$ Since $\mu^{(k)}(\Omega_{k})\to 1$ and $\mu^{(k)}\otimes \mu^{(k)}(\Theta_{k})\to 1,$ for all large $k$ we may choose a $(A^{(k)},B^{(k)})\in (\Omega_{k}\times \Omega_{k})\cap \Theta_{k}.$ Then Proposition \ref{prop:almost unitary conjugacy} shows
\[\liminf_{k\to\infty}d^{\orb}(F(A^{(k)}),F(B^{(k)})\geq \varepsilon,\]
whereas our choice of $\Omega_{k}$ implies
\[\lim_{k\to\infty}d^{\orb}(F(A^{(k)}),F(B^{(k)})=0.\]
This gives a contradiction, which completes the proof.

\end{proof}

\subsection{Deduction of Theorem \ref{T:main}  from Theorem \ref{T:more general main theorem intro}}\label{subsec: proof of main theorem}

In this section, we shall deduce Theorem \ref{T:main}  from Theorem \ref{T:more general main theorem intro}. We also state a version for free families of \emph{Haar unitaries} instead of \emph{free semicirculars}. Moreover, it is not hard to see that our proof applies equally well to many other families of random matrices which model $L(\F_{r})$, provided that they exhibit exponential concentration with the correct rate.  We start with the proof of Theorem \ref{T:main}

\begin{proof}[Proof of Theorem \ref{T:main}]
The fact that (\ref{item: Pinskers are amenable}) implies (\ref{item:PT intro}) is the content of Proposition \ref{P:CPE implies PT}. So we focus on proving that (\ref{item:tensor Haus}) implies (\ref{item: Pinskers are amenable}).

Let $s=(s_{1},s_{2},\cdots,s_{r})$ be a free semicircular family with mean zero and variance $1.$ So $W^{*}(s)\cong L(\F_{r}).$
We let $\mu^{(k)}\in \Prob(M_{k}(\C)_{s.a.}^{2r})$ be the distribution of $(X^{(k)},Y^{(k)}).$ It is well known (see \cite[Proof of Lemma 3.3]{HaagThorbNormBound}) that there is a $C>0$ so that
\[\limsup_{k\to\infty}\frac{1}{k}\log \mu^{(k)}((\{A\in M_{k}(\C)_{s.a}^{2r}:\|A\|_{\infty}\leq C\})^{c})<0,\]
and Voiculescu's asymptotic freeness theorem (specifically \cite[Theorem 2.2]{VoicAsyFree}) implies that $\mu^{(k)}$ is asymptotically concentrated on microstates for $s$ (using $R$ as the constant function $C$). Further, exponential concentration of measure with scale $k^{2}$ is well known and follows, e.g., from \cite[Equation (2.10)]{LedouxCM}.
It  is direct to see that the coordinate-wise transpose map $M_{k}(\C)_{s.a.}^{r}\to M_{k}(\C)_{s.a.}^{r}$ preserves $\mu^{(k)}.$ Further, $s^{op}=(s_{1}^{op},s_{2}^{op},\cdots,s_{r}^{op})$ is also a tuple of $r$ free semicircular elements each with mean $0$ and variance $1,$ and so $s^{op}$ has the same distribution as $s.$ So the strong convergence in probability of
$(X^{(k)}\otimes 1_{M_{k}(\C)},1_{M_{k}(\C)}\otimes Y^{(k)})$
to $(s\otimes 1_{C^{*}(x)},1_{C^{*}(x)}\otimes s)$ is equivalent to the strong convergence in probability of $(X^{(k)}\otimes 1_{M_{k}(\C)},1_{M_{k}(\C)}\otimes (Y^{(k)})^{t})$ to $(s\otimes 1_{C^{*}(s^{op})},1_{C^{*}(s)}\otimes s^{op}).$
Thus we may apply Theorem \ref{T:more general main theorem intro} (\ref{item:Pinskers are amenable intro}), and the conclusion of that Theorem gives us exactly what we want.

\end{proof}

We also state a version of Theorem \ref{T:main} using independent Haar unitaries instead of the GUE ensemble.

\begin{thm}\label{T:Haar unitaries}
Fix an  integer $r\geq 2.$ For each $k\in \N,$ let $U_{1}^{(k)},\cdots,U_{r}^{(k)},V_{1}^{(k)},\cdots,V_{r}^{(k)}$ be random $k\times k$ unitary matrices which are independent and are each distributed according to Haar measure on $\mathcal{U}(k).$ Set
\[U^{(k)}\otimes 1_{M_{k}}(\C)=(U^{(k)}_{j}\otimes 1_{M_{k}(\C)})_{j=1}^{r},\,\,\, 1_{M_{k}(\C)}\otimes V^{(k)}=(1_{M_{k}(\C)}\otimes V^{(k)}_{j})_{j=1}^{r}.\]
Let $\F_{r}$ be the free group on $r$ letters $a_{1},\cdots,a_{r},$ and let \[\lambda(a)\otimes 1=(\lambda(a_{j})\otimes 1_{C^{*}_{\lambda}(\F_{r})})_{j=1}^{r},\,\, 1\otimes \lambda(a)=(1\otimes \lambda(a_{j}))_{j=1}^{r}.\]
If the distribution of
$(U^{(k)}\otimes 1_{M_{k}(\C)},1_{M_{k}(\C)}\otimes U^{(k)})$ converges (as $k\to\infty$) to the distribution of $ (\lambda(a)\otimes 1_{C^{*}_{\lambda}(\F_{r})},1_{C^{*}_{\lambda}(\F_{r})}\otimes U^{(k)}) $
strongly  in probability, then for any $Q\leq L(\F_{r})$ with $h(Q:L(\F_{r}))\leq 0$ we have that $Q$ is amenable.

\end{thm}

\begin{proof}
It is well known that the distribution of $(U^{(k)}_{j})_{j=1}^{r}$ satisfies exponential concentration of measure with scale $k^{2}$ (for example this follows from \cite[Theorem 5.3]{LedouxCM} and \cite[Theorem 15]{Meckes2013}). By compactness, the Haar measure on $\mathcal{U}(k)$ is invariant under right multiplication, and thus under anti-automorphisms. So the distribution of $(U^{(k)}_{j})^{t}$ is the same as the distribution of $U^{(k)}_{j}$ for all $j=1,\cdots,r.$ Additionally, it is direct to show that the unique homomorphism $\F_{r}\to\F_{r}$ sending $a_{j}$ to $a_{j}^{-1}$ is bijective, and so $\lambda(a)^{op}$ has the same law as $\lambda(a).$ The proof now proceeds exactly as in the proof of (\ref{item:tensor Haus}) implies (\ref{item: Pinskers are amenable}) of Theorem \ref{T:main}.

\end{proof}

\section{Intermediate conjectures and relation to Jung's Theorem}\label{sec: Jung}

In this section, we collect various conjectures which imply the Peterson-Thom conjecture, and discuss their relative strength. We start by stating the conjectures already discussed in the introduction.

\begin{conj}\label{conj:tensor Haus}
Let $X_{1}^{(k)},X_{2}^{(k)},\cdots,X_{r}^{(k)}$,$Y_{1}^{(k)},Y_{2}^{(k)},\cdots,Y_{r}^{(k)}$ be random, self-adjoint $k\times k$ matrices which are independent and are each GUE distributed. Set $X^{(k)}=(X_{j}^{(k)})_{j=1}^{r}$,$Y^{(k)}=(Y_{j}^{(k)})_{j=1}^{r}.$ Let $s=(s_{1},\cdots,s_{r})$ be a tuple of free semicircular elements which have mean zero and variance $1$. Then for every $P\in \C^{*}\ip{(T_{j})_{j\in J}}$ we have
\[\|P(X^{(k)}\otimes 1_{M_{k}(\C)},1_{M_{k}(\C)}\otimes Y^{(k)})\|_{\infty}\to_{k\to\infty} \|P(s\otimes 1_{C^{*}(s)},1_{C^{*}(s)}\otimes s)\|_{\infty}\]
in probability.
\end{conj}

\begin{conj}\label{C:1bounded ent conjecture}
Let $Q\leq L(\F_{r})$ be diffuse and nonamenable, then
\[h(Q:L(\F_{r}))>0.\]
\end{conj}

We now explain some intermediate conjectures, the first of which is formulated in an ultraproduct framework.
\begin{defn}
Let $\omega$ be a free ultrafilter on $\N,$ and let $(M_{k},\tau_{k})_{k=1}^{\infty}$ be a sequence of tracial von Neumann algebras. We define their \emph{tracial ultraproduct with respect to $\omega$} by
\[\prod_{k\to\omega}(M_{k},\tau_{k})=\frac{\{(x_{k})_{k}\in \prod_{k}M_{k}:\sup_{k}\|x_{k}\|_{\infty}<\infty\}}{\{(x_{k})_{k}\in \prod_{k}M_{k}:\sup_{k}\|x_{k}\|_{\infty}<\infty,\mbox{ and}\lim_{k\to\omega}\|x_{k}\|_{L^{2}(\tau_{k})}=0\}}.\]
If $(x_{k})_{k}\in \prod_{k}M_{k}$ and $\sup_{k}\|x_{k}\|_{\infty}<\infty,$ we let $(x_{k})_{k\to\omega}$ be the image of $(x_{k})_{k}$ under the quotient map.
If $J$ is an index set, and $(x_{k})_{k}\in \prod_{k}M_{k}^{J}$ and
\[\sup_{k}\|x_{k,j}\|_{\infty}<\infty\mbox{ for all $j\in J$},\]
then we let $(x_{k})_{k\to\omega}\in \left(\prod_{k\to\omega}(M_{k},\tau_{k})\right)^{J}$ be the tuple whose j$^{th}$ coordinate is $(x_{k,j})_{k\to\omega}.$
\end{defn}

As is well known,  $\prod_{k\to\omega}(M_{k},\tau_{k})$ is a tracial von Neumann algebra with the $*$-algebra operations defined pointwise and the trace given by $\tau_{\omega}((x_{k})_{k\to\omega})=\lim_{k\to\omega}\tau_{k}(x_{k})$ (this follows from the same argument as  \cite[Lemma A.9]{BO}). It will helpful to know that the noncommutative functional calculus described in Section \ref{S:nc func calc} commutes with passing to the ultraproduct.

\begin{lem}\label{lem:nc func calc and the ultraproduct}
Let $(M_{k},\tau_{k})_{k}$ be a sequence of tracial von Neumann algebras and let $\omega$ be a free ultrafilter on the natural numbers. Fix an index set $J,$  $R\in [0,\infty)^{J}$ and suppose
\[(x_{k})_{k}\in \prod_{k}M_{k}^{J}\]
with
$\|x_{k,j}\|_{\infty}\leq R_{j}\mbox{ for all $k\in \N$, $j\in J$.}$
Then for any $f\in \mathcal{F}_{R,J,\infty}$
\[f((x_{k})_{k\to\omega})=(f(x_{k}))_{k\to\omega}.\]
\end{lem}

\begin{proof}
First, note that the conclusion of the lemma is true for $f\in \mathcal{A}_{R,J}.$ For the general case, fix $f\in \mathcal{F}_{R,J,\infty}.$ Given $\varepsilon>0,$ choose a $g\in \mathcal{A}_{R,J}$ with
$\|f-g\|_{R,2}<\varepsilon.$
Then
\begin{align*}
   \|f((x_{k})_{k\to\omega})-(f(x_{k}))_{k\to\omega}\|_{2}&\leq \|(f-g)((x_{k})_{k\to\omega})\|_{2}
+\|((f-g)(x_{k}))_{k\to\omega}\|_{2}\\
&\leq  \|f-g\|_{R,2}+\lim_{k\to\omega}\|(f-g)(x_{k})\|_{2}\\
&\leq 2\|f-g\|_{R,2}\\
&<2\varepsilon.
\end{align*}

\end{proof}

\begin{thm}\label{thm:as conjugation}
Suppose we are given a tracial von Neumann algebra $(M,\tau)$, a countable index set $J,$ and an $x\in M^{J}$ with $W^{*}(x)=M.$ Suppose $R\in [0,\infty)^{J}$ satisfies $\|x_{j}\|_{\infty}\leq R_{j}$ for all $j\in J.$ Assume we are given a sequence of natural numbers $n(k)\to\infty,$ and a sequence  $\mu^{(k)}\in \Prob(M)_{k}(\C)^{J})$ such that
\begin{itemize}
    \item $\sum_{k}\mu^{(k)}\left(\left(\prod_{j\in J}\{A\in M_{k}(\C)^{J}:\|A_{j}\|_{\infty}\leq R_{j}\right)^{c}\right)<\infty$.
    \item $\mu^{(k)}(\Gamma^{(n(k))}_{R}(\mathcal{O}))\to 1$ for all weak$^{*}$ neighborhoods $\mathcal{O}$ of $\ell_{x}$ in $\Sigma_{R,J}$,
    \item $\mu^{(k)}$ has exponential concentration with scale $n(k)^{2}.$
\end{itemize}
Then:
\begin{enumerate}[(i)]
    \item there is a conull subset $\Omega_{0}\subseteq \prod_{k}M_{n(k)}(\C)^{J}$ so that for any $A=(A^{(k)})_{k}\in \Omega_{0},$ and for every free ultrafilter $\omega$ on $\N,$ there is a unique trace-preserving $*$-homomorphism $\Theta_{A,\omega}\colon M\to\prod_{k\to\omega}(M_{n(k)}(\C),\tr_{n(k)})$ so that
    \[\Theta_{A,\omega}(P(x))=(P(A^{(k)}))_{k\to\omega}\,\mbox{ for all $P\in \C^{*}\ip{(T_{j})_{j\in J}}$}.\] \label{item:as defined homom}
    \item If $Q\leq M$ satisfies $h(P:M)\leq 0,$ then there is a conull subset $\Omega\subseteq \Omega_{0}$ so that for all $A,B\in \Omega$ and for every free ultrafilter $\omega$ on $\N,$ we have that  $\Theta_{A,\omega}\big|_{Q},$ $\Theta_{B,\omega}\big|_{Q}$ are unitarily conjugate.  \label{item:almost sure conjugacy}
\end{enumerate}

\end{thm}

\begin{proof}

Let $\nu^{(k)}$ be the measure on $M_{n(k)}(\C)^{J}$ given by
\[\nu^{(k)}(E)=\frac{\mu^{(k)}\left(E\cap  \left(\prod_{j\in J}\{A\in M_{n(k)}(\C):\|A\|_{\infty}\leq R_{j}\}\right)\right)}{\mu^{(k)}\left(\prod_{j\in J}\{A\in M_{n(k)}(\C):\|A\|_{\infty}\leq R_{j}\}\right)}.\]

(\ref{item:as defined homom}): It suffices to find a $\bigotimes_{k}\mu^{(k)}$-conull $\Omega_{1}\subseteq \prod_{k}M_{n(k)}(\C)^{J}$ so that for every $A=(A^{(k)})_{k}\in \Omega_{1},$ we have $\ell_{A_{k}}\to \ell_{x}.$  Fix a decreasing sequence $\mathcal{O}_{m}\subseteq \Sigma_{R,J}$ of weak$^{*}$-neighborhoods of $\ell_{x}$  with
\[\bigcap_{m=1}^{\infty}\mathcal{O}_{m}=\{\ell_{x}\},\]
this is possible as $J$ is countable.
By Lemma \ref{lem: reduction lemma} (\ref{item: auto exp decay}),
\[\sum_{k}\left(\otimes_{s}\nu^{(s)}\right)\left(\{(A^{(s)})_{s}:A^{(k)}\in \Gamma_{R}^{(n(k))}(\mathcal{O}_{m})^{c}\}\right)<\infty\]
for every $m\in \N$.
For every $k\in \N,$ we have that
\begin{align*}
 \mu^{(k)}(\Gamma_{R}^{(n(k))}(\mathcal{O}_{m})^{c}))&\leq \mu^{(k)}\left(\Gamma_{R}^{(n(k))}(\mathcal{O}_{m})^{c}\cap \prod_{j\in J}\{A\in M_{n(k)}(\C):\|A\|_{\infty}\leq R_{j}\}\right)\\
 &+\mu^{(k)}\left(\left(\prod_{j\in J}\{A\in M_{n(k)}(\C):\|A\|_{\infty}\leq R_{j}\}\right)^{c}\right)\\
 &=\nu^{(k)}\left(\Gamma_{R}^{(n(k))}(\mathcal{O}_{m})^{c}\right)\mu^{(k)}\left(\prod_{j\in J}\{A\in M_{n(k)}(\C):\|A\|_{\infty}\leq R_{j}\}\right)\\
 &+\mu^{(k)}\left(\left(\prod_{j\in J}\{A\in M_{n(k)}(\C):\|A\|_{\infty}\leq R_{j}\}\right)^{c}\right).
\end{align*}
So
\[\sum_{k}\left(\otimes_{s}\mu^{(s)}\right)\left(\{(A^{(s)})_{s}:A^{(k)}\in \Gamma_{R}^{(n(k))}(\mathcal{O}_{m})^{c}\}\right)<\infty.\]
Hence
\[\Omega_{0}=\bigcap_{m}\left(\bigcup_{k}\bigcap_{l\geq k}\{(A^{(s)})_{s}:A^{(l)}\in \Gamma_{R}^{(n(l))}(\mathcal{O}_{m})\}\right)\]
is a conull subset of $\prod_{k}M_{n(k)}(\C).$ By construction, for every $A\in \Omega_{0}$ we have $\ell_{A^{(k)}}\to \ell_{x}.$

(\ref{item:almost sure conjugacy}):
Fix a countable set $J'$ and a tuple $y\in M^{J'}$ with  $W^{*}(y)=Q.$ Choose an $R\in [0,\infty)^{J}$ with $\|x_{j}\|_{\infty}\leq R_{j}$ for all $j\in J$ and an $f\in (\mathcal{F}_{R,J,\infty})^{J'}$ with $f(x)=y.$
By Lemma \ref{lem:nc func calc and the ultraproduct}  and Theorem \ref{thm:nc func calc properties} (\ref{item:naturality of nc func calc}), for any free ultrafilter $\omega,$ and any $A=(A^{(k)})\in \Omega$ we have
\[\Theta_{A,\omega}(y)=\Theta_{A,\omega}(f(x))=f(\Theta_{A,\omega}(x))=(f(A^{(k)}))_{k\to\omega}.\]
So it suffices to find a conull subset $\Omega$ of $\prod_{k}M_{n(k)}(\C)^{J}$ so that for all $A=(A^{(k)}),B=(B^{(k)})\in \Omega$ there is a sequence of unitaries $U^{(k)}\in M_{n(k)}(\C)$ so that
\[\|U^{(k)}f_{j'}(A^{(k)})(U^{(k)})^{*}-f_{j'}(B^{(k)})\|_{2}\to_{k\to\infty}0,\mbox{ for all $j'\in J'$.}\]
Since $J$ is countable, by a diagonal argument it is sufficient to show that for every $\varepsilon>0,$ and for every finite $F'\subseteq J,$ there is a conull subset $\Upsilon_{F',\varepsilon}$ of $\prod_{k}\Prob(M_{k}(\C)^{J})$ so that for all $A=(A^{(k)})_{k},B=(B^{(k)})_{k}\in \Upsilon_{F',\varepsilon}$ we have
\[\limsup_{k\to\infty}d^{\orb{}}_{F'}(f(A^{(k)}),f(B^{(k)}))\leq \varepsilon.\]
So fix an $\varepsilon>0$ and a finite $F'\subseteq J'.$ Then by Theorem \ref{thm:nc func calc properties} (\ref{item:unif cont nc func calc}), we may find a $\delta>0$ and a finite $F\subseteq J$ so that if $(M,\tau)$ is any tracial von Neumann algebra, and if $a,b\in \prod_{j\in J}\{c\in M:\|c\|_{\infty}\leq R_{j}\}$ satisfy $\|a-b\|_{2,F}<\delta,$ then $\|f(a)-f(b)\|_{2,F'}<\varepsilon/4.$ Since $f$ commutes with unitary conjugation by Theorem \ref{thm:nc func calc properties} (\ref{item:naturality of nc func calc}), it follows that for every $n\in \N,$ and all $A,B\in \prod_{j\in J}\{C\in M_{n}(\C):\|C\|_{\infty}\leq R_{j}\}$ with $d^{\orb{}}_{F}(A,B)<\delta,$ we have  $d^{\orb{}}_{F'}(f(A),f(B))<\varespilon/4.$
By Theorem \ref{thm:microstates collapse generalization}, we may choose a sequence $\widetilde{\Upsilon}_{k}\subseteq \prod_{j\in J}\{C\in M_{n(k)}(\C):\|C\|_{\infty}\leq R_{j}\}$ with $\nu^{(k)}(\widetilde{\Upsilon}_{k})\to 1$ and so that
\[\lim_{k\to\infty}\sup_{A_{1},A_{2}\in\widetilde{\Upsilon}_{k}}d^{\orb{}}_{F}(f(A_{1}),f(A_{2}))=0.\]
Now choose $K$ so that for all $k\geq K$ we have $\nu^{(k)}(\widetilde{\Upsilon}_{k})\geq 1/2$ and
\[\sup_{A_{1},A_{2}\in \widetilde{\Upsilon}_{k}}d^{\orb{}}_{F}(f(A_{1}),f(A_{2}))<\varepsilon/2.\]
Then, by exponential concentration,
\[\sum_{k\geq K}\nu^{(k)}(N_{\delta}(\widetilde{\Upsilon}_{k})^{c})<\infty.\]
As in part (\ref{item:as defined homom}) we have
\[\sum_{k\geq K}\mu^{(k)}(N_{\delta}(\widetilde{\Upsilon}_{k})^{c})<\infty.\]
If $A,B\in N_{\delta}(\widetilde{\Upsilon}_{k}),$ choose $A_{1},B_{1}\in \widetilde{\Upsilon}_{k}$ with $d^{\orb{}}_{F}(A,A_{1}),$ $d^{\orb{}}_{F}(B,B_{1})<\delta.$ Then,
\[d^{\orb{}}_{F'}(f(A),f(B))<\varepsilon/2+d^{\orb{}}_{F'}(f(A_{1}),f(B_{1}))<\varespilon.\]
So
\[\Upsilon_{F,\varepsilon}=\bigcup_{k\geq K}\bigcap_{l\geq k}\{(A^{(m)})_{m}:A^{(l)}\in N_{\delta}(\widetilde{\Upsilon}_{l})\}\]
is $\bigotimes_{k}\mu^{(k)}$-conull and for all $(A^{(k)})_{k},(B^{(k)})_{k}\in \Upsilon_{F,\varepsilon}$ we have
\[\limsup_{k\to\infty}d^{\orb{}}_{F'}(f(A^{(k)}),f(B^{(k)}))\leq \varepsilon.\]
This completes the proof.

\end{proof}

The conclusion of  Theorem \ref{thm:as conjugation} (\ref{item:almost sure conjugacy}) is interesting in light of the following theorem of Jung.

\begin{thm}[Jung, \cite{JungConj}]\label{T:JungThm}
Let $(M,\tau)$ be a tracial von Neumann which admits an embedding into a tracial ultraproduct of matrix algebras. Then given any nonamenable $Q\leq M$  and any free ultrafilter $\omega$ on $\N,$ there  are trace-preserving, normal $*$-homomorphisms
\[\Theta_{j}\colon M\to \prod_{k\to\omega} M_{k},j=1,2\]
so that $\Theta_{1}\big|_{Q}$ is not unitarily equivalent to $\Theta_{2}\big|_{Q}.$
Conversely, if $Q\leq M$ is amenable, then any two embeddings of $Q$ into an ultraproduct of matrix algebras are unitarily equivalent.
\end{thm}
Strictly speaking, Jung only proved the case $Q=M$ and when $M$ is finitely generated of Theorem \ref{T:JungThm}. However, by analyzing his proof and replacing microstates spaces with microstates spaces in the presence, it is not hard to prove the case $Q$ is nonamenable and finitely generated of Theorem \ref{T:JungThm}. Since any nonamenable von Neumann algebra has a finitely generated nonamenable von Neumann subalgebra, this is sufficient to handle the general case of Theorem \ref{T:JungThm}.

Under the hypotheses of Theorem \ref{thm:as conjugation}, if for every nonamenable $Q\leq M,$  almost every $(A,B)\in \Omega_{0}$ and every free ultrafilter $\omega$ on the natural numbers, we had that $\Theta_{A,\omega}\big|_{Q}$ and $\Theta_{B,\omega}\big|_{Q}$ are not unitarily conjugate, then it would follow that any $N\leq M$ with $h(N:M)\leq 0$ must be amenable. In particular, any amenable subalgebra of $M$ must have a maximal amenable extension. We can think of the statement that almost surely $\Theta_{A,\omega}\big|_{Q}$ is not unitarily equivalent to $\Theta_{B,\omega}\big|_{Q}$ as a ``randomized Jung theorem". It would mean that not only can we find a pair of homomorphisms satisfying the conclusion of Jung's theorem, but that a \emph{randomly chosen pair} satisfies Jung's theorem. This motivates the following conjecture.
\begin{conj}\label{conj:random Jung}
Fix an integer $r\geq 2,$ and let $\mu^{(k)}\in \Prob(M_{k}(\C)_{s.a.}^{r})$ be the $r$-fold product of the GUE distribution. Set $\mu=\prod_{k}\mu^{(k)},$ and choose a $\mu$-conull $\Omega_{0}\subseteq \prod_{k}M_{k}(\C)_{s.a.}^{r}$ as in Theorem \ref{thm:as conjugation}.  Then for every $Q\leq L(\F_{r})$ nonamenable and for every free ultrafilter $\omega$ on $\N,$ there is a $\mu\otimes\mu$-conull subset $\Omega\subseteq\Omega_{0}\times \Omega_{0}$ so that $\Theta_{A,\omega}\big|_{Q}$,$\Theta_{B,\omega}\big|_{Q}$ are not unitarily conjugate for all $(A,B)\in \Omega.$
\end{conj}

Related to Jung's theorem, we can use strong convergence and local reflexivity to give criteria so that a concrete pair of embeddings into ultraproducts of matrices are not unitarily conjugate when restricted to any nonamenable subalgebra.

\begin{prop}\label{prop:not conj up}
Let $(M,\tau)$ be a tracial von Neumann algebra, $I$ an index set and $x\in M^{I}$ with $W^{*}(x)=M.$  Suppose we are given positive integers $n(k)\to\infty$ and $(A^{(k)},B^{(k)})\in M_{n(k)}(\C)^{I}$ so that
the law of $(A^{(k)}\otimes 1_{M_{n(k)}(\C)},1_{M_{n(k)}(\C)}\otimes (B^{(k)})^{t})$ converges strongly to the law of $(x\otimes 1_{C^{*}(x)^{op}},1_{C^{*}(x)}\otimes x^{op}).$ For a free ultrafilter $\omega,$ let $\Theta_{A,\omega}\colon M\to \prod_{k\to\omega}M_{n(k)}(\C),$ $\Theta_{B,\omega}\colon M\to \prod_{k\to\omega}M_{n(k)}(\C)$ be the unique trace-preserving, normal $*$-homomorphisms which satisfy
\[\Theta_{A,\omega}(x)=(A^{(k)})_{k\to\omega},\,\,\, \Theta_{B,\omega}(x)=(B^{(k)})_{k\to\omega}.\]
If $C^{*}(x)$ is locally reflexive,
then for any nonamenable $Q\leq M$ we have that $\Theta_{A,\omega}\big|_{Q}$ and $\Theta_{B,\omega}\big|_{Q}$ are not unitarily conjugate.
\end{prop}

\begin{proof}
Note that strong convergence implies that for all $i\in I,$
\[R_{i}=\sup_{k}\max(\|A^{(k)}_{i}\|_{\infty},\|B^{(k)}_{i}\|_{\infty})<\infty.\]
From here, it is an exercise to derive this from Proposition \ref{prop:almost unitary conjugacy}.

\end{proof}

Recall that if $x\in M_{n}(\C)\otimes M_{n}(\C)$ then we have an operator $x\#\colon M_{n}(\C)\to M_{n}(\C)$ defined on elementary tensors by
\[(A\otimes B)\#C=ACB^{t}.\]
Moreover, $x\mapsto x\#$ is an injective $*$-homomorphism $M_{n}(\C)\otimes M_{n}(\C)\to B(S^{2}(n,\tr))$ and as such it is isometric. So
\[\|x\|_{M_{n}(\C)\otimes M_{n}(\C)}=\|x\#\|_{B(S^{2}(n,\tr))}\]
and this is precisely what we used in our reduction to strong convergence. However, it is natural to view $x\#$ as an operator between other noncommutative $L^{p}$-spaces. Recall that if $1\leq p<\infty,$ then we have a norm $\|\cdot\|_{p}$ on $M_{n}(\C)$ by
\[\|A\|_{p}=\tr(|A|^{p})^{1/p},\mbox{ with $|A|=(A^{*}A)^{1/2}$.}\]
As usual, we let $\|A\|_{\infty}$ be the operator norm of $A\in M_{n}(\C).$
We let $S^{p}(n,\tr)$ be $M_{n}(\C)$ equipped with the norm $\|\cdot\|_{p}$, and for $x\in M_{n}(\C)\otimes M_{n}(\C)$ and $1\leq p,q\leq \infty,$ we let $\|x\#\|_{p,q}$ be the norm of the operator $A\mapsto x\#A$ as an operator $S^{p}(n,\tr)\to S^{q}(n,\tr).$ So our discussion above shows that
\[\|x\#\|_{2,2}=\|x\|_{M_{n}(\C)\otimes M_{n}(\C)}.\]
Because we are using the normalized trace, we have that $\|A\|_{p}\leq \|A\|_{q}$ for $1\leq p\leq q\leq \infty$ and $A\in M_{n}(\C).$ So
\[\|x\#\|_{p_{1},q_{1}}\leq \|x\#\|_{p_{2},q_{2}}\]
if $p_{2}\leq p_{1}$,$q_{1}\leq q_{2}$. We now state a conjecture weaker than our strong convergence conjecture in terms of operator norms $M_{n}(\C)\to S^{1}(n,\tr).$
\begin{conj}\label{conj:pq conjecture}
Fix an integer $r\geq 2.$ Then there is a constant $C>0$ with the following property. Let $X_{1}^{(k)},X_{2}^{(k)},\cdots,X_{r}^{(k)}$,$Y_{1}^{(k)},Y_{2}^{(k)},\cdots,Y_{r}^{(k)}$ be random, self-adjoint $k\times k$ matrices which are independent and are each GUE distributed. Let $s=(s_{1},\cdots,s_{r})$ be a free semicircular family each with mean zero and variance $1.$ Then for any $P\in \C\ip{T_{1},\cdots,T_{2r}}$
we have that
\[\limsup_{k\to\infty}\|P(X^{(k)}\otimes 1_{M_{k}(\C)},1_{M_{k}(\C)}\otimes Y^{(k)})\#\|_{\infty,1}\leq C\|P(s\otimes 1_{C^{*}(s)},1_{C^{*}(s)}\otimes s)\|_{\infty},\]
where the norm on the right-hand side is taken in $C^{*}(s)\otimes_{\min{}}C^{*}(s).$
\end{conj}

\begin{prop}\label{prop:conejctures upon conjecture}
We have the following implications between the above conjectures and the Peterson-Thom conjecture.
Conjecture \ref{conj:tensor Haus} implies Conjecture \ref{conj:pq conjecture}, Conjecture \ref{conj:pq conjecture} implies Conjecture \ref{conj:random Jung}, Conjecture \ref{conj:random Jung} implies Conjecture \ref{C:1bounded ent conjecture}, and Conjecture \ref{C:1bounded ent conjecture} implies the Peterson-Thom conjecture.
\end{prop}

\begin{proof}
Conjecture \ref{conj:tensor Haus} implies Conjecture \ref{conj:pq conjecture}: Take $C=1,$ and use that $\|x\#\|_{\infty,1}\leq \|x\#\|_{2,2}=\|x\|_{M_{n}(\C)\otimes M_{n}(\C)}$ for all $n\in \N,$ and all $x\in M_{n}(\C)\otimes M_{n}(\C).$

Conjecture \ref{conj:pq conjecture} implies Conjecture \ref{conj:random Jung}:
It is well known (see \cite[Proof of Lemma 3.3]{HaagThorbNormBound}) that we may find an $R>0$ so that
\[\limsup_{k\to\infty}\frac{1}{k}\log \mu^{(k)}\left(\left(\prod_{j=1}^{r}\{A\in M_{k}(\C):\|A\|_{\infty}\leq R\}\right)^{c}\right)<0.\]
Suppose $Q\leq L(\F_{r})$ is nonamenable, and apply \cite[Lemma 2.2]{HaagerupAmenable} to find a nonzero projection $f\in Z(Q)$ and $u_{1},\cdots,u_{r}\in \mathcal{U}(Qf)$ so
\[D'=\frac{1}{r}\left\|\sum_{j=1}^{r}u_{j}\otimes \overline{u_{j}}\right\|_{\infty}<1.\]
By replacing $\left(\frac{1}{r}\sum_{j=1}^{r}u_{j}\otimes \overline{u_{j}}\right)$ with $\left(\frac{1}{r}\sum_{j=1}^{r}u_{j}\otimes \overline{u_{j}}\right)^{s}$ for a suitably large $s\in \N,$ we may, and will, assume that $D'<\frac{\tau(f)}{C}.$ Let $\Omega_{0}$ be as in Theorem \ref{thm:as conjugation} (\ref{item:as defined homom}). By Conjecture \ref{conj:pq conjecture}, we may choose a conull $\Xi\subseteq \Xi\times \Omega_{0}$ so that for all $A=(A^{(k)}),B=(B^{(k)})\in \Xi,$ we have
\[\limsup_{k\to\infty}\|P(A^{(k)}\otimes 1_{M_{k}(\C)},1_{M_{k}(\C)}\otimes B^{(k)})\#\|_{\infty,1}\leq C\|P(s\otimes 1_{C^{*}(s)},1_{C^{*}(s)}\otimes s)\|_{\infty}.\]
Suppose that the negation of Conjecture \ref{conj:random Jung} holds. Then there is a positive measure $\Upsilon\subseteq \Xi$ and a free ultrafilter $\omega$ on $\N$ so that for all $(A,B)\in \Upsilon$ we have that $\Theta_{A,\omega}\big|_{Q}$ and $\Theta_{B,\omega}\big|_{Q}$ are unitarily conjugate. Fix $(A,B)\in \Upsilon.$ Let $v\in \mathcal{U}\left(\prod_{k\to\omega}M_{k}(\C)\right)$ be such that
\[v\Theta_{B,\omega}(x)v^{*}=\Theta_{A,\omega}(x) \mbox{ for all $x\in Q$,}\]
and write $v=(V^{(k)})_{k\to\omega}$ with $V^{(k)}\in \mathcal{U}(k).$

Observe that for all $(X_{k})_{k\to\omega}\in \prod_{k\to\omega}M_{k}(\C)$ we have
\[\|(X_{k})_{k\to\omega}\|_{1}=\lim_{k\to\omega}\|X_{k}\|_{1}.\]
Indeed, this follows from the fact that $|(X_{k})_{k\to\omega}|=(|X_{k}|)_{k\to\omega},$
which is in turn a consequence of the fact that continuous functional calculus commutes with the operation of passing to the ultraproduct. Since $C^{*}(s)$ is exact, and thus locally reflexive, as in the proof of Theorem \ref{T:more general main theorem intro} (\ref{item:Pinskers are amenable intro})) we may choose a $D\in (D',\frac{\tau(f)}{C})$ and a sequence $P_{j,m}\in \C\ip{(T_{j})_{j\in J}}$ so that
\begin{itemize}
    \item $\|P_{j,m}\|_{R,\infty}\leq 1,$
    \item $\|P_{j,m}(s)-u_{j}\|_{2}\to_{m\to\infty}0,$
    \item $\frac{1}{r}\left\|\sum_{j=1}^{r}P_{j,m}(s)\otimes \overline{P_{j,m}(s)}\right\|_{\infty}\leq D$ for all $m.$
\end{itemize}
Note that as a consequence of the second item
\[\|P_{j,m}(s)-u_{j}\|_{1}\leq \|P_{j,m}(s)-u_{j}\|_{2}\to_{m\to\infty}0.\]
Then for every $m\in \N,$
\begin{align*}
 \tau(f)=\|vf\|_{1}=\frac{1}{r}\left\|\sum_{j=1}^{r}\Theta_{A,\omega}(u_{j})v\Theta_{B,\omega}(f)\Theta_{B,\omega}(u_{j})^{*}\right\|_{1}&\leq \frac{2}{r}\sum_{j=1}^{r}\|P_{j,m}(s)-u_{j}\|_{1}\\
 &+\frac{1}{r}\left\|\sum_{j=1}^{r}P_{j,m}((A^{(k)})_{k\to\omega})v\Theta_{B,\omega}(f)P_{j,m}((B^{(k)})_{k\to\omega})^{*}\right\|_{1},
\end{align*}
where in the last step we use that $\|P_{j,m}(s)\|_{\infty}\leq 1$ and the fact that $\Theta_{A,\omega}$,$\Theta_{B,\omega}$ are $\|\cdot\|_{1}-\|\cdot\|_{1}$,$\|\cdot\|_{\infty}-\|\cdot\|_{\infty}$ isometries. Write $\Theta_{B,\omega}(f)=(F^{(k)})_{k\to\omega}$ where $F^{(k)}$ are projections in $M_{n(k)}(\C)$. We can estimate the second term above as follows:
\begin{align*}
  \frac{1}{r}\left\|\sum_{j=1}^{r}P_{j,m}((A^{(k)})_{k\to\omega})v\Theta_{B,\omega}(f)P_{j,m}((B^{(k)})_{k\to\omega})^{*}\right\|_{1}&=\lim_{k\to\omega}\frac{1}{r}\left\|\sum_{j=1}^{r}P_{j,m}(A^{(k)})V^{(k)}F^{(k)}P_{j,m}(B^{(k)})^{*}\right\|_{1}\\
  &\leq \limsup_{k\to\infty}\frac{1}{r}\left\|\sum_{j=1}^{r}P_{j,m}(A^{(k)})\otimes \overline{P_{j,m}(B^{(k)})}\#\right\|_{\infty,1}\\
  &\leq \frac{C}{r}\left\|\sum_{j=1}^{r}P_{j,m}(s)\otimes \overline{P_{j,m}(s)}\right\|_{\infty}\\
  &\leq CD.
\end{align*}
So we have shown that for every $m\in \N$ we have
\[\tau(f)-CD\leq\frac{2}{r}\sum_{j=1}^{r}\|P_{j,m}(s)-u_{j}\|_{1}.\]
Since $D<\frac{\tau(f)}{C},$ we obtain a contradiction by letting $m\to\infty.$

Conjecture \ref{conj:random Jung} implies Conjecture \ref{C:1bounded ent conjecture}: This follows from Theorem \ref{thm:as conjugation} (\ref{item:almost sure conjugacy}).

Conjecture \ref{C:1bounded ent conjecture} implies the Peterson-Thom conjecture: This is the content of Proposition \ref{P:CPE implies PT}.

\end{proof}

We remark that it is likely helpful to consider \emph{operator spaces and operator space tensor products} to tackle Conjecture \ref{conj:pq conjecture}. For instance, one can imagine that instead of working with $M_{n}(\C)\otimes M_{n}(\C)$ one considers $S^{p}(n,\tr)\otimes_{\alpha}S^{q}(n,\tr)$ for some $p,q\in [1,\infty]$ and some operator space tensor product $\otimes_{\alpha}.$ One would want to choose $\alpha$ so the map $S^{p}(n,\tr)\otimes_{\alpha}S^{q}(n,\tr)\to CB(M_{n}(\C),S^{1}(n,\tr))$ given by $A\otimes B\mapsto (C\mapsto ACB^{t})$ is completely bounded. It is natural to choose $p,q$ with $\frac{1}{p}+\frac{1}{q}=1$ so that \[\|ACB^{t}\|_{1}\leq \|A\|_{p}\|B\|_{\infty}\|C\|_{q}.\]
Thus it would make sense to consider an operator space tensor norm on $S^{1}(n,\tr)\otimes M_{n}(\C)$ or on $OS^{2}(n,\tr)\otimes OS^{2}(n,\tr)$ where $OS^{2}(n,\tr)$ is Pisier's operator space structure on $S^{2}(n,\tr)$ (or potentially other natural operator space structures on $S^{2}(n,\tr)$).

We close this section by mentioning that the full strength of Conjecture \ref{conj:tensor Haus} is not needed to deduce Conjecture \ref{conj:random Jung}. In fact, we only need that for all $P_{1},\cdots,P_{l}\in \C\ip{T_{1},\cdots,T_{r}}$ we have
\[\left\|\sum_{j=1}^{l}P_{l}(X^{(k)})\otimes \overline{P_{j}(Y^{(k)})}\right\|_{\infty}\to \left\|\sum_{j=1}^{l}P_{j}(s)\otimes \overline{P_{j}(s)}\right\|_{\infty}.\]
And so we can allow a certain symmetry in the elements of $\C\ip{T_{1},\cdots,T_{r},S_{1},\cdots,S_{r}}$ we are testing strong convergence on. Similar remarks apply to the other conjectures in this section. Lastly, in Conjectures \ref{conj:tensor Haus}, \ref{conj:random Jung},\ref{conj:pq conjecture} we may replace the GUE ensemble with Haar unitaries, or any other ensemble provided it has exponential concentration, and converges in law to the law of a generator $x$ of a free group factor with the property that $C^{*}(x)$ is locally reflexive. The details as to why these alternate conjectures imply the Peterson-Thom conjecture are the same as in Proposition \ref{prop:conejctures upon conjecture}.

\section{Closing Remarks}\label{s:close}
We close with some comments around Theorem \ref{T:main}. First is that in Theorem \ref{T:main} (\ref{item:tensor Haus}) it is crucial that we are taking $X^{(k)}$,$Y^{(k)}$ independent of each other. In fact, tensoring tends to behave rather poorly in the strong topology, as we now show.

\begin{defn}\label{D:nonamenability tuple}
Let $(M,\tau)$ be a tracial von Neumann algebra, and $J$ a countable index set. We say that $x\in M^{J}$ is a  \emph{nonamenability tuple} if
\begin{itemize}
    \item $\sup_{j}\|x_{j}\|_{\infty}<\infty,$
    \item there is a $\mu\in\Prob(J)$ so that
$\sum_{j\in J}\mu_{j}|x_{j}\otimes 1-1\otimes x_{j}^{op}|^{2}\in M\overline{\otimes}M^{op}$ is invertible.
\end{itemize}

\end{defn}
The sum in question in the second item converges in $\|\cdot\|_{\infty}$-norm.
By  \cite{Connes} (see also \cite[Theorem 10.2.9]{anantharaman-popa}), every nonamenable von Neumann algebra admits a finite nonamenability tuple.

\begin{prop}
Let $(M,\tau)$ be a tracial von Neumann algebra, $J$ a countable  index set and $x\in M^{J}.$
Suppose either that $x$ is
a nonamenability tuple, or that $W^{*}(x)$ is nonamenable and that $C^{*}(x)$ is locally reflexive. Fix an $R>0$ with $\sup_{j}\|x_{j}\|_{\infty}<\infty.$  Given any sequence  $n(k)\in \N,$ and $x_{k}\in M_{k}(\C)^{J}$ with $\sup_{k,j}\|x_{k,j}\|_{\infty}\leq R$ and  $\ell_{x_{k}}\to \ell_{x}$ strongly we have that $\ell_{x_{k}\otimes 1,1\otimes x_{k}^{t}}$ does not converge strongly to $\ell_{x\otimes 1,1\otimes x^{op}}.$
\end{prop}

\begin{proof}
The case that $C^{*}(x)$ is locally reflexive and that $W^{*}(x)$ is nonamenable follows from Proposition \ref{prop:not conj up}, so we assume that $x$ is a nonamenability set.
Let $\mu\in\Prob(J)$ be so that
\[\sum_{j\in J}\mu_{j}|x_{j}\otimes 1-1\otimes x_{j}^{op}|^{2}\]
is invertible. Since invertible elements in a Banach algebra are open and the sum above converges in $\|\cdot\|_{\infty},$ it follows that we may choose a finite $F\subseteq J$ so that
\[\sum_{j\in F}\mu_{j}|x_{j}\otimes 1-1\otimes x_{j}^{op}|^{2}\]
is invertible.

By strong convergence and
 Lemma \ref{L:strong is just Haus},  the spectrum of $\sum_{j\in F}\mu_{j}|x_{k,j}\otimes 1-1\otimes x_{k,j}^{t}|^{2}$ Hausdorff converges to the spectrum of $\sum_{j\in F}\mu_{j}|x_{j}\otimes 1-1\otimes x_{j}^{op}|^{2}.$ Since $0$ is not in the spectrum of $\sum_{j\in F}\mu_{j}|x_{j}\otimes 1-1\otimes x_{j}^{op}|^{2}$, it follows that $0$ is not in the spectrum of $\sum_{j\in F}\mu_{j}|x_{k,j}\otimes 1-1\otimes x_{k,j}^{t}|^{2}$ for all sufficiently large $k.$ But since
\[\sum_{j\in F}\mu_{j}|x_{k,j}\otimes 1-1\otimes x_{k,j}^{t}|^{2}\#1=0,\]
we have that $0$ is in the spectrum of $\sum_{j\in J}\mu_{j}|x_{k,j}\otimes 1-1\otimes x_{k,j}^{t}|^{2}$ for all $k.$ So we have a contradiction, and this completes the proof.
\end{proof}

More positively, we remark that many previous proofs of strong convergence (e.g. for a mixture of deterministic and random matrices see \cite{Male, MaleCollins}) involve replacing some coordinates of the tuple with their strong limits. A similar approach holds here.

\begin{prop}
Let $r\geq 2$ be an integer and $s=(s_{1},\cdots,s_{r})$ a free semicircular family each with mean zero and variance one. Let $X^{(k)}$ be as in Theorem \ref{T:main} (\ref{item:tensor Haus}). In order to prove Theorem \ref{T:main}, it is enough to show that for any $P\in \C\ip{(T_{j})_{j=1}^{r},(S_{j})_{j=1}^{r}}$ we have
\[\left|\|P(X^{(k)}\otimes 1_{M_{k}(\C)},1_{M_{k}(\C)}\otimes Y^{(k)})\|_{\infty}-\|P(X^{(k)}\otimes 1_{C^{*}(s)},1_{M_{k}(\C)}\otimes s)\|_{\infty}\right|\to 0\]
in probability.
\end{prop}

\begin{proof}
It suffices to show that
\[\|P(X^{(k)}\otimes 1_{C^{*}(s)},1_{M_{k}(\C)}\otimes s)\|_{\infty}\to \|P(s\otimes 1_{C^{*}(s)},1_{C^{*}(s)}\otimes s)\|_{\infty}\]
in probability. Let $\mu^{(k)}\in \Prob(M_{k}(\C)_{s.a}^{r})$ be the distribution of $(X^{(k)}).$
By Haagerup-Thorjb\o rnsen \cite[Theorem A]{HTExt}, we may find a sequence $\Omega_{k}\subseteq M_{k}(\C)_{s.a.}^{r}$ so that
\begin{itemize}
    \item $\mu^{(k)}(\Omega_{k})\to 1,$
    \item for all $(A^{(k)})_{k}\in \prod_{k}\Omega_{k}$ we have that $\ell_{A^{(k)}}\to \ell_{s}$ strongly.
\end{itemize}
Let
\[B=\{(a_{k})_{k}\in \prod_{k}M_{k}(\C):\sup_{k}\|a_{k}\|_{\infty}<\infty\},\,\,\,\,\, J=\{(a_{k})_{k}\in \prod_{k}M_{k}(\C):\|a_{k}\|_{\infty}\to_{k\to\infty}0\},\]
and set $A=B/J$.
Then $A$ is a $C^{*}$-algebra under the norm
\[\|(a_{k})_{k}+J\|=\limsup_{k\to\infty}\|a_{k}\|_{\infty}\]
and we have an exact sequence of $C^{*}$-algebras
\begin{equation}\label{E:exact}
\begin{CD}
0@>>> J @>>> B @>>> A @>>> 0.
\end{CD}
\end{equation}
For $(A^{(k)})_{k}\in \prod_{k}\Omega_{k},$ strong convergence guarantees that we have a $*$-homomorphism
\[\pi\colon C^{*}(s)\to B\]
satisfying $\pi(P(s))=(P(A^{(k)}))_{k}+J$ for all $P\in \C\ip{T_{1},T_{2},\cdots,T_{r}}.$  We thus have a natural $*$-homomorphism
\begin{equation}\label{E:exactness of s}
\pi\otimes \id\colon C^{*}(s)\otimes_{\min{}}C^{*}(s)\to B\otimes_{\min{}}C^{*}(s).
\end{equation}
Since $C^{*}(s)$ is exact, the exact sequence (\ref{E:exact}) produces an exact sequence
\begin{equation}\label{E:exact2}
\begin{CD}
0 @>>> J\otimes_{\min{}} C^{*}(s) @>>> B\otimes_{\min{}}C^{*}(s) @>>> A\otimes_{\min{}}C^{*}(s) @>>> 0.
\end{CD}
\end{equation}
We have  a natural identification
\[J\otimes_{\min{}} C^{*}(s)\cong\left\{(a_{k})_{k}\in \prod_{k}(M_{k}(\C)\otimes_{\min{}}C^{*}(s)):\|a_{k}\|_{\infty}\to 0\right\},\]
and a natural isometric embedding
\[B\otimes_{\min{}}C^{*}(s)\hookrightarrow \left\{(a_{k})_{k}\in \prod_{k}(M_{k}(\C)\otimes_{\min{}}C^{*}(s)):\sup_{k}\|a_{k}\|_{\infty}<\infty\right\}.\]
Combining this with $(\ref{E:exactness of s}),$ $(\ref{E:exact2}),$ we have  produced a $*$-homomorphism
\[C^{*}(s)\otimes C^{*}(s)\to \frac{\{(a_{k})_{k}\in \prod_{k}(M_{k}(\C)\otimes_{\min{}}C^{*}(s)):\sup_{k}\|a_{k}\|_{\infty}<\infty\}}{\{(a_{k})_{k}\in \prod_{k}M_{k}(\C)\otimes_{\min{}}C^{*}(s):\|a_{k}\|_{\infty}\to 0\}}\]
satisfying
\[P(s\otimes 1_{C^{*}(s)},1_{C^{*}(s)}\otimes s)\mapsto (P(A^{(k)}\otimes 1_{C^{*}(s)},1_{C^{*}(s)}\otimes s))_{k}+\left\{(a_{k})_{k}\in \prod_{k}(M_{k}(\C)\otimes_{\min{}}C^{*}(s)):\|a_{k}\|_{\infty}\to 0\right\}\]
for all $P\in \C\ip{T_{1},\cdots,T_{r},S_{1},\cdots,S_{r}}$. Since $*$-homomorphisms between $C^{*}$-algebras are contractive, this implies that
\[\|P(s\otimes 1_{C^{*}(s)},1_{C^{*}(s)}\otimes s)\|_{\infty}\geq \limsup_{k\to\infty}\|P(A^{(k)}\otimes 1_{C^{*}(s)},1_{C^{*}(s)}\otimes s)\|_{\infty}.\]
The inequality
\[\|P(s\otimes 1_{C^{*}(s)},1_{C^{*}(s)}\otimes s)\|_{\infty}\leq \liminf_{k\to\infty}\|P(A^{(k)}\otimes 1_{C^{*}(s)},1_{C^{*}(s)}\otimes s)\|_{\infty}\]
is a consequence of the weak$^{*}$-convergence of the law of $(A^{(k)}\otimes 1_{C^{*}(s)},1_{M_{k}(\C)}\otimes s)$ to the law of  $(s\otimes 1_{C^{*}(s)},1_{C^{*}(s)}\otimes s)$. So we have shown
\[\|P(s\otimes 1_{C^{*}(s)},1_{C^{*}(s)}\otimes s)\|_{\infty}=\lim_{k\to\infty}\|P(A^{(k)}\otimes 1_{C^{*}(s)},1_{C^{*}(s)}\otimes s)\|_{\infty}\]
for all $(A^{(k)})_{k}\in \prod_{k}\Omega_{k}$ and all $P\in \C\ip{T_{1},\cdots,T_{r},S_{1},\cdots,S_{r}}.$ Since $\mu^{(k)}(\Omega_{k})\to 1,$ it follows that
\[\|P(X^{(k)}\otimes 1_{C^{*}(s)},1_{C^{*}(s)}\otimes s)\|_{\infty}\to\|P(s\otimes 1_{C^{*}(s)},1_{C^{*}(s)}\otimes s)\|_{\infty}\]
in probability.

\end{proof}



\end{document}